\documentclass{amsart}
\usepackage{amssymb}
\usepackage[mathscr]{eucal}
\usepackage{eufrak}
\usepackage{xspace}
\usepackage{mathtools}
\usepackage{color}
\newcommand{\R}{\mathbb{R}}

\newcommand{\Z}{\mathbb{Z}}
\newcommand{\N}{\mathbb{N}}

\newcommand{\LL}{\mathcal{L}}

\renewcommand{\H}{\mathcal{H}}

\newcommand{\E}{{\mathcal E}}
\newcommand{\F}{{\mathcal F}}

\newcommand{\po}{\partial}

\newcommand{\wto}{\rightharpoonup}
\newcommand{\ve}{\varepsilon}

\newcommand{\la}{\langle}
\newcommand{\ra}{\rangle}

\newcommand{\loc}{{\text{\rm loc}}}
\newcommand{\X}{\times}

\renewcommand{\d}{\delta}
\renewcommand{\l}{\lambda}

\renewcommand{\a}{\alpha}
\renewcommand{\b}{\beta}
\renewcommand{\t}{\theta}
\newcommand{\s}{\sigma}
\newcommand{\g}{\gamma}
\newcommand{\z}{\zeta}
\renewcommand{\k}{\kappa}
\newcommand{\sgn}{\text{\rm sgn}}

\newcommand{\Om}{\Omega}
\newcommand{\om}{\omega}

\newcommand{\supp}{\text{\rm supp}\,}

\newcommand{\M}{{\mathcal M}}

\renewcommand{\E}{{\mathcal E}}

\renewcommand{\div}{\text{\rm div}\,}

\renewcommand{\supp}{\text{\rm supp}\,}

\newcommand{\x}{\mathbf{x}}

\newcommand{\cO}{{\mathcal O}}
\newcommand{\cU}{{\mathcal U}}
\newcommand{\cV}{{\mathcal V}}
\newcommand{\cE}{{\mathcal E}}
\newcommand{\cL}{{\mathcal L}}
\newcommand{\cD}{{\mathcal D}}
\newcommand{\cI}{{\mathcal I}}

\newcommand{\esslim}{\operatorname{ess}\!\lim}

\newcommand{\Abf}{\mathbf{A}}
\newcommand{\abf}{\mathbf{a}}

\newcommand{\Lip}{\text{\rm Lip}}

\renewcommand{\S}{{\mathcal S}}

\newcommand{\Xal}{{\mathcal X}}

\newcommand{\bbE}{{\mathbb E}}

\newcommand{\bbP}{{\mathbb P}}

\newcommand{\lQ}{\big\la\!\!\big\la}
\newcommand{\rQ}{\big\ra\!\!\big\ra}

\newcommand{\e}{\epsilon}

\renewcommand{\Lip}{\text{Lip\,}}

\newcommand{\ul}{\underline}
\newcommand{\wh}{\widehat}
\newcommand{\wt}{\widetilde}
\newcommand{\ff}{\mathfrak{f}}

\newcommand{\uf}{\mathfrak{u}}

\newcommand{\qq}{\mathfrak{q}}
\newcommand{\gf}{\mathfrak{g}}
\newcommand{\FF}{\mathfrak{F}}
\newcommand{\vf}{\mathfrak{v}}
\newcommand{\rr}{\mathfrak{r}}
\newcommand{\cc}{\mathfrak{c}}

\theoremstyle{plain}
\newtheorem{theorem}{Theorem}[section]

\newtheorem{lemma}{Lemma}[section]
\newtheorem{proposition}{Proposition}[section]
\theoremstyle{definition}
\newtheorem{definition}{Definition}[section]
\theoremstyle{remark}

\numberwithin{equation}{section}
\begin{document}

%Topmatter
\title[Neumann Problem for Stochastic Conservation Laws] 
{The Strong Trace Property and \\ the Neumann Problem for Stochastic Conservation Laws }
\author[H.~Frid]{Hermano Frid}
\thanks{H.~Frid gratefully acknowledges the support from CNPq, through grant proc.\ 305963/ 2014-7, and FAPERJ, through grant proc.\ E-26/103.019/2011.}

\address{Instituto de Matem\'atica Pura e Aplicada - IMPA\\
         Estrada Dona Castorina, 110 \\
         Rio de Janeiro, RJ 22460-320, Brazil}
\email{hermano@impa.br}

\author[Y.~Li]{Yachun Li}
\thanks{Y.~Li gratefully acknowledges the support from NSF of China, through grant 11831011 and 11571232.}

\address{School of Mathematical Sciences, MOE-LSC, and SHL-MAC,  Shanghai Jiao Tong University\\ Shanghai 200240, P.R.~China}

\email{ycli@sjtu.edu.cn}

\author[D.~Marroquin]{Daniel Marroquin}
\thanks{D.~Marroquin thankfully acknowledges the support from CNPq, through grant proc. 150118/2018-0.}

\address{Instituto de Matem\'{a}tica - Universidade Federal do Rio de Janeiro\\
Cidade Universit\'{a}ria, 21945-970, Rio de Janeiro, Brazil}
\email{marroquin@im.ufrj.br}

\author[J.~Nariyoshi]{Jo\~ao F.C.~Nariyoshi}
\thanks{J.F.C.~Nariyoshi appreciatively acknowledges the support from CNPq, through grant proc. 140600/2017-5.}

\address{Instituto de Matem\'atica Pura e Aplicada - IMPA\\
         Estrada Dona Castorina, 110 \\
         Rio de Janeiro, RJ 22460-320, Brazil}
\email{jfcn@impa.br}

\author[Z.~Zeng]{Zirong Zeng}
\thanks{Z.~Zieng gratefully acknowledges the support from NSF of China, through grant 11831011 and 11571232.}

\address{School of Mathematical Sciences, Shanghai Jiao Tong University\\ Shanghai 200240, P.R.~China}

\email{beckzzr@sjtu.edu.cn}

\keywords{stochastic partial differential equation, scalar conservation law, Neumann problem}
\subjclass[2010]{Primary: 60H15, 35R60, 35L65; Secondary: 26B20, 35F31, 35B51}
\date{}
\thanks{}

\begin{abstract}  We establish the well-posedness of the Neumann problem for stochastic conservation laws with multiplicative noise. As a major step  
for establishing the uniqueness of the kinetic solution to the referred problem we establish the new  strong trace property  for stochastic conservation laws.   
Existence of kinetic solutions is proved through the vanishing viscosity method and the detailed analysis of the corresponding stochastic parabolic problem 
is also made here for the first time, as far as the authors know.

\end{abstract}

\maketitle

\section{Introduction}\label{S:1}
We consider the following initial boundary value problem for a stochastic conservation law, on a bounded smooth domain $\cO\subset \R^d$,
\begin{align}
& du+ \nabla\cdot \Abf(u)\,dt = \Phi(u)\, dW, \qquad (t,x)\in (0,\infty)\X\cO, \label{e1.1}\\
& u(0,x)= u_0(x),\qquad x\in \cO, \label{e1.2}\\
& \Abf(u(t,x))\cdot\nu(x) =0,\qquad (t,x)\in(0,\infty)\X\po\cO.  \label{e1.3}
\end{align}
 Here $\Abf\in C^3(\R;\R^d)$ is the flux function and $\nu$ is the normal vector to $\po\cO$. Let  $(\Om,\F, \bbP, (\F_t))$  be a stochastic basis, where $(\Om, \F, \bbP)$ is a probability space and $(\F_t)$ is a complete filtration.
 We use the same framework as in \cite{DV}.
 We assume that $W$ is a cylindrical Wiener process: $W=\sum_{k\ge1}\b_k e_k$, where $\b_k$ are independent Brownian processes and $(e_k)_{k\ge1}$ is a complete orthonormal basis in a Hilbert space ${\frak U}$.
 For each $u\in L^2(\cO)$, $\Phi(u): {\frak U}\to L^2(\cO)$ is defined by $\Phi(u)e_k=g_k(\cdot,u(\cdot))$, where $g_k(\cdot,u(\cdot))$ is a regular function on $\cO$. More specifically, we assume that,  for some $M>0$,  $g_k\in C_c(\bar \cO \X(-M, M))$, with the bounds
 \begin{equation}\label{e1.4*}
 |g_k(x,0)|+|\nabla_xg_k(x,\xi)|+|\po_\xi g_k(x,\xi)|\le \a_k,\qquad \forall x\in \cO,\ \xi\in\R,
 \end{equation}
 where $(\a_k)_{k\ge 1}$ is a sequence of positive numbers satisfying $D:=4\sum_{k\ge1}\a_k^2<\infty$.
 Observe that  \eqref{e1.4*} implies
 \begin{align}
 &G^2(x,u)=\sum_{k\ge1}|g_k(x,u)|^2\le D(1+|u|^2),\label{e1.4}\\
 &\sum_{k\ge1}|g_k(x,u)-g_k(y,v)|^2\le D(|x-y|^2+|u-v|^2),\label{e1.5}
 \end{align}
for all  $x,y\in\cO$, $u,v\in\R$.

The conditions on $\Phi$ imply that $\Phi: L^2(\cO)\to L_2({\frak U}; L^2(\cO))$,  where the latter denotes the space of Hilbert-Schmidt operators from ${\frak U}$ to $L^2(\cO)$. In particular, given a predictable process $u\in L^2(\Om\X[0,T]; L^2(\cO))$, the stochastic integral is a well defined process taking values in $L^2(\cO)$. Indeed, for $u\in L^2(\cO)$, from \eqref{e1.4}, it follows 
$$
\sum_{k\ge1}\|g_k(\cdot,u(\cdot)\|_{L^2(\cO)}^2\le D(1+\|u\|_{L^2(\cO)}^2).
$$

Since, clearly, the series defining $W$ does not converge in ${\frak U}$, in order to have $W$ properly defined as a Hilbert space valued  Wiener process,  one usually introduces an auxiliary space ${\frak U}_0\supset {\frak U}$, such as
$$
{\frak U_0}=\{ v=\sum_{k\ge1} a_k e_k\,:\, \sum_{k\ge 1}\frac{a_k^2}{k^2}<\infty\},
$$
 endowed with the norm
 $$
 \|v\|_{{\frak U}_0}^2=\sum_{k\ge 1}\frac{a_k^2}{k^2},\qquad v=\sum_{k\ge 1} a_k e_k.
 $$
 In this way, one may check that the trajectories of $W$ are $\bbP$-a.s.\ in $C([0,T], {\frak U}_0)$ (see \cite{DPZ}).

For simplicity, we will assume that $u_0$ is independent of $\om\in\Om$, that is, $u_0\in L^\infty(\cO)$.
More precisely, we assume that $u_0\in L^\infty(\cO)$, and  there exists an interval $[a,b]$, with $(-M,M)\subset [a,b]$ , such that
\begin{equation}\label{e1.5'}
a\le u_0(x)\le b, \ a.e.\ x\in\cO.
\end{equation}
Note that, in this setting, we can assume without loss of generality that the $\sigma$-algebra $\mathcal{F}$ is countably generated and $(\mathcal{F}_t )_{t\geq 0}$ is the filtration generated by the Wiener process.

 We also assume that
\begin{equation}\label{e1.6}
\Abf(a)=\Abf(b)=0.
\end{equation}
The extension to the case where $u_0\in L^\infty(\Omega,\mathcal{F}_0,\mathbb{P};L^\infty(\cO))$ is straightforward and comes down to taking expectation wherever an integral involving $u_0$ is present.

We also need to impose a non-degeneracy condition on the symbol ({\em cf.} \cite{GH})
$$
\LL(i\tau, i\k, \xi):= i(\tau+\abf(\xi)\cdot \k),
$$
$\tau\in\R$, $\k\in\R^d$, and $\abf(\xi)=\Abf'(\xi)$. For $\k=(\k_1,\cdots,\k_d)\in\R^d$, let $|\k|^2=\k_1^2+\k_2^2+\cdots+\k_d^2$.  We suppose there exist $\a\in(0,1)$ such that
\begin{equation}\label{e1.7}
\sup_{\tiny{\begin{matrix} \tau\in\R, \k\in\R^d\\ |\k|=1\end{matrix}}} |\Om_{\LL}(\tau,\k,\d)|\lesssim \d^\a,
%\begin{aligned}
%\om_{\cL}(\d) \lesssim \d^\a,\\
%\sup_{\tiny{\begin{matrix}\tau\in\R,n\in\Z^d\\|n|\sim J \end{matrix}}}\sup_{\xi\in [a,b]}|\cL_\xi(i\tau,in;\xi)|&\lesssim J^\b,\qquad \forall \d>0,\, J\gtrsim 1,
%\end{aligned}
\end{equation}
for some $L_0>0$, such that $[a,b]\subset (-L_0,L_0)$, where, for any $\d>0$,
\begin{align*}
\Om_{\LL}(\tau,\k;\d)&:=\{\xi\in (-L_0,L_0) \,:\, |\LL(i\tau,i\k,\xi)|\le \d\}.%\\
%\om_{\LL}(\d)&:=\sup_{\tiny{\begin{matrix} \tau\in\R, n\in\Z^d\\ |n|=1\end{matrix}}} |\Om_{\LL}(\tau,n,\d)|.
\end{align*}
Here we employ the usual notation $x\lesssim y$, if $x\le Cy$, for some absolute constant $C>0$, and $x\sim y$, if $x\lesssim y$ and $y\lesssim x$.

Since we expect solutions of \eqref{e1.1}-\eqref{e1.3} to be bounded from above and from below by $b$ and $a$, respectively (see Theorem \ref{T:4.3}), we only need to impose the nondegeneracy assumptions \eqref{e1.7} for $\xi$ in an open interval containing $[a,b]$.

Examples of flux functions $\Abf(u)$ satisfying \eqref{e1.6} and \eqref{e1.7} are given by ({\em cf.} \cite{TT})
$$
\Abf(u)=\left(\frac1{(l_1+1)}(u-a)^{l_1+1}(u-b)^{l_1+1},\cdots,\frac1{(l_d+1)}(u-a)^{l_d+1}(u-b)^{l_d+1}\right),
$$
where, $l_i\in\N$,  $l_i\ne l_j$, if $i\ne j$, $i,j=1,\cdots,d$,  as it is not difficult to check. Note that, in this case,  for each $\k\in\R^d$ with $|\k|=1$, the  $\sup$ of  $|\Om_{\LL}(\tau,\k,\d)|$,
for $\tau\in\R$,  will be assumed when $-\tau\pm\d$ is a critical value  of $a(\xi)\cdot\k$. Moreover, if $l_{i_0}=\max\{l_1,\cdots,l_d\}$, then it is not difficult to see that the $\sup$ of  $|\Om_{\LL}(\tau,\k,\d)|$, for $\tau\in\R$ and $\k\in\R^d$, $|\k|=1$,  will be assumed for $\k=e_{i_0}$,
the $i_0$-th element of the canonical basis, and it is achieved for $-\tau\pm\d$ running along the  local extremes  of $a_{i_0}(\xi)$, in the interval $[a,b]$, with $+$ or $-$ depending on whether it is a maximum or a minimum, respectively,   and so, condition \eqref{e1.7} is satisfied for $\a=1/l_{i_0}$.

Evidently, \eqref{e1.7} implies the following weaker condition: For $(\tau,\k)\in\R^{d+1}$, $(\tau,\k)\ne0$,
\begin{equation}\label{e1.8}
\LL^1\{\xi\in [-L_0,L_0]\,:\, \tau+\abf(\xi)\cdot\k=0\}=0,
\end{equation}
for some  $L_0>0$, such that  $[a,b]\subset (-L_0,L_0)$, where $\LL^1$ denotes the one-dimensional Lebesgue measure.

\subsection{Definitions and Main Theorem} \label{SS:1.1}

\begin{definition}[Kinetic measure]\label{D:2.1} As in \cite{DV}, we call a map $m$ from $\Om$ to the set of non-negative finite measures over $\cO\X[0,T]\X\R$ a kinetic measure if
\begin{enumerate}
\item $m$ is measurable, in the sense that for each $\phi\in C_b(\cO\X[0,T]\X\R)$, $\la m,\phi\ra:\Om\to\R$ is measurable;
\item $m$ vanishes for large $\xi$, that is, if $B_R^c=\{\xi\in\R\,:\, |\xi|\ge R\}$, then
\begin{equation}
\lim_{R\to\infty} \bbE\, m([0,T]\X\cO\X B_R^c)=0;\label{e2.1}
\end{equation}
\item  for all $\phi\in C_b(\cO\X\R)$, the process
$$
t\mapsto \int_{[0,t]\X\cO\X\R}\phi(x,\xi)\,dm(s,x,\xi)
$$
is predictable.
\end{enumerate}
\end{definition}

\begin{definition} [Kinetic solution] \label{D:2.2} Let $u_0\in L^\infty(\cO)$. A measurable function $u:\Om\X[0,T]\X\cO \to\R$ is said to be a kinetic solution to \eqref{e1.1}--\eqref{e1.3}, if
$(u(t))$ is predictable, $u\in L^\infty(\Om\X[0,T]\X\cO)$ and  there exists a kinetic measure $m$ such that $f:=1_{u>\xi}$ satisfies: for all $\varphi\in C_c^1([0,T)\X\cO\X\R)$,
\begin{multline}\label{e2.2}
\int_0^T\la f(t),\po_t\varphi(t)\ra\,dt +\la f_0,\varphi(0)\ra+\int_0^T\la f(t),\abf(\xi)\cdot\nabla \varphi(t)\ra\,dt\\
=-\sum_{k\ge1} \int_0^T\int_{\cO} g_k(x,u(t,x)) \varphi(t,x,u(t,x))\,dx\,d\b_k(t)\\
-\frac12\int_0^T\int_{\cO}\po_\xi\varphi(t,x,u(t,x)) G^2(x,u(t,x))\,dx\,dt+m(\po_\xi\varphi),
\end{multline}
a.s., where $f_0=1_{u_0(x)>\xi}$, $G^2:=\sum_{k=1}^\infty|g_k|^2$ and $\abf(\xi):=\Abf'(\xi)$. Concerning the Neumann condition \eqref{e1.3}, we ask that for all $\psi\in C_c^\infty ((0,T)\X\R^d)$ we have a.s.\
\begin{equation}\label{e2.3}
 \int_0^T\int_{\cO} u\psi_t+\Abf(u)\cdot\nabla\psi \,dx\,dt +\sum_{k\ge1} \int_0^T\int_{\cO} g_k(x,u(t,x))\psi(t,x)\,dx\,d\b_k(t)=0.
 \end{equation}
 \end{definition}

 We now state the main result of this paper.

 \begin{theorem}\label{T:1.1} Let $u_0\in L^\infty(\cO)$ satisfying  $a\le u_0(x)\le b$ a.e.\ in $\cO$. Assume that conditions \eqref{e1.4*}--\eqref{e1.7}  are  satisfied.
 Then there is a unique kinetic solution to \eqref{e1.1}--\eqref{e1.3}.
 \end{theorem}

 The proof of Theorem~\ref{T:1.1} is  given along the remaining sections.

 Before we pass to a description of earlier works and an overview of the paper, we state the definition of weak entropy solution and the equivalence between this concept and the one of kinetic solution.

  \begin{definition} [Weak entropy solution] \label{D:2.3} Let $u_0\in L^\infty(\cO)$. A bounded measurable function  $u\in L^\infty( \Om\X [0,T]\X\cO)$ is said to be a weak entropy solution to \eqref{e1.1}--\eqref{e1.3} if $(u(t))$ is an adapted $L^2(\cO)$-valued process,
  and for all convex $\eta\in C^2(\R)$, for all non-negative $\varphi\in C_c^1([0,T)\X\cO)$,
  \begin{multline}\label{e2.4}
 \int_0^T\la\eta(u),\po_t\varphi\ra\,dt+\la\eta(u_0),\varphi(0)\ra+\int_0^T\la q(u),\nabla\varphi\ra\,dt\\
 \ge-\sum_{k\ge1}\int_0^T\la g_k(\cdot,u(t))\eta'(u(t)), \varphi\ra\,d\b_k(t)-\frac12\int_0^T\la G^2(\cdot,u(t))\eta''(u(t)),\varphi\ra\,dt,
   \end{multline}
   a.s.\ where $q(u)=\int_0^u\abf(\xi)\eta'(\xi)\,d\xi$ and $\la\cdot,\cdot\ra$ represents the inner product of $L^2(\cO)$.
   Also, $u$ must satisfy \eqref{e2.3}.
   \end{definition}

The following proposition is proven exactly as proposition~15 of \cite{DV} with minor adaptations, and we refer to the latter for its proof.

\begin{proposition}  \label{P:2.1} Let $u_0\in L^\infty(\cO)$ . For a measurable function $u:\Om\X[0,T]\X\cO\to\R$ it is equivalent to be a kinetic solution of \eqref{e1.1}--\eqref{e1.3} and a weak entropy solution of \eqref{e1.1}--\eqref{e1.3}.
\end{proposition}

\medskip

\subsection{Earlier works and overview of the paper}\label{SS:1.2}
In the deterministic case, that is, in the absence of the stochastic term $\Phi(u)\,dW$, the system \eqref{e1.1}--\eqref{e1.3} is a well-known model for many natural phenomena, such as the sedimentation of suspensions in closed vessels, the dispersal of a single species of animals in a finite territory, etc. (see, e.g., \cite{BFK} and the references therein). One may thus introduce such a random perturbation to take into account uncertainties and fluctuations arising in these applications.

The deterministic counterpart of \eqref{e1.1}--\eqref{e1.3} has long been addressed. First, Karlsen, Lie and Risebro \cite{KLR} constructed a weak solution to \eqref{e1.1}--\eqref{e1.3} in one spatial dimension via the front-tracking method, whose uniqueness was established only in the class of solutions obtained by the front-tracking approximations. Later on, B\"urger, Frid and Karlsen \cite{BFK} adopting a natural definition of entropy solution showed the existence and uniqueness of such solutions in arbitrary space dimensions.  A decisive tool for the proof in \cite{BFK} of the uniqueness of the solution was the strong trace property by Vasseur \cite{Va}. See also \cite{AS} and  \cite{FL} for related generalized problems.

On the other hand, stochastic conservation laws have a recent yet intense history. For the sake of examples, we mention Kim \cite{Kim} for the first result of existence and uniqueness of entropy solutions of the Cauchy problem for a one-dimensional stochastic conservation law, in the additive case, that is, $\Phi$ does not depend on $u$. Feng and Nualart \cite{FN}, where a notion of strong entropy solution is introduced, which is more restrictive than that of entropy solution, and  for which the uniqueness is established in the class of entropy solutions in any space dimension, in the multiplicative case, i.e., $\Phi$ depending on $u$; existence of such strong entropy solutions is proven only in the one-dimensional case.  Chen, Ding and Karlsen \cite{CDK}, where the result in \cite{FN} was improved and existence in any dimension was proven in the context of the functions of bounded variation.    Debussche and Vovelle in  \cite{DV}, where a major step  in the development of this theory was made with the extension of the concept of kinetic solution, originally introduced by Lions, Perthame and Tadmor in \cite{LPT},  for deterministic conservation laws,  to the context of stochastic conservation laws, for which the  well-posedness of the Cauchy problem was established in the  periodic setting in any space dimension.  Bauzet, Vallet and Wittbold \cite{BVW}, where the existence and uniqueness of entropy solutions for the general Cauchy problem  was proved in any space dimension (see also,   \cite{KSt}). Concerning boundary value problems,  Vallet and  Wittbold \cite{VW}, in the additive case,  and Bauzet, Vallet and Wittbold \cite{BVW1}, in the multiplicative case,  obtain existence and uniqueness of entropy solutions to  the homogeneous Dirichlet  problem, i.e., null boundary condition. See also \cite{KoN} where the notion of renormalized kinetic solution is introduced to provide the well-posedness of the non-homogeneous Dirichlet problem. Finally, we mention that  the methods and results introduced in \cite{DV}  were later extended to degenerated parabolic problems by  Debussche, Hofmanov\'a and Vovelle \cite{DHV} and Gess and Hofmanov\'a \cite{GH}.

The stochastic Neumann problem \eqref{e1.1}--\eqref{e1.3} is  investigated here for the first time. In the present work, we adopt a definition of  kinetic solution to this initial-boundary value problem, which is a natural extension of the one introduced in \cite{DV} in the periodic case. For the proof of the uniqueness of the kinetic solution  the decisive tool is again the strong trace property, this time for stochastic conservation laws, which is established in this paper for the first time and is a significant extension of
 Vasseur's result in \cite{Va}. 
  On the other hand, the existence of solutions is here addressed through the well known vanishing viscosity method for whose convergence we apply a compactness argument combining Prohorov theorem, Skorokhod representation theorem and the criterion for convergence in probability by Gy\"ongy and Krylov \cite{GK}. The same compactness argument was used before in \cite{Ha} and \cite{DHV} (see also references therein).  A fundamental point in our application of the just mentioned compactness argument is the space regularity in a fractional Sobolev space, uniformly with respect to the viscosity parameter, which is provided by the stochastic averaging lemma established by Gess and Hofmanov\'a in  \cite{GH}. The detailed analysis of the parabolic approximate equation is another important contribution of this paper.

The rest of this paper is organised as follows. in Section~\ref{S:2}, we establish the strong trace property for stochastic conservation laws. In Section~\ref{S:3}, we prove a comparison principle for kinetic solutions to \eqref{e1.1}--\eqref{e1.3}, from which the uniqueness of such solutions follows. In Section~\ref{S:4}  we address the corresponding initial-boundary value problem for the parabolic approximation of \eqref{e1.1}--\eqref{e1.3} obtained with the addition of an artificial viscosity.  In Section~\ref{S:5}, we analyze the passage to the limit when the artificial viscosity goes to zero and show the desired convergence.  We have also included an Appendix concerning the smoothing effects of the semigroup associated to the heat equation with Neumann condition, which play a central role in the analysis developed in Section~\ref{S:4}.

 \section{ Stochastic Strong Trace Property}\label{S:2}

  In this section we establish the  strong trace property for stochastic conservation laws which extends the corresponding property for deterministic conservation laws  first established by  Vasseur in \cite{Va}. We mention in passing that in \cite{Pv} an extension of  the result in \cite{Va} was established,  also concerning deterministic conservation laws,
  relaxing the non-degeneracy condition \eqref{e1.8} on the flux function, which we do not follow here since the even more restrictive condition \eqref{e1.7}  will  be needed for the existence of kinetic solutions to the problem \eqref{e1.1}--\eqref{e1.3}.

\begin{definition}\label{D:3.1}  Let $\cU\subset\R^{N}$ be  an open set. We say that  $\po \cU$ is a {\em   Lipschitz deformable boundary}  if the following hold:

\begin{enumerate}

\item[(i)] For each $x\in\po \cU$, there exist $r>0$ and a Lipschitz mapping $\g :\R^{N-1}\to\R$  such that, upon relabelling, reorienting and translation,
$$
\cU\cap Q(x,r) = \{\,y\in\R^{N-1}\,:\, \g(y_1,\cdots,y_{N-1})<y_N\,\}\cap Q(x,r),
$$
where $Q(x,r)=\{\,y\in\R^{N}\,:\, |y_i-x_i|\le r,\ i=1,\cdots,N\,\}$. We denote by $\hat \g$ the map $ \hat y\mapsto (\hat y,\g(\hat y))$, $\hat y=(y_1,\cdots,y_{N-1})$.

\item[(ii)]  There exists a map $\Psi:[0,1]\X\po \cU\to \bar \cU$ such that $\Psi$ is a bi-Lipschitz homeomorphism over its image and $\Psi(0,x)=x$, for all $x\in\po \cU$.
For $s\in[0,1]$, we denote by $\Psi_s$ the mapping from $\po \cU$ to $\bar \cU$ given by $\Psi_s(x)= \Psi(s,x)$, and set $\po \cU_s:=\Psi_s(\po \cU)$. We call such map a Lipschitz deformation for $\po \cU$.

\end{enumerate}
 \end{definition}

 \begin{definition} \label{D:3.2}  Let $\cU\subset\R^{N}$ be an open set with a  Lipschitz deformable boundary  and $\Psi:[0,1]\X\po \cU \to \bar \cU$ a Lipschitz deformation for $\po \cU$.
  The Lipschitz deformation is said to be {\em regular over} $\Gamma\subset\po \cU$,  if $D\Psi_s\to \operatorname{Id}$, as $s\to0$, in $L^1(\Gamma, \H^{N-1})$. It is  simply said to be {\em regular} if
  it is regular over $\po\cU$. The Lipschitz deformation  is said to be {\em strongly regular over} $\Gamma\subset\po \cU$,  if it is regular over $\Gamma$ and the Jacobian determinants  $J[\Psi_s]$, $0\le s\le 1$, defined through a convenient parametrization for $\Gamma$, belong to $\Lip(\Gamma)$ and $J[\Psi_s]\to 1$ in $\Lip(\Gamma)$ as $s\to0$.
  \end{definition}

The following result is a straightforward consequence of the normal trace formula proved in  \cite{Fr} (see also \cite{FL}) for divergence-measure fields in $L^p$, $1\le p\le \infty$,  extending the one in  \cite{CF1} for fields in $L^\infty$.  
Those are fields in $L^p(U;\R^N)$, $1\le p\le \infty$, whose divergence is   Radon measure over $U$, for some open set $U\subset \R^N$. In the statement below the field is only partially in $L^\infty$, but the component which is not in $L^\infty$ is
orthogonal to the normal to the boundary surfaces. 

\begin{theorem}  \label{T:3.0}
	Let $\cU \subset \R^d$ be an open set with strongly regular deformable boundary and $F = (F^0,F^1) \in L^p((0,T)\X\cU)\X L^\infty((0,T)\X\cU; \R^d)$, $1\le p<\infty$,   be a vector field such that the distribution 
	$\operatorname{div}_{t,x}F = \partial_t F^{0} + \operatorname{div}_{x} F^{1}$ is a Radon measure in $(0,T)\X\cO$. Then there exists an element $F^{1,b}\cdot\nu \in L^\infty((0,T)\X\partial\cU)$ such that, for every $\partial\cU$-strongly regular Lipschitz deformation $\psi$,
	\begin{equation}
	\operatornamewithlimits{ess\,lim}_{s \rightarrow 0} F^{1}(\cdot, \psi(\cdot, s))\cdot\nu_s(\cdot) = F^{1,b}\cdot\nu \text{ weakly-$\star$ in } L^\infty((0,T)\X\partial\cU), \label{e03.00}
	\end{equation}
	where $\nu_s$ denotes the unit outward normal vector field of $\psi(\{s\}\X\partial\cU)$.
\end{theorem}  

Next we state and prove  our theorem concerning the strong trace property for stochastic conservation laws. 

\begin{theorem} \label{T:3.1} 
	Assume that $T>0$, $\cO \subset \R^d$ is a bounded open set with regular deformable Lipschitz boundary, and that $\Abf$ is a $C^3$ flux function satisfying \eqref{e1.8}. 
	Let also $u \in L^\infty(\Omega\X(0,T)\X\cO) \cap L^2(\Omega\X[0,T];L^2(\cO))$, $a\le u\le b$,  be an entropy solution to \eqref{e1.1}, that is, $u$ is predictable and for all convex $\eta\in C^2(\R)$,
	and $q(u)=\int_0^u \eta'(\z)\Abf'(\z)\,d\z$, we have
	\begin{equation}\label{e03.1}
	d\eta(u)+ \nabla_x\cdot q(u)\,dt\le \sum_{k\ge1} g_k(x,u(t,x))\eta'(u) \,d\b_k+ \frac12 G^2(x,u)\eta''(u)\,dt,
	\end{equation}
	in the sense of the distributions in $(0,T)\X\cO$, a.s.\ in $\Om$. 
	Then, there exists a function $u^\tau \in L^\infty(\Omega\X(0,T)\X\partial\cO)$ such that, for every $\po\cO$-strongly regular Lipschitz deformation $\psi : [0,1] \X \partial\cO \rightarrow \overline \cO$,
	\begin{equation}\label{e03.2}
	\operatornamewithlimits{ess\,lim}_{s \rightarrow 0}  \bbE\int_0^T\int_{\partial \cO} |u(t, \psi(s,\widehat{x})) - u^\tau(t,\widehat{x})|\, d\H^{d-1}(\widehat{x})dt = 0, 
	\end{equation}
 where $\H^{d-1}$ denotes the $(d-1)$--dimensional Hausdorff measure. In particular, for any $C^2$ function $G : \R \rightarrow \R$, we have
	$$
	[G(u)]^\tau = G(u^\tau).
	$$

Moreover, we also have that 
\begin{equation}\label{e03.2'}
	\operatornamewithlimits{ess\,lim}_{s \rightarrow 0}  \int_0^T\int_{\partial \cO} |u(t, \psi(s,\widehat{x})) - u^\tau(t,\widehat{x})|\, d\H^{d-1}(\widehat{x})dt = 0, 
	\end{equation}
 almost surely.
 
\end{theorem}

\begin{proof}  
We begin by observing that \eqref{e03.1} in the limit for a suitable sequence of smooth approximations of  $\eta(u;\xi)=(u-\xi)_+-(\xi)_+$ gives rise to a measure $m\in \M((0,T)\X\cO\X\R)$, defined for a.e.\ $\om\in \Om$ by
\begin{multline}\label{e03.3}
m:=-\po_t\eta(u;\xi)- \nabla_x\cdot q(u;\xi) + \sum_{k\ge1} g_k(x,u(t,x))\eta'(u;\xi) \,\dot \b_k\\ + \frac12 G^2(x,u)\eta''(u;\xi),
\end{multline}
where $\eta'(u;\xi)=\frac{d}{du}\eta(u;\xi)= 1_{(\xi,\infty)}(u)$  and $\eta''(u;\xi)=\frac{d^2}{du^2}\eta(u;\xi)=\d_{u=\xi}$. 
In particular, $m$ is a kinetic measure in the sense of Definition~\ref{D:2.1}, and for $L>\|u\|_{L^\infty(\Om\X(0,T)\X\cO)}$ we have that $\supp m\subset [0,T]\X\overline \cO\X(-L,L)$. This can be easily 
seen by applying $m$ to  $\phi(t,x)\rho(\xi)$  with  $\phi\in C_c^\infty((0,T)\X\cO)$ and $\rho\in C_c(\R)$ so that $\rho(\xi)=0$, for $|\xi|\le L$, observing that $\eta(u;\xi)=-\xi$, for $\xi> L$, $\eta(u(t,x);\xi)=u-\xi$, for $\xi<-L$, 
$q(u;\xi)=0$, for $\xi>L$, $q(u;\xi)=\Abf(u)-\Abf(\xi)$, for $\xi<-L$,   $\eta'(u;\xi)=0$, for $\xi>L$, $\eta'(u;\xi)=1$, for $\xi<-L$,  $G^2(x,\xi)=0$, if $|\xi|\ge L$,  and $u(t,x)$  satisfies \eqref{e1.1} in the sense of 
the distributions on $(0,T)\X\cO$ a.s.\ in $\Om$.  

We also
observe that $\po_\xi\eta(u;\xi)=1_{(-\infty,u)}(\xi)-1_{(-\infty,0)}(\xi)=\chi(u;\xi)$, which is the well known kinetic function. Similarly, we observe that $\po_\xi q(u;\xi)= \abf(\xi)\chi(u;\xi)$. We also see that 
$\po_\xi \eta'(u;\xi)=\po_\xi 1_{(u,+\infty)}(\xi)=\d_{u=\xi}$. Therefore, deriving \eqref{e03.3} with respect to $\xi$, in the sense of the distribution on $(0,T)\X\cO\X\R$, we obtain the kinetic equation
\begin{equation}\label{e03.4}
\frac{\po\ff}{\po t} +\abf(\xi)\cdot \nabla_x \ff=\frac{\po\qq }{\po\xi} -\sum_{k=1}^\infty \frac{\po \ff}{\po \xi} g_k\dot\b_k+ \sum_{k=1}^\infty\d_{\xi=0}g_k\dot\b_k,
\end{equation}
for  $\ff(t,x,\xi):=\chi(u(t,x);\xi)$,  where $\qq=m-\frac12 G^2\d_{u=\xi}$, which could also be obtained by using the fact that $S(u)-S(0)=\int_{\R}S'(\xi)\chi(u;\xi)\,d\xi$, for any $S\in C^1(\R)$.

\bigskip
\centerline{\em \#1. Preparing for the proof.}
\bigskip
Let us henceforth fix $L>\|u\|_{L^\infty(\Om\X(0,T)\X\cO\X\R)}$. Concerning the measure $m$,  defined by \eqref{e03.3}, we first observe the following fact, whose  proof is somehow standard by apply a suitable test function to equation \eqref{e03.3},
and so we omit it. 

\begin{lemma}\label{L:3.01} For any $1\le p <\infty$, it holds
\begin{equation}\label{e03.5}
	\bbE\Vert m \Vert_{\M}^p = \bbE m\big( (0,T)\X\cO\X[-L,L] \big)^p \leq C(p).	
	\end{equation}
\end{lemma}

Taking in \eqref{e03.1} $\eta(u)=\pm u$ we conclude that $u$ satisfies
\begin{equation} \label{e03.6'}
du+\nabla\cdot \Abf(u)\, dt=  \sum_{k\ge1} g_k(x,u(t,x))\, d\b_k ,
\end{equation}
in the sense of the distributions on $(0,T)\X\cO$, a.s.\ in $\Om$,  which also follows from \eqref{e03.4}, using as test function $\varphi(t,x,\xi)=\phi(t,x)$, with $\phi\in C_c^\infty((0,T)\X\cO)$, and the fact that $\int_{\R}\ff(t,x,\xi)\,d\xi=u(t,x)$,
and we use the fact that $\qq$ is supported in $[0,T]\X\overline\cO\X(-L,L)$.  

\bigskip
\centerline{\em \#2. The existence of weak traces.}
\bigskip
After localizing as done in \cite{Va}, we may assume that $\cO$ is of the form
\begin{equation}\label{e.orient}
\cO_0 = \big\{ x =(\widehat{x}, x_d) \in (-r,r)^{d-1}\X(-r,r)\,: \, x_d > \gamma_0(\widehat{x}) \big\},
\end{equation}
where $r>0$ and $\gamma_0 : (-r,r)^{d-1} \rightarrow \R$ is a $C^2$-function satisfying $-r<\gamma_0(\widehat{x})<r$ everywhere. Hence, the boundary subset we are interested in is
$$
\Gamma_0 = \big\{ x = (\widehat{x}, x_d) \in (-r,r)^{d}\X(-r,r)\,: \, x_d = \gamma_0(\widehat{x}) \big\}.
$$
Notice that $\Gamma_0$ is parametrized in $\widehat{x}$, since it is the graph of $\gamma_0$. As a consequence, for any $\Gamma_0$-strongly regular Lipschitz deformation $\psi$, we can write
$$
\ff_\psi(t,\widehat{x},s,\xi) = \ff(t,\psi(s,\widehat{x}),\xi) \text{ for every $\widehat{x} \in (-r,r)^{d-1}$}.
$$

In order to make easier the writing, let us set
$$
\begin{cases}
Q      := (0,T) \X (-r,r)^d,     \text{ and} \\
\Sigma := (0,T) \X (-r,r)^{d-1}.
\end{cases}
$$

As in \cite{Va}, the following result establishing the existence of a  weak trace  for the function $\ff(\om,t,x,\xi)$  is a decisive step towards the strong trace property.  Although to write a generic neighborhood
$\cO_0$ in the form \eqref{e.orient} requires, in general, the use of an  affine transformation, that is, translation, relabeling of coordinates and possibly reorientation of one axis, the effect of such transformation in
equation \eqref{e03.4} would be just that, instead of $\abf(\xi)$, we would have $\mathcal{R}\abf(\xi)$, where $\mathcal{R}$ is the linear part of the referred affine transformation. So, for simplicity, we may assume, without
loss of generality, that $\ff$ keeps satisfying the same equation \eqref{e03.4} in the new coordinates which are still denoted in the same way.   

\begin{lemma}[Existence of the weak traces] \label{L:3.02}
	There exists a unique function $\ff^\tau \in L^\infty(\Omega\X\Sigma\X(-L,L))$ such that, for any $\Gamma_0$-strongly regular Lipschitz deformation, for a.e.\ $\om\in\Om$, we have
	\begin{equation}
	\esslim_{s \rightarrow 0} \ff_\psi(\om, \cdot, s, \cdot) = \ff^\tau(\om,\cdot,\cdot) \text{ in the weak-$\star$ topology in } L^\infty(\Sigma\X(-L,L)). \label{e03.7'}
	\end{equation}
	
Moreover,
\begin{equation}\label{e03.7''}
\operatornamewithlimits{ess\,lim}_{s \to 0} \ff_\psi(\cdot, s, \cdot) = \ff^\tau(\cdot,\cdot) \text{ in the weak-$\star$ topology of $L^\infty(\Omega\X\Sigma\X(-L,L))$}.
\end{equation}
\end{lemma}

\begin{proof}  {\em Step 1:} Let $(h_n)_{n\in\N} \subset C_c^1((-L,L))$ be dense in $L^1((-L,L))$. We consider representatives in $L^\infty(\Om\X(0,T)\X\cO)$ for all functions of the forms 
$$
\int_{-L}^L h_n(\xi) \ff(t,x,\xi)\,d\xi,\quad \int_{-L}^L h_n(\xi)  \ff(t,x,\xi)\abf(\xi)\,d\xi,\quad   h_n'(u(t,x))G^2(x,u(t,x)),
$$
and in $L^2(\Om\X(0,T)\X\cO)$ for functions of the form
$$
          \sum_{k=1}^\infty  \int_0^t h_n(u(t,x))g_k(x,u(s,x)) \,d\b_k(s) .
$$
 Let $\Om_0\subset\Om$ be a subset of total measure such that for all $\om\in\Om_0$, the corresponding paths of these functions, viewed as Banach space-valued stochastic processes,  are well defined functions 
 in $L^\infty((0,T)\X\cO_0)$ and $L^2((0,T)\X\cO_0)$, respectively.  We also assume that for all $\om\in\Om_0$ there exists $C(\om)>0$ such that 
 \begin{align*}
m \big((0,T)\X\cO \X [-L,L]\big) \leq C(\om), 
\end{align*}
which is possible by \eqref{e03.5}.  So, let us fix for the moment  $\om\in\Om_0$. 

Let us consider the vector fields $F_{n}$ given by
\begin{align}
F_{n}(t,x) = \Bigl( \int_{-L}^L h_n(\xi) \ff(t,x,\xi) \, d\xi - \sum_{k=1}^\infty \int_0^t h_n&(u(t,x)) g_k(x, u(s,x)) \, d\beta_k(s)\Bigr. , \nonumber \\
&\Bigl. \int_{-L}^L h_n(\xi) \ff(t,x,\xi) \abf(\xi) \, d\xi\Bigr) . \label{e03.6}
\end{align}
We see that  $F_{n} \in L^2((0,T)\X\cO_0)\X  L^\infty( (0,T)\X\cO_0; \R^{d})$. Moreover, it is not hard to check that
$$
\operatorname{div}_{t,x} F_{n} = -  \int_{-L}^L h_n'(\xi) \qq(t,x,\xi)\,d \xi  \in \M((0,T)\X\cO_0).
$$
As a consequence, since $\Gamma_0$ is a strongly regular deformable Lipschitz boundary, Theorem~\ref{T:3.0} tell us that there exists a set $\S_{n} \subset [0,1]$ of total measure and some $F_{n}^{1,b}\cdot \nu \in L^\infty((0,T)\X(-r,r)^{d-1})$, which does not depend on $\psi$, such that
\begin{align}
F_{n}^{1}(\cdot, \psi(\cdot, s))\cdot\nu_s(\cdot) \stackrel{\star}{\rightharpoonup} F_{n}^{1,b}\cdot\nu &\text{ $\star$-weakly in $L^\infty(\Sigma)$} \nonumber \\ &\text{ as $s \rightarrow 0$ along $s \in \S_{n}$}. \label{e03.7}
\end{align}
Write $\S = \cap_{n=1}^\infty \S_{n}$ so that $\S$ also has total measure in $[0,1]$. Let us now check that $F_{n}$ depends linearly on  $h_n$. For any integer $M \geq 1$ and $\varphi_p \in L^1(\Sigma)$, $1 \leq p \leq M$, 
the relations \eqref{e03.00}, \eqref{e03.6} and \eqref{e03.7}, the latter taking $s \in \S$,  say that
\begin{align*}
\Big|\int_\Sigma&  \sum_{n,p=1}^M  (F_{n}^{1,b} . \nu)(t,\widehat{x}) \varphi_p(t,\widehat x) \,dt d\widehat{x}    \Big| \\ &\leq \Vert \abf \Vert_{L^\infty(-L,L)} \int_\Sigma \int_{-L}^L \Bigg|\sum_{n,p=1}^M    h_n(\xi) \varphi_p(t,x) \Bigg| \, d\xi d\widehat{x} dt \\
						 &= \text{(const.)} \Bigg\Vert \sum_{n,p=1}^M  h_n \otimes \varphi_p \Bigg\Vert_{L^1(\Sigma\X(-L,L))}.
\end{align*}
Thus, as $(L^1)^* = L^\infty$, there exists some $H \cdot\nu \in L^\infty(\Sigma\X(-L,L))$ such that, for  all $h \in L^1(-L,L)$ and all $\varphi \in C_c^\infty(\Sigma),$
\begin{align}
 \int_\Sigma \int_{-L}^L h(\xi) &\varphi(t,\widehat{x}) \ff(t,\widehat{x},\psi(\widehat{x},s)) \abf(\xi)\cdot \nu_s(\widehat{x}) \, d\xi d\widehat{x}\, dt \nonumber \\
&\rightarrow \int_\Sigma \int_{-L}^L h(\xi) \varphi(t,x) (H\cdot\nu)(t,\widehat{x}) \, d\xi d\widehat{x}\, dt,  \label{e03.8}
\end{align}
as $s \rightarrow 0$ along $s \in \S$. Notice that $H\cdot\nu$ does not depend on $\psi$.

{\em Step 2:} To conclude, let us observe that, since $\Vert \ff_\psi(\cdot,s,\cdot) \Vert_{L^\infty} \leq 1$, the Banach--Alaoglu theorem asserts that, for every strongly regular Lipschitz deformation $\psi$ and every sequence $s_n$ in $\S$ converging to $0$, there exists a subsequence $s_{n_k}$ and some $\ff_\psi^\tau \in L^\infty(\Sigma\X(-L,L))$  such that
$$
\ff_\psi(\cdot,s_{n_k},\cdot) \stackrel{\star}{\rightharpoonup} \ff_\psi^\tau \text{ $\star$-weakly in $L^\infty(\Sigma\X(-L,L))$ as $k \rightarrow \infty$}.
$$
Thus, by \eqref{e03.8}, we deduce from the fact that $\nu_s \rightarrow \nu$ in $L^1((-r,r)^{d-1};\R^d)$ that
\begin{align*}
 \int_\Sigma \int_{-L}^L h(\xi) &\varphi(t,\widehat{x}) \ff_\psi^\tau (t,\widehat{x}) \abf(\xi)\cdot \nu(\widehat{x}) \, d\xi\, d\widehat{x}\, dt    \nonumber \\
&=  \int_\Sigma \int_{-L}^L h(\xi) \varphi(t,x) (H\cdot\nu)(t,\widehat{x}) \, d\xi\, d\widehat{x}\, dt, 
\end{align*}
for every  $h \in L^1(-L,L)$ and $\varphi \in C_c^\infty(\Sigma)$. Since the right-hand term is independent of $\psi$ and $s_n$, so is $\abf(\xi)\cdot\nu(\widehat{x}) \ff_\psi^\tau(t,\widehat{x},\xi)$. On the other hand, because the nondegeneracy condition implies that
$$
\LL^1\big\{ \xi \in (-L,L)\,:\, \abf(\xi) \cdot \nu(\widehat{x}) = 0 \big\} = 0,
$$
we conclude that $\ff_\psi^\tau$ also does not depend on $\psi$ and $s_n$, hence we may denote it by $\ff^\tau$. 

%{\em Step 3:}  Thus far we have kept $\om\in\Om_0$ fixed. 
%In particular, $\ff^\tau$ just defined depends on $\om\in\Om_0$, therefore $\ff^\tau$ is an almost everywhere defined function in $\Om\X\Sigma\X(-L,L)$ and, for each $\om\in\Om_0$, we have that for any $\Gamma_0$-Lipschitz regular deformation $\psi$
%\begin{multline}\label{e03.9} 
%\ff^\tau(\om,\cdot,\cdot)=\esslim_{s\to0}\ff_\psi(\om,\cdot,s,\cdot), \\
%\text{in the weak-$\star$ topology of $L^\infty(\Sigma\X(-L,L))$}.  
%\end{multline}
%Moreover, we claim that $\ff^\tau$ is measurable and $\ff^\tau\in L^\infty(\Om\X\Sigma\X(-L,L))$.
%Indeed, to see this it suffices to observe that, for all $\om\in\Om_0$ and any sequence $(s_n)_{n\in\N}\subset (0,1)$, with $s_n\to0$, as $n\to\infty$, we have
%\begin{multline*}
%\ff^\tau(\om,\cdot,\cdot)=\lim_{n\to\infty}\frac{1}{s_n} \int_0^{s_n}\ff_\psi(\om,\cdot,s,\cdot)\,ds, \\
%\text{in the weak-$\star$ topology of $L^\infty(\Sigma\X(-L,L))$}, 
%\end{multline*}
%and, since $\Om_0$ is a subset of total measure of $\Om$, we easily deduce that
%\begin{multline}\label{e03.9'}
%\ff^\tau(\cdot,\cdot,\cdot)=\lim_{n\to\infty}\frac{1}{s_n} \int_0^{s_n}\ff_\psi(\cdot,\cdot,s,\cdot)\,ds, \\
%\text{in the weak-$\star$ topology of $L^\infty(\Om\X%\Sigma\X(-L,L))$}.
%\end{multline}
%This relation shows that $f^\tau(\cdot,\cdot,\cdot)\in %L^\infty(\Om\X\Sigma\X(-L,L))$.

{\em Step 3:} Arguing as before, but considering now the vector fields
	\begin{align*}
	F_{m,n}(t,x) = \bbE \Big[ X_m \, \Big( \int_{-L}^L h_n(\xi) f(t,x,\xi) \, d\xi - \sum_{k=1}^\infty \int_0^t h_n&(u(t,x)) g_k(x, u(s,x)) \, d\beta_k(s) , \nonumber \\& \int_{-L}^L h_n(\xi) f(t,x,\xi) a(\xi) \, d\xi \Big)  \Big],
	\end{align*}
	where $(X_m)_{m \in \N}$ is a sequence in $L^\infty(\Omega)$ that is dense in $L^1(\Omega)$ (notice that we can always suppose that $\Omega$ is countably generated), we can deduce the existence of some $f^b \in L^\infty(\Omega \X \Sigma \X (-L,L))$ such that 
	$$\operatornamewithlimits{ess\,lim}_{s \to 0} \ff_\psi(\cdot, s, \cdot) = \ff^b \text{ in the $\star$-weak topology of $L^\infty(\Omega\X\Sigma\X(-L,L))$.}$$
	
{\em Step 4:} It remains to show that $\ff^b(\omega, \cdot, \cdot) = \ff^\tau(\om, \cdot, \cdot, \cdot)$ for almost all $\om \in \Omega$ in the $L^1$-sense. This, however, can be seen from the fact that both are the weak-$\star$ limit of $\frac{1}{s} \int_0^s \ff_\psi(\cdot, \s, \cdot)\, d\s$ in $L^\infty(\Om \X \Sigma \X (-L,L))$ as $s \to 0$. Observe that this also shows that $\ff^\tau$ is measurable and $\ff^\tau\in L^\infty(\Om\X\Sigma\X(-L,L))$.

%It remains to show \eqref{e03.7''}. For this, note that for any sequence $(s_n)_{n\to\mathbb{N}}$ with $s_n\to 0$, as $n\to \infty$, the sequence $(f_\psi(\cdot,\cdot,s_n,\cdot))_{n\in\mathbb{N}}$ has a subsequence converging in the weak-$\star$ topology of $L^\infty(\Om\times\Sigma\times(-L,L))$. On the other hand, the limit of any subsequence that converges weakly-$\star$ in $L^\infty(\Om\times\Sigma\times(-L,L))$ has to be $f^\tau$ due to \eqref{e03.9'}. Thus, we conclude that
%\begin{multline*}
%\ff^\tau(\cdot,\cdot,\cdot)=\lim_{s_n\to 0}\ff_\psi(\cdot,\cdot,s_n,\cdot)\,ds, \\
%\text{in the weak-$\star$ topology of $L^\infty(\Om\X\Sigma\X(-L,L))$},
%\end{multline*}
%which implies \eqref{e03.7''}.
\end{proof}

%We remark, concerning the previous proof, that, although the set $\S\subset(0,1)$ depends on $\om\in\Om_0$, and so we should denote it $\S(\om)$, we may, for each $\om\in\Om_0$, modify the definition of $\ff$ over the set 
%$$
%\CN_\om:=\bigcup_{s\in(0,1)\setminus\S(\om)} \{\om\}\X(0,T)\X\psi(s,\Gamma_0)\X(-L,L),
%$$
%setting $\ff=\ff^\tau$ over $\CN_\om$. Now, by Fubini theorem, we have that the set $\CN:=\bigcup_{\om\in\Om_0} \CN_\om$ is a subset of measure zero of $\Om\X\Sigma\X(-L,L)$, so the resulting function is still in the same class in $L^\infty(\Om\X\Sigma\X(-L,L))$. 
%Therefore, we also have
% \begin{equation*}
%	\esslim_{s \rightarrow 0}\ff_\psi(\cdot, \cdot, s, \cdot) = \ff^\tau \text{ weakly-$\star$ in } L^\infty(\Omega\X\Sigma\X(-L,L)). 
%\end{equation*}	  

Next, we need to convert the weak-$\star$ convergence in $L^\infty(\Sigma\X(-L,L))$ in the statement of the previous lemma to a strong convergence in $L^1(\Sigma\X(-L,L))$. 
For that we recall the following criterion from \cite{Va}, to which we refer for the proof. We first recall that, for a measure space $S$,  a function $z\in L^\infty(S\X\R)$ is a $\chi$-function if for almost all
$x\in S$, there exists $a(x)\in \R$ such that 
$$z(x,\xi)= 1_{(-\infty, a(x))}(\xi)-1_{(-\infty,0)}(\xi).
$$ 
Note that, in this case, $a(x)=\int_{-\infty}^\infty z(x,\xi)d\xi$. Furthermore, $\| a\|_{L^\infty(S)}\leq L$ if and only if the corresponding $\chi$-function satisfies $z(x,\xi)=0$, for a.e. $x\in S$ and $|\xi|\geq L$. In this case, we may simply consider $z$ as an element of $L^\infty(S\times (-L,L))$.

\begin{lemma}\label{2.criterion0}
Let $S$ be a finite measure space and let $g_n\in L^\infty(S\times (-L,L))$ be a sequence of $\chi$-functions converging weakly to $g\in L^\infty(S\times (-L,L))$. Define $v_n(\cdot)=\int_{-L}^Lg_n(\cdot,\xi)d\xi$ and $v=\int_{-L}^Lg(\cdot,\xi)d\xi$. Then, the three following propositions are equivalent:
\begin{itemize}
\item[(i)] $g_n$ converges strongly to $g$ in $L^1(S\times \R)$,
\item[(ii)] $v_n$ converges strongly to $v$ in $L^1(S)$,
\item[(iii)] $g$ is a $\chi$-function.
\end{itemize}
\end{lemma}

As a direct consequence of Lemma~\ref{2.criterion0} we have the following.

\begin{lemma} \label{2.criterion}
	For  every regular Lipschitz deformation $\psi$,
	$$
	\esslim_{s \rightarrow 0} \ff_\psi(\cdot,\cdot, s,\cdot) = \ff^\tau(\cdot,\cdot,\cdot) \text{ strongly in $L^1(\Om\X\Sigma\X(-L,L))$},
	$$
	if, and only if, $\ff^\tau(\cdot,\cdot,\cdot)$ is a $\chi$-function a.e.\ in $\Om\X\Sigma\X(-L,L)$.  
\end{lemma}

We now pass to the verification that $\ff^\tau$ is indeed a $\chi$-function. 

\bigskip

\centerline{\em \# 3. The blow-up procedure.}

\bigskip

Let us keep $\cO_0$ fixed. Since $\ff^\tau$ is independent on the $\Gamma_0$-strongly regular Lipschitz deformation, we may choose the special deformation $\psi(s,\widehat{x}) = (\widehat{x}, \gamma(\widehat{x}) + s)$, which is trivially strongly regular. 
Identifying $y_d = s$ 
and $\widehat{y} = \widehat{x}$, define
$$
\widetilde{\ff}(t,y,\xi) = \ff_{\psi}(t,\widehat{y}, y_d, \xi) = \ff(t, \psi(y_d,\widehat{y}), \xi).
$$
Notice that there exists an $r_0>0$ such that $\psi(y_d,\wh y) \in \cO_0$ provided that $(\widehat{y}, y_d) \in (-r,r)^{d-1} \X (0, r_0)$. As a result, we see from \eqref{e03.4} that $\widetilde{\ff}$ is a solution to
\begin{multline}\label{3.eqftil}
\frac{\partial \widetilde \ff}{\partial t} + \widehat{\abf}(\xi) \cdot \nabla_{\widehat{y}} \wt \ff + \widetilde{a_d}(\widehat{y},\xi) \frac{\partial \widetilde{\ff}}{\partial y_d} = \frac{\partial \widetilde{\qq}}{\partial \xi} - \sum_{k=1}^\infty \frac{\partial\widetilde \ff}{\partial\xi} \widetilde{g_k} \dot{\beta_k} + \sum_{k=1}^\infty \delta_{\xi = 0} \widetilde{g_k}  \dot{\beta_k}  \\ \quad\text{a.s.\ in the sense of the distributions in $(0,T)\X(-r,r)^{d-1}\X(0,r_0).$}
\end{multline}
In the equation above, we have denoted $\abf(\xi) = (\widehat \abf(\xi), a_d(\xi))$ and
\begin{equation}
\widetilde{a}_d(\widehat{y}, \xi) = a_d(\xi) - \nabla \gamma_0(\widehat{y})\cdot\widehat{\abf}(\xi) = \lambda(\widehat{y})\abf(\xi) \cdot \nu(\widehat{y}), \label{3.eqa}
\end{equation}
for some $\lambda(\widehat{y}) < 0$, and $\nu$ is the outward unit normal, due to \eqref{e.orient}.

Moreover, we have also written $\widetilde{\qq} = \widetilde{m} - \frac{1}{2} \widetilde{G}^2 \delta_{\xi = \widetilde u(t,y)}$, where
\begin{equation*}
\begin{aligned}
&\widetilde u (t,y) = u(t,\psi(\widehat{y},y_d)) = \int_{-L}^L \widetilde{\ff}(t,y,\xi)\, d\xi, \\
&\widetilde{m}(t,y,\xi) = m(t, \psi(y), \xi), \\
&\widetilde{g}_k(y, z)  = g_k(\psi(y), z) \text{ for all $k \geq 1$}, \text{ and}\\
&\widetilde{G}^2(y,z)   = \sum_{k=1}^\infty \widetilde{g}_k(y,z)^2.
\end{aligned}
\end{equation*}

Before we rescale $\widetilde{\ff}$, let us recall some lemmas in \cite{Va} slightly adapted to our setting, to which we refer for the proofs. Let $\Om_0$ be as in the proof of Lemma~\ref{L:3.02}.
\begin{lemma} \label{3.lemma1}
        There exists  a sequence    $0<\ve_n\to 0$   and a set of total measure $\E \subset \Sigma$  such that, 
	for every $(t_0,\widehat{y}_0)\in \E$, 
	 and every $R>0$,
	\begin{align}
	&\lim\limits_{n \rightarrow \infty} \bbE \frac{1}{\ve_n^d} \widetilde{m} \Big( \Big\{ (t_0,\widehat{y}_0) + (-R\ve_n, R\ve_n)^{d} \Big\} \X (0, R\ve_n) \X [-L,L] \Big) = 0, \text{ and} \label{3.limm}\\
	&\lim\limits_{n \rightarrow \infty} \bbE \frac{1}{\ve_n^d} \int \int_{ \{ (t_0,\widehat{y}_0) + (-R\ve_n, R\ve_n)^{d} \} \X (0, R\ve_n)} \frac{1}{2} \widetilde{G^2}(y,\widetilde{u}(t,y))\, dy dt = 0. \label{3.limG}
	\end{align}
	Consequently, for every  $(t_0,\widehat{y}_0) \in \E$  and every $R>0$,
	$$
	\lim\limits_{n \rightarrow \infty} \bbE \frac{1}{\ve_n^d} |\widetilde{\qq} | \Big( \Big\{ (t_0,\widehat{y}_0) + (-R\ve_n, R\ve_n)^{d} \Big\} \X (0, R\ve_n) \X [-L,L] \Big) = 0,
	$$
	where, as usual, $|\tilde\qq|(A)$ denotes the total variation of $\tilde\qq$ on the set $A$.  
	
	Therefore,  given $(t_0,\wh y_0)\in\E$, there exists a subsequence of $\ve_n$, still denoted $\ve_n=\ve_n(t_0,\wh y_0)$, and
	a subset of total measure $\Om_1=\Om_1(t_0,\wh y_0)\subset\Om_0$, such that, for all $\om\in\Om_1$,
	  \begin{align}
	&\lim\limits_{n \rightarrow \infty}  \frac{1}{\ve_n^d} \widetilde{m} \Big( \Big\{ (t_0,\widehat{y}_0) + (-R\ve_n, R\ve_n)^{d} \Big\} \X (0, R\ve_n) \X [-L,L] \Big) = 0, \label{3.limm'}\\
	&\lim\limits_{n \rightarrow \infty}  \frac{1}{\ve_n^d} \int \int_{ \{ (t_0,\widehat{y}_0) + (-R\ve_n, R\ve_n)^{d} \} \X (0, R\ve_n)} \frac{1}{2} \widetilde{G^2}(y,\widetilde{u}(t,y))\, dy dt = 0, \label{3.limG'}\\
	&\lim\limits_{n \rightarrow \infty}  \frac{1}{\ve_n^d} |\widetilde{\qq} | \Big( \Big\{ (t_0,\widehat{y}_0) + (-R\ve_n, R\ve_n)^{d} \Big\} \X (0, R\ve_n) \X [-L,L] \Big) = 0.\label{3.qq}
	\end{align}
\end{lemma}

\begin{lemma} \label{3.lemma2} 
	There exists a subsequence of $\ve_n$, still denoted by $\ve_n$, and  a subset  of  $\E\subset \Sigma$,  also of total measure and still denoted by $\E$,  such that, for every $(t_0, \widehat{y}_0) \in \E$,
	 every $R > 0$, and every $1 \leq p < \infty$,
	\begin{align}
	& \int_{-R}^R \int_{(-R,R)^{d-1}} \int_{-L}^L |\wt a_d(t_0,\widehat{y}_0, \xi) - \wt a_d(t_0+\ve_n \ul t, \widehat{y}_0 + \ve_n\widehat{\ul y}, \xi) |^p \, d\xi \, d\widehat{\ul y} d\ul t \to 0, \label{3.lima}  \\
	& \bbE \int_{-R}^R \int_{(-R,R)^{d-1}} \int_{-L}^L |\ff^\tau(t_0 + \ve_n \ul t\, , \widehat{y}_0 + \ve_n \widehat{\ul y} \,,  \xi) - \ff^\tau(t_0,\widehat{y}_0, \xi) | \, d\xi \,d\widehat{\ul y}\, d\ul t \to 0,\label{3.limfb}
	\end{align}
	as $n \to \infty$. 
	
	Again, it follows that, given $(t_0,\wh y_0)\in\E$ there exists a subsequence of $\ve_n(t_0,\wh y_0)$ also denoted $\ve_n=\ve_n(t_0,\wh y_0)$, and a subset of $\Om_1(t_0,\wh y_0)$, also of total measure, and
	also denoted $\Om_1=\Om_1(t_0,y_0)$, such that, for all $\om\in\Om_1$,
	\begin{equation}
	   \int_{-R}^R \int_{(-R,R)^{d-1}} \int_{-L}^L |\ff^\tau(t_0 + \ve_n \ul t\, , \widehat{y}_0 + \ve_n \widehat{\ul y} \,,  \xi) - \ff^\tau(t_0,\widehat{y}_0, \xi) | \, d\xi \,d\widehat{\ul y}\, d\ul t \to 0. \label{3.limfb'}
	\end{equation}
\end{lemma}

\bigskip

 Let   $(t_0, \widehat{y_0}) \in \E$, which will be kept fixed until the end of the proof.  Our goal now is to show that $\ff^\tau(t_0,\wh y_0, \xi)$ is a $\chi$-function. 
 
 Let  $R = R(t_0,\widehat{y}_0)$ be the least number between $r$, $r_0$, $T - t_0$ and $t_0$. Let us denote $y^0:=(\widehat{y}_0,0)$. For any $\ve > 0$, consider
\begin{equation}
\widetilde{\ff}_\ve(\underline t, \underline y, \xi) = \widetilde{\ff}(t_0+ \ve \underline t, y^0+ \ve \underline y, \xi), \label{trace.chive}
\end{equation}
for $\omega \in \Omega$, $-L < \xi < L$, and  
$$
(\underline t, \underline y) = (\underline t, \underline{\widehat{y}}, \underline{y_d})  \in (-R/ \ve, R/ \ve ) \X (-R/\ve, R/\ve)^{d-1} \X (0,R/\ve)\stackrel{\text{def}}{=} Q_\ve .
$$
Cearly, $\widetilde{\ff}_\ve$ depends on $(t_0, \widehat{y}_0)$. However, since this point will be fixed henceforth,  we will omit this dependence.

Each $\widetilde{\ff}_\ve$ is still a $\chi$-function, and, in the sense of weak traces, 
\begin{equation}
\widetilde{\ff}_\ve(\ul{t}, \wh {\ul y}, 0, \xi) = \ff^\tau(t_0+\ve\ul t,  \widehat{y}_0 + \ve \wh{\ul y}, \xi), \label{3.chive2} 
\end{equation} 
for $-L<\xi<L$ and 
$$
(\ul t, \ul{\wh y}) \in (-R/\ve, R/\ve)\X(-R/\ve , R/\ve)^{d-1} \stackrel{\text{def}}{=} \Sigma_\ve.
$$

If $X(t)$ is a predictable  stochastic process, we denote $\int_a^b X(t)\, d\b_k(t):=\int_0^b X(t)\, d\b_k(t)-\int_0^a X(t)\,d\b_k(t)$, for all $k\in\N$. Observe that $\int_a^b X(t)\,d\b_k(t)=-\int_b^a X(t)\,d\b_k(t)$.

{}From  \eqref{3.eqftil} we get that $\wt \ff_\ve$ satisfies  the equation
\begin{multline}\label{3.eqfve0}
\frac{\partial \widetilde \ff_\ve}{\partial \ul t} + \widehat{\abf}(\xi) \cdot \nabla_{\widehat{\ul y}} \widetilde \ff_\ve + \widetilde{a_d}(\widehat{y}_0 + \ve \ul {\wh y}, \xi) \frac{\partial \widetilde \ff_\ve}{\partial \ul y_d} = 
\frac{\partial \widetilde{\qq}_\ve}{\partial \xi} \\
- \sum_{k=1}^\infty  \frac{\po}{\po \ul t} \left(\int_{t_0}^{t_0+\ve\ul t} \frac{\po \wt \ff}{\po \xi}(t,y^0+\ve \ul y, \xi)\widetilde{g}_{k,\ve}(\ul y,\xi)\, d\b_k(t)\right) 
+ \sum_{k=1}^\infty \d_{\xi=0}  \frac{\po}{\po \ul t}   \left(\int_{t_0}^{t_0+\ve\ul t}  \widetilde{g}_{k, \ve} \,d\b_k(t)\right), 
\end{multline}
where we use the notations 
$$
\begin{aligned}
&\wt g_{k,\ve}(\ul y,\xi):= \wt g_k(y^0+\ve \ul y, \xi),  \\
%& \frac{\partial\widetilde \ff_\ve}{\partial\xi} \widetilde{g}_{k,\ve} \dot{\widetilde{\beta}}_{k,\ve} :=\frac{\po}{\po \ul t} \left(\int_{t_0}^{t_0+\ve\ul t} \frac{\po \wt \ff}{\po \xi}(t,y^0+\ve \ul y, \xi)g_{k,\ve}(\ul y,\xi)\, d\b_k(t)\right), \\
%&\delta_{\xi = 0} \widetilde{g}_{k, \ve}  \dot{\widetilde{\beta}}_{k,\ve}:=\frac{\po}{\po \ul t} \left(\int_{t_0}^{t_0+\ve\ul t} \d_{\xi=0}   \widetilde{g}_{k, \ve} \,d\b_k(t)\right), \\
&\widetilde{\qq}_\ve :=\widetilde{m}_\ve - \widetilde{G}_\ve^2(\ul y,\xi)\d_{\wt u_\ve(\ul t,\ul y) },\\
&\wt u_\ve(\ul t, \ul y) := \int_{-L}^L \wt \ff_\ve(\ul t, \ul y, \xi) \, d\xi = \wt u(t _0+ \ve \ul t, y^0 + \ve \ul y), \\
&\widetilde{G}^2_{\ve}(\ul y, \xi) := \sum_{k=1}^\infty \wt g_{k,\ve}^2( \ul y, \xi)=  \wt G^2(y^0 + \ve \ul y, \xi).\\
\end{aligned}
$$
Regarding  $\widetilde{m}_\ve$,  it is, almost surely, the measure such that, for every $a_0<b_0$, $\ldots$, $a_d < b_d$ and $L_1 < L_2$,
\begin{multline*}
\widetilde{m}_\ve \Big( \operatornamewithlimits{\Pi}_{j=0}^d [a_j, b_j] \X [L_1, L_2]  \Big) \\ = \frac{1}{\ve^d} \widetilde{m} \Big( [t_0 + \ve a_0, t_0 + \ve b_0] \X \Big[ \widehat{y}_0 + \operatornamewithlimits{\Pi}_{j=1}^{d-1} [\ve a_j, \ve b_j] \Big]\X [\ve a_d, \ve b_d] \X [L_1, L_2]  \Big).
\end{multline*}
Therefore, also, almost surely,  for every $a_0<b_0$, $\ldots$, $a_d < b_d$ and $L_1 < L_2$, we have
\begin{multline*}
\widetilde{\qq}_\ve \Big( \operatornamewithlimits{\Pi}_{j=0}^d [a_j, b_j] \X [L_1, L_2]  \Big) \\ = \frac{1}{\ve^d} \widetilde{\qq} \Big( [t_0 + [\ve a_0  , \ve b_0]] \X \Big[ \widehat{y}_0 + \operatornamewithlimits{\Pi}_{j=1}^{d-1} [\ve a_j, \ve b_j] \Big]\X [\ve a_d, \ve b_d] \X [L_1, L_2]  \Big).
\end{multline*}

Equation \eqref{3.eqfve0} can also be written in the following sometimes more convenient  form 
\begin{multline}\label{3.eqfve}
\frac{\partial \widetilde \ff_\ve}{\partial \ul t} + \widehat{\abf}(\xi) \cdot \nabla_{\widehat{\ul y}} \widetilde \ff_\ve + \widetilde{a_d}(\widehat{y}_0, \xi) \frac{\partial \widetilde \ff_\ve}{\partial \ul y_d} 
= \frac{\partial}{\partial \ul y_d} \Big( \big(\widetilde{a_d}(\widehat{y}_0, \xi) - \widetilde{a_d}(\widehat{y}_0 
+ \ve \widehat{\ul y}, \xi) \big) \widetilde \ff_\ve \Big) \\
 +  \frac{\partial \widetilde{\qq}_\ve}{\partial \xi} 
- \sum_{k=1}^\infty \frac{\po^2}{\po \ul t\po \xi} \left(\int_{t_0}^{t_0+\ve\ul t} 
 \widetilde{g}_{k,\ve}(\ul y,\xi) \wt \ff(t,y^0+\ve \ul y, \xi ) \, d{\beta}_{k}(t)\right)
\\ + \sum_{k=1}^\infty \frac{\po}{\po \ul t} \left(\int_{t_0}^{t_0+\ve\ul t} 
\frac{\partial \widetilde{g}_{k,\ve}}{\po \xi}(\ul y,\xi) \wt \ff(t,y^0+\ve \ul y, \xi \big) \,d\beta_{k}(t)\right)
\\
+ \sum_{k=1}^\infty  \d_{\xi=0} \frac{\po}{\po \ul t}  \left(\int_{t_0}^{t_0+\ve\ul t}  \widetilde{g}_{k, \ve} \,d\b_k(t)\right).
\end{multline}

In what follows, motivated by \cite{Va}, we are going to prove that $\ff^\tau$ is a $\chi$-function  by proving that, along a suitable subsequence,    
$\wt \ff_\ve\to \ff^\tau(t_0,\wh{y}_0, \cdot )$ in $L_\loc^1(\R^d\X(0,\infty)\X\R)$, for all $\om$ in a subset of total measure of $\Om$.  Here the subsequence and the subset of $\Om$ will depend, in general, on $(t_0,\wh y_0)$, 
as opposed to the deterministic case  in \cite{Va}, where the subsequence does not depend on $(t_0,\wh y_0)$. Nevertheless, this dependence does not have any effect in the conclusion. 
More specifically, keeping the notation in Lemma~\ref{3.lemma1} and Lemma~\ref{3.lemma2}, we will  first obtain a  set of total measure $\Om_2(t_0,\wh y_0)\subset \Om_1(t_0,\wh y_0)$ and a subsequence of $\ve_n(t_0,\wh y_0)$, also denoted $\ve_n(t_0,\wh y_0)$,  so that for each $\om\in\Om_2$, 
$\wt \ff_{\ve_n}\wto \ff^\tau(t_0,\wh y_0, \cdot )$
in the sense of the distributions on $\R^d\X(0,\infty)\X\R$. Then, we will obtain another subset of total measure $\Om_3(t_0,\wh y_0)\subset  \Om_2(t_0,\wh y_0)$ and a subsequence $\ve_{n_k}(t_0,\wh y_0)$ of 
$\ve_n(t_0,\wh y_0)$ 
such that, for any $\om\in\Om_3(t_0,\wh y_0)$,      $\wt \ff_{\ve_{n_k}}\to \ff^\tau(t_0,\wh y_0, \cdot)$ in $L_\loc^1(\R^d\X(0,\infty)\X\R)$.

\begin{lemma}\label{L:3.03} There exists a subset of total measure $\Om_2(t_0,\wh y_0)\subset \Om_1(t_0, \wh y_0)$ and a sequence $\ve_n=\ve_n(t_0,\wh y_0)\to0$, such that for all $\om\in\Om_2(t_0,\wh y_0)$, $\wt \ff_{\ve_n}\wto \ff^\tau(t_0,\wh y_0,\cdot )$ in the sense of the distributions on $\R^d\X(0,\infty)\X\R$. 
\end{lemma} 

\begin{proof} Let  $\frak D\subset  C_c^\infty(\R^d\X(0,\infty)\X(-L,L))$ be countable and dense in $C_c^2(\R^d\X(0,\infty)\X(-L,L))$.  Let $\varphi\in \frak D$  and let 
$\Lambda_{\ve}$ denote the distribution corresponding to the stochastic Wiener processes 
in \eqref{3.eqfve0}.  That is,
\begin{multline*}
\Lambda_{\ve}:=- \sum_{k=1}^\infty  \frac{\po}{\po \ul t} \left(\int_{t_0}^{t_0+\ve\ul t} \frac{\po \wt \ff}{\po \xi}(t,y^0+\ve \ul y, \xi)\wt g_{k,\ve}(\ul y,\xi)\, d\b_k(t)\right) \\
+ \sum_{k=1}^\infty  \d_{\xi=0} \frac{\po}{\po \ul t}   \left(\int_{t_0}^{t_0+\ve\ul t}  \widetilde{g}_{k, \ve} \,d\b_k(t)\right).
\end{multline*}
It is not hard to verify that, for $\ve$ sufficiently small,
$$
\la \Lambda_{\ve} , \varphi\ra= -\sum_{k=1}^\infty \int_\R \int_{\R_+^{d}}\int_{t_0}^{t_0+\ve \ul t} \wt{g}_{k,\ve}(\ul y,\wt u(t, \wh{y}_0+\ve\ul y))\varphi_{\ul t} (\ul t,\ul y, \wt u(t,\wh{y}_0+\ve \ul y))d\b_k(t) \, d\ul y\,d \ul t,
$$
where we denote $\R_+^d=\R^{d-1}\X(0,\infty)$.  We have
\begin{multline*}
\bbE |\la \Lambda_{t_0,\wh y_0,\ve}, \varphi\ra| \\
\le \bbE \int_\R \int_{\R_+^{d}}\left| \sum_{k=1}^\infty  \int_{t_0}^{t_0+\ve \ul t} \wt{g}_{k,\ve}(\ul y,\wt u(t, \wh{y}_0+\ve\ul y))\varphi_{\ul t} (\ul t,\ul y, \wt u(t,\wh{y}_0+\ve \ul y))d\b_k(t)\right|  \, d\ul y\,d \ul t\\
\le \bbE \int_\R \int_{\R_+^{d}}\left( \sum_{k=1}^\infty \left|\int_{t_0}^{t_0+\ve \ul t} \big|\wt{g}_{k,\ve}(\ul y,\wt u(t, \wh{y}_0+\ve\ul y))\varphi_{\ul t} (\ul t,\ul y, \wt u(t,\wh{y}_0+\ve \ul y)) \big|^2\,dt \right|\right)^{1/2}   \, d\ul y\,d \ul t\\
\le \bbE \int_\R \int_{\R_+^{d}}\sup_\xi |\varphi_{\ul t}(\ul t,\ul y,\xi)| \left( \sum_{k=1}^\infty \left| \int_{t_0}^{t_0+\ve \ul t} \big|\wt{g}_{k,\ve}(\ul y,\wt u(t, \wh{y}_0+\ve\ul y)) \big|^2\,dt\right| \right)^{1/2}   \, d\ul y\,d \ul t\\
\le  \bbE \int_\R \int_{\R_+^{d}}\sup_\xi |\varphi_{\ul t}(\ul t,\ul y,\xi)| \left( D \left| \int_{t_0}^{t_0+\ve \ul t} (1+ | \wt u(t, \wh{y}_0+\ve\ul y)|^2)\,dt\right| \right)^{1/2}   \, d\ul y\,d \ul t\\
\le D^{1/2} \text{diam}\,\{\supp \varphi_{\ul t} \}^{d+1} \|\varphi_{\ul t}\|_\infty (1+\|u\|_\infty^2)^{1/2} \ve^{1/2}\to 0\quad\text{as $\ve\to0$},
\end{multline*} 
where we have used Burkholder inequality (see, e.g., \cite{O}) in the second inequality above, and \eqref{e1.4} in the fourth inequality  above. 
Therefore, using a diagonal process, we can obtain a subsequence of $\ve_n(t_0,\wh y_0)$, obtained from Lemma~\ref{3.lemma1} and Lemma~\ref{3.lemma2},
 which we still denote $\ve_n(t_0,\wh y_0)$ and  a set of total measure $\Om_2(t_0,\wh y_0)\subset \Om_1(t_0,\wh y_0)$ such that, for all $\om\in\Om_2$,  
$$
|\la \Lambda_{\ve_n}, \varphi\ra| \to 0,\quad \text{for all $\varphi\in C_c^\infty(\R^d\X(0,\infty)\X(-L,L))$}.
$$ 

Now, let us fix $\om\in\Om_2(t_0,\wh y_0)$. Let $\varphi\in C_c^\infty(\R^d\X(0,\infty)\X(-L,L))$ be of the form $\varphi=\rho_h(\ul y_d)\tilde\varphi$, with $\tilde \varphi\in C_c^\infty(\R^d\X[0,\infty)\X(-L,L))$  and
$\rho_h(\ul y_d)=\int_0^{\ul y_d}\z_h(s)\,ds$, where $\z_h(s)=h^{-1}\z(h^{-1}s)$ and $\z\in C_c^\infty((0,1))$, with $\z\ge0$ and $\int_0^1 \z(s)\,ds=1$.  Taking $\varphi$
of this form as a test function in \eqref{3.eqfve0} and using Lemma~\ref{L:3.02} we get, after letting $h\to0$, that
\begin{multline}\label{3.eqfve2}
\int_{\R}\int_0^\infty\int_{\R^{d-1}}\int_\R  \widetilde \ff_\ve \frac{\po \wt \varphi}{\partial \ul t} +  \widetilde{\ff}_\ve\, \wh {\abf}(\xi) \cdot  \nabla_{\widehat{\ul y}}\wt \varphi
+\wt\ff_\ve \,  \widetilde{a_d}(\widehat{y}_0 + \ve \ul {\wh y}, \xi) \frac{\po\wt \varphi}{\po\ul y_d}\,d\ul t\, d\wh {\ul y}\,d\ul y_d\,d\xi \\ +\int_{\R}\int_{\R^{d-1}}\int_\R   \widetilde{a_d}(\widehat{y} + \ve \ul {\wh y}, \xi) \ff^\tau((t_0, \widehat{y}_0 + \ve \ul {\wh y}, \xi) 
\wt \varphi (\ul t, \wh{\ul y},0,\xi)\, d\ul t\,d\wh{\ul y}\, d\xi\\
=\la \wt \qq_\ve,\frac{\po\wt\varphi}{\po \xi}\ra-\la\Lambda_\ve, \wt\varphi\ra.
\end{multline}
Taking $\ve=\ve_n(t_0,\wh y_0)$, passing to a subsequence of $\ve_n(t_0,\wh y_0)$ if necessary so that $\wt \ff_{\ve_n}\wto \wt\ff$ in the weak-$\star$ topology of $L^\infty(\R^d\X(0,\infty)\X(-L,L)))$, 
for some $\wt \ff \in L^\infty(\R^d\X(0,\infty)\X(-L,L))$, and making $\ve_n(\om)\to0$ we get for $\wt\ff$,
\begin{multline}\label{3.eqfve3} 
\int_{(-L,L)} \int_{\R_+^d}\int_\R  \widetilde \ff \,\frac{\po \wt \varphi}{\partial \ul t} +  \widetilde{\ff}\, \wh {\abf}(\xi) \cdot  \nabla_{\widehat{\ul y}}\wt \varphi
+\wt\ff \,  \widetilde{a_d}(\widehat{y}_0, \xi) \frac{\po\wt \varphi}{\po\ul y_d }\,d\ul t \,d\ul y \,d\xi \\ +\int_{\R} \int_{\R^{d-1}} \int_\R  \widetilde{a_d}(\widehat{y}_0, \xi)\, \ff^\tau(t,\wh y, \xi) \wt \varphi (\ul t, \wh{ \ul y},0,\xi)\,d\ul t \,d\wh{\ul y}\, d\xi=0. 
\end{multline}
Now, since  $\wt \ff$ and $\ff^\tau$ vanish for $\xi\notin (-L,L)$ and $a_d(\wh y_0,\xi)\neq 0$, for a.e.\ $\xi\in(-L,L)$, by choosing $\tilde \varphi(\ul t, \ul y, \xi)=\rho(\xi)\tilde \phi(\ul t,\ul y)$, with $\rho\in C_c^\infty(\R)$, $ \tilde\phi\in C_c^\infty(\R\X\R^{d-1}\X[0,\infty))$,
we get that for almost every $\xi$, $\wt \ff(\cdot,\cdot,\xi)$ satisfies
 \begin{multline}\label{3.eqfve4} 
 \int_{\R_+^d}\int_\R  \widetilde \ff \,\frac{\po \wt \varphi}{\partial \ul t} +  \widetilde{\ff}\, \wh {\abf}(\xi) \cdot  \nabla_{\widehat{\ul y}}\wt \varphi
+\wt\ff \,  \widetilde{a_d}(\widehat{y}_0, \xi) \frac{\po\wt \varphi}{\po\ul y_d }\,d\ul t \,d\ul y  \\ + \int_{\R^{d-1}} \int_\R  \widetilde{a_d}(\widehat{y}_0, \xi)\, \ff^\tau(t_0,\wh y_0, \xi) \wt \varphi (\ul t, \wh{ \ul y},0,\xi)\,d\ul t \,d\wh{\ul y} =0. 
\end{multline}  
Now, if $\xi\in S_{\text{nd}}^+:=\{\z\,:\, a_d(\wh y_0,\z)>0\}$, making the change of coordinates 
$$
\ul y_d= a_d(\wh y_0,\xi)\, \ul z_d,\quad  \ul t  =\ul \tau+ \ul y_d,\quad \wh{\ul y}= \wh{\ul z} + \ul z_d\, \wh a(\xi),
$$   
we get that, in this new system of coordinates, $\wt \ff$ satisfies
$$
\int_{\R_+^d} \int_\R \wt \ff \frac{\po \wt \phi}{\po \ul z_d}\, d\ul \tau\, d\wh{\ul z}\, d\ul z_d + \int_{\R^{d-1}} \int_\R   \ff^\tau(t_0,\wh y_0, \xi) \phi (\ul t, \wh{ \ul z},0)\,d\ul \tau \,d\wh{\ul z} =0,
$$
for all $\tilde \phi\in C_c^\infty(\R^d\X[0,\infty))$. Hence, choosing $\wt \phi$ of the form $\wt \phi(\ul t, \wh{\ul z}, \ul z_d)=\theta(\ul t,\wh{\ul z})\z_h(\ul z_d)$, for a suitable sequence $\z_h\in C_c^\infty ([0,\infty))$, 
with $h\to0+$,  we then easily conclude that 
$$
\wt \ff(\ul \tau, \wh {\ul z}, \ul z_d,\xi)= \ff^\tau(t_0, \wh y_0,\xi), \qquad\text{for a.e.\ $(\ul \tau,\wh{\ul z}, \ul z_d)\in \R\X\R^{d-1}\X(0,\infty)$}.
$$  
Coming back to the original coordinates $(\ul t,\wh{\ul y}, \ul y_d)$, we arrive at the asserted conclusion. Finally, if $\xi\in S_{\text{nd}}^-:=\{\z\,:\, a_d(\wh y_0,\z)<0\}$, the same procedure, with minor modifications, yields the conclusion of the lemma. The only difference is that in this case the equation corresponding to the new set of variables the time $\tau$ evolves in the opposite orientation, which does not pose any inconveniences.
\end{proof}

\bigskip
\centerline{\# 4. Convergence a.s.\ in $L_\loc^1(\R\X\R_+^d\X\R)$.}
\bigskip

Let us denote $\wt {\abf}_{\wh y_0}(\xi):= (\wh a(\xi), \wt a_d(\wh y_0,\xi))$. Regarding the modified symbol 
$$
\cL_{\wh y_0}(i\tau, i\k, \xi) = i(\tau + \wt {\abf}_{\wh y_0}(\xi)\cdot \k),
$$ 
we have the following proposition.

\begin{proposition} \label{3.propL}
		   Condition \eqref{e1.8} implies that $\cL_{\wh y}(i\tau, i\k, \xi)$ satisfies the following nondegeneracy condition: For $\d>0$, 
		\begin{equation}
		\lim_{\d\to0} \sup_{\tau^2+|\k|^2 = 1} \LL^1 \big\{ \xi \in [-L,L] \,:\, |\cL_{\wh y_0}(i\tau,i\kappa,\xi)| < \d \big\} =0. \label{3.nondeg}
		\end{equation}
\end{proposition}
\begin{proof} Condition \eqref{e1.8} clearly implies that, for all $(\tau,\k)$, with $\tau^2+|\k|^2=1$,   	
	\begin{equation}
		 \LL^1 \big\{ \xi \in [-L,L]\,:\, |\cL_{\wh y_0}(i\tau,i\kappa,\xi)| =0 \big\} =0. \label{3.nondeg1}
		\end{equation}
Now, consider the functions  $\varpi_\d:\mathbb{S}^{d+1}\to[0,\infty)$, defined by
$$
\varpi_\d(\tau,\k):= \LL^1 \big\{ \xi \in [-L,L]\,:\,  |\cL_{\wh y_0}(i\tau,i\kappa,\xi)| < \d \big\}.
$$
The functions $\varpi_\d$ are continuous on the compact $\mathbb{S}^{d+1}$  and 	$\varpi_{\d_1}(\tau,\k)\le \varpi_{\d_2}(\tau,\k)$, if $\d_1\le \d_2$.
Then, Dini theorem implies that $\varpi_\d\to0$ uniformly.
	
\end{proof}

We thus arrive at the decisive step of the proof of Theorem~\ref{T:3.1}.  Let $\ve_n(t_0,\wh y_0)$  and $\Om_2(t_0,\wh y_0)\subset \Om_1(t_0,\wh y_0)$ be the sequence and the subset of total measure obtained in Lemma~\ref{L:3.03}. 

\begin{lemma}\label{L:3.04} There is a subset of total measure $\Om_3(t_0,\wh y_0)\subset \Om_2(t_0,\wh y_0)$ and a subsequence of $\ve_n(t_0,\wh y_0)$ still denoted $\ve_n$, such that for all $\om\in\Om_3$, 
$\wt \ff_{\ve_n}\to \ff^\tau(t_0,\wh y_0, \xi)$ in $L_\loc^1(\R^d\X(0,\infty)\X(-L,L))$. In particular,  $ \ff^\tau(t_0,\wh y_0, \xi)$ is a $\chi$-function a.e.\ in $\Om\X(-L,L)$.  
\end{lemma}

\begin{proof} {\em Step 1:} For $(t_0,\wh y_0)\in\E$ and $\om\in\Om_1$,  let  $\wt \ff(\ul t,\ul y,\xi)=\ff^\tau(t_0,\wh y_0,\xi)$. We observe that $\wt \ff_\ve$ and $\wt \ff$ are supported in $|\xi|\le L$. 
	Let $V \subset\subset \R\X\R^{d-1}\X(0,\infty)$. 
	We consider $\varphi, \theta \in C_c^\infty(\R\X\R^{d-1}\X(0,\infty))$ such that
	\begin{enumerate}
		\item $0 \leq \varphi \leq \theta \leq 1$ everywhere,
		\item $\varphi \equiv 1$ in $V$, and
		\item $\theta \equiv 1$ in the support of $\varphi$.
	\end{enumerate}
	In this case, from \eqref{3.eqfve}, we see that $(\varphi \wt \ff_{\ve}) = \wt \ff_{\ve}^\varphi$ satisfies for $\ve$ sufficiently small
	\begin{equation} \label{5.eqfvephi0}
	\begin{aligned}
	& \frac{\partial \wt \ff_{\ve}^\varphi}{\partial \ul t} + \widetilde{a}_{\widehat{y}_0}(\xi) \cdot \nabla \wt \ff_{\ve}^\varphi  = \wt \ff_{\ve}\Bigg(\frac{\partial \varphi}{\partial t} - \wt a_{\wh y_0}(\xi) \cdot \nabla \varphi \Bigg)  \\
	&\qquad+\frac{\partial}{\partial \ul y_d} \Big( \big(\widetilde{a_d}(\widehat{y}_0, \xi) - \widetilde{a_d}(\widehat{y}_0 + \ve \widehat{\ul y}, \xi) \big) \widetilde \ff_{\ve}^\varphi \Big) \\ 
	 							&\qquad- \wt \ff_{\ve}\frac{\partial}{\partial \ul y_d} \Big( \big(\widetilde{a_d}(\widehat{y}_0, \xi) - \widetilde{a_d}(\widehat{y} _0+ \ve \widehat{\ul y}, \xi) \big) \varphi \Big)
								+ \frac{\partial}{\partial \xi} \big(\varphi \widetilde{\qq}_{\ve} \big)  \\
& \qquad  - \sum_{k=1}^\infty \frac{\po^2}{\po \ul t\po \xi} \left(\varphi \int_{t_0}^{t_0+\ve\ul t} 
 \widetilde{g}_{k,\ve}(\ul y,\xi) \wt \ff(t,y^0+\ve \ul y, \xi ) \, d\beta_{k}(t)\right)
\\ 
&\qquad + \sum_{k=1}^\infty \frac{\po}{ \po \xi} \left(\varphi _{\ul t} \int_{t_0}^{t_0+\ve\ul t} 
 \widetilde{g}_{k,\ve}(\ul y,\xi) \wt \ff(t,y^0+\ve \ul y, \xi ) \, d\beta_{k}(t)\right)\\
&\qquad+ \sum_{k=1}^\infty \frac{\po}{\po \ul t} \left(\varphi \int_{t_0}^{t_0+\ve\ul t} 
\frac{\partial \widetilde{g}_{k,\ve}}{\po \xi}(\ul y,\xi) \wt \ff(t,y^0+\ve \ul y, \xi \big) \,d\beta_{k}(t)\right)
\\
&\qquad- \sum_{k=1}^\infty \left(\varphi_{\ul t} \int_{t_0}^{t_0+\ve\ul t} 
\frac{\partial \widetilde{g}_{k,\ve}}{\po \xi}(\ul y,\xi) \wt \ff(t,y^0+\ve \ul y, \xi \big) \,d\beta_{k}(t)\right)\\
&\qquad+ \sum_{k=1}^\infty  \d_{\xi=0} \frac{\po}{\po \ul t}  \left( \varphi \int_{t_0}^{t_0+\ve\ul t}  \widetilde{g}_{k, \ve} \,d\b_k(t)\right)\\
&\qquad-\sum_{k=1}^\infty  \d_{\xi=0}  \left( \varphi_{\ul t}\int_{t_0}^{t_0+\ve\ul t}  \widetilde{g}_{k, \ve} \,d\b_k(t)\right).									
	 \end{aligned}
	 \end{equation}
Let us denote
$$
\begin{aligned}
&\ell_\ve^{(1)} (\ul t,\ul y,\xi):= \sum_{k=1}^\infty \int_{t_0}^{t_0+\ve \ul t} \wt g_{k,\ve}(\ul y,\xi)\wt \ff (t,y^0+\ve\ul y,\xi)\, d\b_k(t),\\
& \ell_\ve^{(2)} (\ul t,\ul y,\xi):= \sum_{k=1}^\infty \int_{t_0}^{t_0+\ve \ul t} \frac{\po\wt g_{k,\ve}}{\po \xi}(\ul y,\xi)\wt \ff(t,y^0+\ve\ul y,\xi)\, d\b_k(t),\\
& \ell_\ve^{(3)}(\ul t,\ul y) := \sum_{k=1}^\infty \int_{t_0}^{t_0+\ve \ul t} \wt g_{k,\ve}(\ul y,0)\, d\b_k(t).
\end{aligned}
$$
So, \eqref{5.eqfvephi0} can be written in a more concise way  as
\begin{equation} \label{e3.26-05-20-1}
	\begin{aligned}
	& \frac{\partial \wt \ff_{\ve}^\varphi}{\partial \ul t} + \widetilde{a}_{\widehat{y}_0}(\xi) \cdot \nabla \wt \ff_{\ve}^\varphi  = \wt \ff_{\ve}\Bigg(\frac{\partial \varphi}{\partial t} - \wt a_{\wh y_0}(\xi) \cdot \nabla \varphi \Bigg)  \\
	&\qquad+\frac{\partial}{\partial \ul y_d} \Big( \big(\widetilde{a_d}(\widehat{y}_0, \xi) - \widetilde{a_d}(\widehat{y}_0 + \ve \widehat{\ul y}, \xi) \big) \widetilde \ff_{\ve}^\varphi \Big) \\ 
	 							&\qquad- \wt \ff_{\ve}\frac{\partial}{\partial \ul y_d} \Big( \big(\widetilde{a_d}(\widehat{y}_0, \xi) - \widetilde{a_d}(\widehat{y} _0+ \ve \widehat{\ul y}, \xi) \big) \varphi \Big)
								+ \frac{\partial}{\partial \xi} \big(\varphi \widetilde{\qq}_{\ve} \big)  \\
&   \qquad- \frac{\po^2}{\po \ul t\po \xi} (\varphi \ell_\ve^{(1)}) +\frac{\po}{\po\xi}(\varphi_{\ul t}\ell_\ve^{(1)})+  \frac{\po}{\po \ul t} (\varphi \ell_\ve^{(2)}) -\varphi_{\ul t}\ell_\ve^{(2)} \\
&\qquad+ \d_{\xi=0}\frac{\po}{\po\ul t} (\varphi \ell_\ve^{(3)})-\d_{\xi=0}\varphi_{\ul t}\ell_\ve^{(3)}.								
	 \end{aligned}
	 \end{equation}

	 On the other hand, from \eqref{3.eqfve3},    $(\varphi\wt\ff) = \wt\ff^\varphi$ verifies
	 \begin{equation}
	 \frac{\partial \wt \ff^\varphi}{\partial \ul t} + a_{\widehat{y}}(\xi)\cdot \nabla_{\ul y} \wt \ff^\varphi =\wt \ff\Bigg(\frac{\partial \varphi}{\partial \ul t} + a_{\widehat{y}}(\xi). \nabla_{\ul y} \varphi\Bigg) . \label{5.eqfphi}
	 \end{equation}
	 
{\em Step 2: }	Let $\Om_2(t_0,\wh y_0)$ and $\ve_n(t_0,\wh y_0)$ be the set of total measure and the subsequence obtained in Lemma~\ref{L:3.03}. We claim that, for any bounded open set $V\subset\subset \R^d\X(0,\infty)$, there is a set of total measure 
$\Om_3(t_0,\wh y_0)\subset \Om_2(t_0,\wh y_0)$ and a subsequence of $\ve_n((t_0,\wh y_0)$, also denoted $\ve_n(t_0,\wh y_0)$,  such that, for all  $\om\in\Om_2$, $\ell_{\ve_n}^{(1)}, \ell_{\ve_n}^{(2)}\in L^2(V\X(-L,L))$,
	 $\ell_{\ve_n}^{(3)}\in L^2(V)$, and $\ell_{\ve_n}^{(1)}, \ell_{\ve_n}^{(2)}\to  0$ in $L^2(V\X(-L,L))$ and $\ell_{\ve_n}^{(3)}\to 0$ in $L^2(V)$ as $n\to\infty$.   
	 Moreover, by a standard diagonal argument, we can find a set of total measure $\Om_3(t_0,\wh y_0)\subset\Om_2(t_0,\wh y_0)$ and a subsequence of $\ve_n(t_0,\wh y_0)$, also denotes $\ve_n(t_0,\wh y_0)$ such that the assertion is true for any $V\subset\subset \R^d\X(0,\infty)$.
	 
	 \bigskip
	  Indeed, it suffices to prove the assertion for $\ell_\ve^{(1)}$ since the proof for the others is similar. By It\^o isometry, we have
	  \begin{multline*}
	  \bbE \int_{-L}^L\int_V |\ell_{\ve}^{(1)}|^2 \,d\ul t\,d\ul y\, d\xi \\ 
	  =\bbE\int_{-L}^L\int_V\sum_{k=1}^\infty\left| \int_{t_0}^{t_0+\ve\ul t} | \wt g_{k,\ve}(\ul y,\xi)|^2 |\wt \ff (t,y^0+\ve\ul y,\xi)|^2\, dt\right| \,d\ul t\,d\ul y\,d\xi\\
	  \le \bbE\int_{-L}^L\int_V \left| \int_{t_0}^{t_0+\ve\ul t} D(1+|\xi|^2)\, dt\right| \,d\ul t\,d\ul y\,d\x\\
	  \le C(V,L, D)\ve\to 0,\quad \text{as $\ve\to0$}. 
	 \end{multline*}
	 Therefore, making $\ve=\ve_n(t_0,\wh y_0)$, we deduce that we can obtain a set of total measure $\Om_3(t_0,\wh y_0)\subset \Om_2(t_0,\wh y_0)$ , and a subsequence of $\ve_n(t_0,\wh y_0)$ also
	 denoted $\ve_n(t_0,\wh y_0)$ such that the claim for $\ell_\ve^{(1)}$ holds. The proof for $\ell_\ve^{(2)}, \ell_\ve^{(3)}$ follows the same lines and that the claim holds for all $V\subset\subset \R_+^d$ follows trivially
	 by a standard diagonal argument.  
	 
	 \bigskip
	In view of the claim just proven in this step, the remaining of the proof can be reduced to averaging techniques along the lines of the averaging lemma by Lions, Perthame and Tadmor \cite{LPT} (see also \cite{P}). 
	For the convenience of the reader, we present a detailed, self contained proof. In what follows,  we are going  to show  that, for all $\om\in\Om_3(t_0,\wh y_0)$ and for  $\ve_n=\ve_n(t_0,\wh y_0)$, just obtained, 	
	  \begin{equation}
	\int_{-L}^L \wt \ff_{\ve_n}^\varphi\, d\xi \rightarrow \int_{-L}^L \wt \ff^{\varphi}\, d\xi \text{ in $L^1(\R\X\R^d)$}, \label{5.desiredlim}
	 \end{equation}
	 implying that $\int_{-L}^L \wt \ff_{\ve_n}\,d\xi$ is a Cauchy sequence in  $L^1(V)$. Since  each $\wt\ff_{\ve_n'}$ is a $\chi$-function, we thus see that $\wt\ff_{\ve_n}$ converges in $L^1(V \X (-L,L))$, by Lemma~\ref{2.criterion0}. 
	 In virtue of Lemma~\ref{L:3.03},  its limit is necessarily 
	 $ \wt \ff =\ff^\tau(t_0,\wh y_0,\xi)$.  Hence, in particular, the latter is a $\chi$-function. 
	
	\bigskip 
{\em Step 3:}  	 So, let $\om\in\Om_3(t_0,\wh y_0)$ be fixed  in the following steps and let us write 
	 \begin{equation}
	\gf_{\ve} = \wt\ff_{\ve}^\varphi -\wt \ff^\varphi. \label{5.defg}
	 \end{equation}
	 Notice that, since $\ff_{\ve}^\varphi$ and $\ff^\varphi$ belong to $L^2(\R\X\R^d)$, we can apply  the  Fourier transform to them. Let us consider two functions $\psi$ and $\zeta$ be $C^\infty(\R)$ such that
	 \begin{enumerate}
	 	\item $\zeta(z) = 1$ for $|z|<1$ and $\zeta(z) = 0$ for $|z|>2$, and
	 	\item $\zeta(z) + \psi(z) = 1$ for all $z \in \R$.
	 \end{enumerate}
 	For $\d>0$ and $\g > 0$ let us decompose $\gf_{\ve}$ as
 	\begin{align}
 	\FF_{\ul t, \ul y} \gf_{\ve} &= \zeta\Bigg(\frac{\sqrt{\tau^2 + |\k|^2}}{\g}\Bigg) \FF_{\ul t, \ul y} \gf_{\ve} + \psi\Bigg(\frac{\sqrt{\tau^2 + |\k|^2}}{\g}\Bigg)\FF_{\ul t, \ul y}\gf_{\ve}  \nonumber\\
 									&= \zeta\Bigg(\frac{\sqrt{\tau^2 + |\k|^2}}{\g}\Bigg)\FF_{\ul t, \ul y}\gf_{\ve} + \psi\Bigg(\frac{\sqrt{\tau^2 + |\k|^2}}{\g}\Bigg)\zeta\Bigg( \frac{|\LL_{\wh y}(i\tau', i\k' , \xi)|}{\d} \Bigg)\FF_{\ul t, \ul y}\gf_{\ve} \nonumber\\ 
									&\quad+ \psi\Bigg(\frac{\sqrt{\tau^2 + |\k|^2}}{\g}\Bigg)\psi\Bigg( \frac{|\LL_{\wh y}(i\tau', i\k', \xi)|}{\d} \Bigg)\FF_{\ul t, \ul y}\gf_{\ve} \nonumber \\
 									&\stackrel{\text{def}}{=} \FF_{\ul t, \ul y} \gf_{\ve}^{(1)} + \FF_{\ul t, \ul y} \gf_{\ve}^{(2)} + \FF_{\ul t, \ul y} \gf_{\ve}^{(3)}, \label{5.decompg}
 	\end{align}
 	where we denote by $(\tau, \k)$ the frequency variables corresponding to  $(\ul t, \ul y)$, and $(\tau',\k') = \frac{1}{\sqrt{\tau^2 + |\k|^2}} (\tau, \k)$ for $(\tau, \k) \neq 0$. 
	 	
 	Finally, let $\eta \in C_c^\infty(\R)$ be such that $\eta \equiv 1$ in $[-L,L]$, $\eta(\xi)=0$, for $\xi\notin [L_0,L_0]$,  and let us put
 	\begin{equation}\label{5.defv}
 	\begin{cases}
 		\int_\R \eta \gf_{\ve}^{(1)}\, d\xi = \vf_{\ve}^{(1)}, \\ \\
		\int_\R \eta \gf_{\ve}^{(2)}\, d\xi = \vf_{\ve}^{(2)}, \\ \\
		\int_\R \eta \gf_{\ve}^{(3)}\, d\xi = \vf_{\ve}^{(3)}. 
 	\end{cases}
 	\end{equation}
	Since $\theta\equiv 1$ on $\supp \varphi$,  we  have that
	\begin{equation}
	\int_{-L}^L \gf_{\ve}\, d\xi = \t \vf_{\ve}^{(1)} + \t \vf_{\ve}^{(2)} + \t \vf_{\ve}^{(3)}. \label{5.decompv}
	\end{equation}
	
{\em Step 4}:   (Analysis of $\vf_{\ve}^{(1)}$). Notice that from \eqref{5.defv} and the Cauchy--Schwarz inequality, we trivially have
	\begin{equation}
	\Vert \vf_{\ve}^{(1)} \Vert_{L^2(\R\X\R^d)} \leq \Vert \eta \Vert_{L^2(\R)}\Vert \gf_{\ve}^{(1)} \Vert_{L^2(\R\X\R^d\X(-L,L))}. \label{5.destrivial}
	\end{equation}
	On other hand,  $\sup_{-L_0\le \xi\le L_0} \Vert \gf_{\ve}(\cdot, \cdot, \xi) \Vert_{L^1(\R\X\R^d)} \leq \text{(const.)}$  uniformly in $\ve$. So, we have  that $ \sup_{-L_0\le \xi\le L_0}\Vert \FF_{\ul t, \ul y} \gf_{\ve}(\cdot, \cdot, \xi) \Vert_{L^\infty} \leq \text{(const.)}$ as well. Thus, by \eqref{5.destrivial} and  the Plancherel identity, we have  
	\begin{align*}
	\int_\R\int_{\R^d} | \vf_{\ve'}^{(1)}(\ul t, \ul y) |^2\, d\ul t d\ul y &\leq \Vert \eta \Vert_{L^2(\R)}^2 \int_{-L_0}^{L_0} \int\limits_{\R^{d+1}}  \Bigg| \zeta\Bigg(\frac{\sqrt{\tau^2 + |\k|^2}}{\g}\Bigg)\FF_{\ul t, \ul y}\gf_{\ve}(\tau, \k,\xi) \Bigg|^2\, d\k\, d\tau\,d\xi \\
	&\leq \Vert \eta \Vert_{L^2(\R)}^2 \int_{-L_0}^{L_0}  \int_{\tau^2 + |\kappa|^2 \leq 4 \gamma^2 } | \FF_{\ul t, \ul y}\gf_{\ve}(\tau, \k,\xi) |^2 \, d\k\, d\tau\,d\xi \\
	&\leq C\,\gamma^{d+1},
	\end{align*}
	where $C$ depends only on $d$, $L_0$, $\varphi$ and $\Vert \eta \Vert_{L^2}$. Thus,
	\begin{equation}
	\limsup_{\ve \to 0} \Vert \t \vf_{\ve'}^{(1)} \Vert_{L^1(\R\X\R^d)} \leq C \gamma^{(d+1)/2}, \label{5.estv1}
	\end{equation}
	for some $C = C(d, V, L, \eta)$.
	
	{\em Step 5}: (Analysis of $\vf_{\ve}^{(2)}$). Let us refine the trivial inequality \eqref{5.destrivial}. Clearly,
	\begin{align*}
	(\FF_{\ul t, \ul y} \vf_{\ve}^{(2)})(\tau,\k) &= \int_\R \eta(\xi) (\FF_{\ul t, \ul y} \gf^{(2)})(\tau, \k, \xi)\, d\xi \\	
										  &= \int_\R \eta(\xi) \psi\Bigg(\frac{\sqrt{\tau^2 + |\k|^2}}{\g}\Bigg)\zeta\Bigg( \frac{|\LL_{\wh y_0}(i\tau', i\k', \xi)|}{\d} \Bigg)(\FF_{\ul t, \ul y}\gf_{\ve})(\tau,\k,\xi)\, d\xi.
	\end{align*}
	Consequently, combining the Cauchy--Schwarz inequality and the Plancherel identity again,
	\begin{align*}
	\Vert \vf_{\ve_n'}^{(2)} \Vert_{L^2(\R\X\R^d)}^2 &= \Vert \FF_{\ul t, \ul y} \vf_{\ve}^{(2)} \Vert_{L^2(\R\X\R^d)}^2 \\
	&\leq \int_{\R\X\R^d} \Bigg(\int_{\R} 1_{\{ |\LL_{\wh y_0}(i\tau', i\k', \cdot)| \leq 2 \d\}}(\xi)\eta(\xi)^2 \, d\xi\Bigg)\Bigg( \int_\R \Bigg| \psi\Bigg(\frac{\sqrt{\tau^2 + |\k|^2}}{\g}\Bigg) \\
	&\quad\quad\quad\quad\quad\quad\quad\quad \zeta\Bigg( \frac{|\LL_{\wh y_0}(i\tau', i\k', \xi)|}{\d} \Bigg)(\FF_{\ul t, \ul y}\gf_{\ve})(\tau,\k,\xi) \Bigg|^2 \, d\xi\Bigg) d\k d\tau\\
	&\leq \Big(\sup_{|\tau|^2 + |\kappa|^2 = 1}\big| \big\{ \xi \in \supp \eta; |\LL_{\wh y_0}(i\tau, i\k, \xi)| \leq 2\d \big\} \big| \Big) \\&\quad\quad\quad\quad\quad\quad\quad\quad\quad\quad\quad\quad\quad\quad 
	\Vert \eta \Vert_{L^\infty(\R)}^2 \Vert \gf_{\ve} \Vert_{L^2(\R\X\R^d\X(-L,L))}^2.
	\end{align*} 
	That is,  since $\Vert \gf_{\ve}\Vert_{L^2(\R\X\R^d\X(-L,L))}^2$ is uniformly bounded,
	\begin{align}
	\Vert \vf_{\ve_n'}^{(2)} \Vert_{L^2(\R\X\R^d)}^2 &\leq C \Big(\sup_{|\tau|^2 + |\kappa|^2 = 1}\big| \big\{ \xi \in \supp \eta\,:\, |\LL_{\wh y_0}(i\tau, i\k, \xi)| \leq 2\d \big\} \big| \Big) \label{5.desdeg0} \\
	&\leq C \, O(\d) \label{5.desdeg1},
	\end{align}
	according to Proposition \ref{3.propL}. As a consequence, for all $\om\in\Om_3(t_0,\wh y_0)$,
	\begin{equation}
	\limsup_{\ve \to 0}  \Vert \t \vf_{\ve}^{(2)} \Vert_{L^1(\R\X\R^d)} \leq C\, O(\d), \label{5.estv2}
	\end{equation}
	for some $C = C(V, L, \eta, \LL_{\widehat{y}_0})$.
	
  {\em Step 6:} (Analysis of $\vf_{\ve_n'}^{(3)}$) Let $\ve_n(t_0,\wh y_0)$ be a subsequence obtained as in step 2. 
  Using the equations \eqref{e3.26-05-20-1} and \eqref{5.eqfphi}, we see that
	\begin{align*}
		&\LL_{\widehat{y}}(i\tau,i\k,\xi) (\FF_{\ul t, \ul y} \gf_{\ve_n}) = \,\, \FF_{\ul t, \ul y} \Bigg[ \big(\wt\ff - \wt \ff_{\ve_n} \big) \Bigg(\frac{\partial \varphi}{\partial t} - \wt a_{\wh y}(\xi) . \nabla \varphi \Bigg)\Bigg] \\
		&- \FF_{\ul t, \ul y}\Bigg[ \frac{\partial}{\partial \ul y_d} \Big( \big(\widetilde{a_d}(\widehat{y}, \xi) - \widetilde{a_d}(\widehat{y} + \ve \widehat{\ul y}, \xi) \big) \widetilde \ff_{\ve_n}^\varphi \Big)\Bigg]\\
		&+ \FF_{\ul t, \ul y} \Bigg[ \wt \ff_{\ve_n}\frac{\partial}{\partial \ul y_d} \Big( \big(\widetilde{a_d}(\widehat{y}, \xi) - \widetilde{a_d}(\widehat{y} + \ve_n \widehat{\ul y}, \xi) \big) \varphi \Big)\Bigg] 
		- \frac{\partial}{\partial \xi} \FF_{\ul t, \ul y}\Bigg[ \big(\varphi \widetilde{\qq}_{\ve_n} \big) \Bigg]\\
		& -\frac{\po}{\po\xi} \FF_{\ul t, \ul y}\Big[ \frac{\po}{\po \ul t} (\varphi \ell_{\ve_n}^{(1)})\Big] + \frac{\po}{\po\xi}\FF_{\ul t, \ul y}\Big[\varphi_{\ul t}\ell_{\ve_n}^{(1)}\Big]+   \FF_{\ul t, \ul y}\Big[\frac{\po}{\po \ul t} (\varphi \ell_{\ve_n}^{(2)})\Big] 
		- \FF_{\ul t, \ul y}\Big[\varphi_{\ul t}\ell_{\ve_n}^{(2)}\Big] \\
&+ \d_{\xi=0}\FF_{\ul t, \ul y} \Big[\frac{\po}{\po\ul t} (\varphi \ell_{\ve_n}^{(3)})\Big]- \d_{\xi=0}\FF_{\ul t, \ul y}\Big[\varphi_{\ul t}\ell_{\ve_n}^{(3)}\Big].
	\end{align*}
	 Multiply now the equation above by $\eta(\xi)$ $\psi(\sqrt{\tau^2 + |\k|^2}/\g)$ $\psi(\LL_{\wh y}(i\tau', i\k', \xi)/\d)$. Since $\psi(\LL_{\wh y}(i\tau', i\k', \xi)/\d)$ avoids the set of degeneracy $[\LL_{\wh y} = 0]$, we may now divide by $\LL_{\wh y}(i\tau, i\k, \xi) = \sqrt{\tau^2 + |\k|^2}\LL_{\wh y}(i\tau', i\k', \xi)$, obtaining thus, after integrating in $\xi$,
	\begin{align}
	\vf^{(3)}=  \frac{1}{\g\d}&  \int_\R \eta(\xi) \FF_{\ul t, \ul y}^{-1}\Bigg[\wt \psi\Bigg(\frac{|\LL_{\widehat{y}}(i\tau', i\k',\xi)|}{\d}\Bigg)\wt \psi\big(\sqrt{\tau^2 + |\k|^2}/\g\big) (\FF_{\ul t, \ul y} \rr_{\ve_n}) \Bigg] \,d\xi \nonumber\\
										 &-\frac{1}{\g\d}  \int_\R \eta(\xi) \FF_{\ul t, \ul y}^{-1}\Bigg[\wt \psi\Bigg(\frac{|\LL_{\widehat{y}}(i\tau', i\k',\xi)|}{\d}\Bigg)\wt \psi\big(\sqrt{\tau^2 + |\k|^2}/\g\big) 
										 \frac{\partial}{\partial \xi}(\FF_{\ul t, \ul y} (\varphi \qq_{\ve_n})) \Bigg]\, d\xi \nonumber\\
										 &+ \frac1{\g\d} \int_\R \eta(\xi)  \FF_{\ul t, \ul y}^{-1}\Bigg[\wt\psi\Bigg(\frac{|\LL_{\widehat{y}}(i\tau', i\k',\xi)|}{\d}\Bigg)\nonumber\\
										 &\quad\quad\quad\quad\quad\quad\quad\quad\wt \psi\big(\sqrt{\tau^2 + |\k|^2}/\g\big)
										 \frac{\partial}{\partial \xi} 
										 \Big(-i\tau \FF_{\ul t, \ul y}\Big[ \varphi \ell_{\ve_n}^{(1)}\Big] +\FF_{\ul t, \ul y}\Big[\varphi_{\ul t} \ell_{\ve_n}^{(1)}\Big] 
										 \Big)\Bigg]\,d\xi \nonumber\\
										&- \frac1{\g\d} \int_\R \eta(\xi)  \FF_{\ul t, \ul y}^{-1}\Bigg[\wt\psi\Bigg(\frac{|\LL_{\widehat{y}}(i\tau', i\k',\xi)|}{\d}\Bigg)\nonumber\\
										 &\quad\quad\quad\quad\quad\quad\quad\quad\tilde \psi\big(\sqrt{\tau^2 + |\k|^2}/\g\big)
										 \Big(-i\tau \FF_{\ul t, \ul y}\Big[\varphi \ell_{\ve_n}^{(2)}\Big] +\FF_{\ul t, \ul y}\Big[\varphi_{\ul t} \ell_{\ve_n}^{(2)}\Big] 
										 \Big)\Bigg]\,d\xi \nonumber\\
									        &- \frac{1}{\g\d}  \eta(0) \FF_{\ul t, \ul y}^{-1} \Bigg[ \wt\psi\Bigg(\frac{|\LL_{\widehat{y}}(i\tau', i\k',0)|}{\d}\Bigg)\nonumber\\
									         &\quad\quad\quad\quad\quad\quad\quad\quad\tilde\psi\big(\sqrt{\tau^2 + |\k|^2}/\g\big)
										 \Big(-i\tau \FF_{\ul t, \ul y}\Big[\varphi \ell_{\ve_n}^{(3)}\Big] +\FF_{\ul t, \ul y}\Big[\varphi_{\ul t} \ell_{\ve_n}^{(3)}\Big] 
										 \Big)\Bigg]\,d\xi \nonumber\\
										 &= (I)_{\ve_n} - (II)_{\ve_n} + (III)_{\ve_n} - (IV)_{\ve_n} - (V)_{\ve_n}, \label{5.defI} 
	\end{align}
	where $\wt \psi(z) = \psi(z)/z$, and $\rr_{\ve_n'} \rightarrow 0$ in $H^{-1}(\R\X\R^d\X\R)$ due to \eqref{3.lima} and Lemma~\ref{L:3.03}. Let us estimate each term separately.
	
	\bigskip
	
	{\em Step 7}: (Analysis of $ (IV)_{\ve_n}, (V)_{\ve_n} $). Since
	 $$
	  \tau \wt\psi\Bigg(\frac{\LL_{\widehat{y}}(i\tau', i\k',\xi)}{\d}\Bigg)\tilde\psi\big(\sqrt{\tau^2 + |\k|^2}/\g\big)\; \text{and} \; 
	\wt\psi\Bigg(\frac{\LL_{\widehat{y}}(i\tau', i\k',\xi)}{\d}\Bigg)\tilde\psi\big(\sqrt{\tau^2 + |\k|^2}/\g\big)
	$$ 
	are bounded for $(\tau,\k)\in\R^{d+1}$, uniformly in $\xi\in(-L,L)$, and since from
	step~2, $\ell_{\ve_n}^{(2)}\to0$ in $L^2(V\X(-L,L))$, for any $V\subset\subset\R^d\X(0,\infty)$, and $ \ell_{\ve_n}^{(3)}\to 0$ in $L^2(V)$, for any $V\subset\subset \R^d\X(0,\infty)$,
	by Plancherel identity we easily conclude that 
	$$
	\|(IV)_{\ve_n}\|_{L^2(\R^{d+1})}, \|(V)_{\ve_n}\|_{L^2(\R^{d+1})} \to 0, \qquad \text{ as $n\to\infty$}.
	$$ 
	And so
	$$
	\|\theta (IV)_{\ve_n}\|_{L^1(\R^{d+1})}, \|\theta (V)_{\ve_n}\|_{L^1(\R^{d+1})} \to 0, \qquad \text{ as $n\to\infty$}.
	$$ 
	
	{\em Step 8}: (Analysis of $(III)_{\ve_n}$). We have
	\begin{align*}
	(III)_{\ve_n} &=   \frac1{\g\d} \int_\R \eta(\xi)  \FF_{\ul t, \ul y}^{-1}\Bigg[\wt\psi\Bigg(\frac{|\LL_{\widehat{y}}(i\tau', i\k',\xi)|}{\d}\Bigg)\nonumber\\
										 &\quad\quad\quad\quad\quad\quad\quad\quad\wt \psi\big(\sqrt{\tau^2 + |\k|^2}/\g\big)
										 \frac{\partial}{\partial \xi} 
										 \Big(-i\tau \FF_{\ul t, \ul y}\Big[ \varphi \ell_{\ve_n}^{(1)}\Big] +\FF_{\ul t, \ul y}\Big[\varphi_{\ul t} \ell_{\ve_n}^{(1)}\Big] 
										 \Big)\Bigg]\,d\xi \nonumber\\
				  &=  -\frac{1}{\g\d} \int_\R \eta'(\xi) \FF_{\ul t, \ul y}^{-1} \Bigg[ \wt \psi\Bigg(\frac{|\LL_{\widehat{y}}(i\tau', i\k',\xi)|}{\d}\Bigg) \\ 
				  &\quad\quad\quad\quad\quad\quad\quad\quad\wt \psi\big(\sqrt{\tau^2 + |\k|^2}/\g\big)
										\Big(-i\tau \FF_{\ul t, \ul y}\Big[ \varphi \ell_{\ve_n}^{(1)}\Big] +\FF_{\ul t, \ul y}\Big[\varphi_{\ul t} \ell_{\ve_n}^{(1)}\Big] 
										 \Big)\Bigg]\,d\xi \nonumber\\
				  &\quad - \frac{1}{\g\d^2} \int_\R \eta(\xi) \FF_{\ul t, \ul y}^{-1} \Bigg[ \wt \psi' \Bigg(\frac{|\LL_{\widehat{y}}(i\tau', i\k',\xi)|}{\d}\Bigg) \frac{\partial \LL_{\wh y}}{\partial\xi}(i\tau',i\k',\xi) \s(\tau',\k',\xi) \\ 	
		  &\quad\quad\quad\quad\quad\quad\quad\quad\wt \psi\big(\sqrt{\tau^2 + |\k|^2}/\g\big)
										\Big(-i\tau \FF_{\ul t, \ul y}\Big[ \varphi \ell_{\ve_n}^{(1)}\Big] +\FF_{\ul t, \ul y}\Big[\varphi_{\ul t} \ell_{\ve_n}^{(1)}\Big] 
										 \Big)\Bigg]\,d\xi, 		  		  
	\end{align*}
where $|\s(\tau',\k',\xi)|=1$. From what has been said in the previous step and from the fact that 
$$
\Big|\wt \psi' \Bigg(\frac{|\LL_{\widehat{y}}(i\tau', i\k',\xi)|}{\d}\Bigg) \frac{\partial \LL_{\wh y}}{\partial\xi}(i\tau',i\k',\xi) \s(\tau',\k',\xi) \wt \psi\big(\sqrt{\tau^2 + |\k|^2}/\g\big)\Big|(1+|\tau|)
$$
is  bounded for $(\tau,\k)\in\R^{d+1}$ uniformly in $\xi\in(-L,L)$, and since from step~2, $\ell_{\ve_n}^{(1)}\to0$ in $L^2(V\X(-L,L))$, for any $V\subset\subset\R^d\X(0,\infty)$,	
by Plancherel identity we easily conclude that 
	$$
	\|(III)_{\ve_n}\|_{L^2(\R^{d+1})} \to 0, \qquad \text{ as $n\to\infty$}.
	$$ 
	And so
	$$
	\|\theta (III)_{\ve_n}\|_{L^1(\R^{d+1})} \to 0, \qquad \text{ as $n\to\infty$}.
	$$ 
	
	{\em  Step 9}:  (Analysis of $(I)_{\ve_n}$). Because $\rr_{\ve_n'} \rightarrow 0$ in  $H^{-1}(\R\X\R^d\X(-L,L))$, we may write $\rr_{\ve_n} = (1 - \Delta_{\ul t, \ul y})^{1/2} (1 + \partial_\xi) c_{\ve_n}(\ul t, \ul y, \xi)$, 
	with $c_{\ve_n} \rightarrow 0$ in $L^2(\R\X\R^d\X(-L,L))$. That is,
	\begin{align*}
	&\wt\psi\big(\sqrt{\tau^2 + |\k|^2}/\g\big) \FF_{\ul t, \ul y} (\rr_{\ve_n}) \\ 
	&= \wt \psi\Bigg(\frac{\sqrt{\tau^2 + |\k|^2}}{\g}\Bigg) (1 + \tau^2 + |\k|^2)^{1/2} \Bigg(1 + \frac{\partial}{\partial_\xi}\Bigg) (\FF_{\ul t, \ul y}c_{\ve_n}).
	\end{align*}
	 However, since $ \psi(\sqrt{\tau^2+|\k|^2}/\g)$ vanishes near the origin, we conclude that
	$$
	\wt \psi\big(\sqrt{\tau^2 + |\k|^2}/\g\big) (\FF_{\ul t, \ul y} (\rr_{\ve_n}))  = \Bigg(1 + \frac{\partial}{\partial \xi} \Bigg)\cc_{\ve_n}(\tau,\kappa,\xi),
	$$
	for some $\cc_{\ve_n} \rightarrow 0$ in $L^2(\R\X\R^d\X(-L,L))$.  Plugging this formula into the definition of $(I)_{\ve_n}$ and performing the required integration by parts yields
	\begin{align*}
	(\FF_{\ul t, \ul y}(I)_{\ve_n}) &= -\frac{1}{\g\d}\int_\R \eta'(\xi)  \wt \psi\Bigg(\frac{|\LL_{\widehat{y}}(i\tau', i\k',\xi)|}{\d}\Bigg) \cc_{\ve_n}(\tau, \k, \xi) \, d\xi \\
				   &-\frac{1}{\g\d^2}\int_\R \eta(\xi)  \wt \psi'\Bigg(\frac{|\LL_{\widehat{y}}(i\tau', i\k',\xi)|}{\d}\Bigg) \frac{\partial \LL_{\widehat{y}}}{\partial \xi}(i\tau', i\k',\xi) \s(\tau',\k',\xi)\cc_{\ve_n}(\tau, \k, \xi) \, d\xi \\
				   &+ \frac{1}{\g\d}\int_\R \eta(\xi)  \wt \psi\Bigg(\frac{|\LL_{\widehat{y}}(i\tau', i\k',\xi)|}{\d}\Bigg) \cc_{\ve_n}(\tau, \k, \xi) \, d\xi.
	\end{align*}
	Now, combining the trivial estimate \eqref{5.destrivial} with the fact that $\wt\psi$, $\wt \psi'$ and $\partial_\xi \LL (i\tau',i\k',\xi)$ are bounded for $\xi \in \supp \eta$, 
	\begin{align*}
	 \Vert (I)_{\ve_n} \Vert_{L^2(\R\X\R^d)}^2 &\leq C\Vert \eta \Vert_{H^1(\R)}^2 \Bigg(\frac{1}{\d} + \frac{1}{\d^2}\Bigg)  \Vert \cc_{\ve_n} \Vert_{L^2(\R\X\R^d)}^2 \\ 
	& \rightarrow 0, \text{ as $\ve_n \to 0$}.
	\end{align*}
	As a result, we establish that
	\begin{equation}
	\limsup_{\ve_n \to 0} \Vert \t (I)_{\ve_n} \Vert_{L^1(\R\X\R^d)} = 0. \label{5.limI}
	\end{equation}
	
	{\em Step 10}:  (Analysis of $(II)_{\ve_n'}$). 
	Observe that, from \eqref{3.eqfve} and the fact that $\wt\qq_{\ve_n}$ is supported in $|\xi|\leq L$, it holds
	\begin{align*}
	\widetilde{\qq}_{\ve_n'} = \frac{\partial}{\partial \ul t} \Bigg\{ \int_{-L}^\xi \Bigg[ \wt \ff_{\ve_n'} - \sum_{k=1}^\infty \int_{t_0}^{t_0+\ve\ul t} \wt g_{k, \ve_n}&\d_{\xi = \wt \uf (t,y^0+\ve\ul y)} \, d\b_k(t)\Bigg] \, d\xi \Bigg\} \\&
	+ \nabla_{\ul y} . \Bigg\{ \int_{-L}^\xi \ff_{\ve_n'} \wt \abf(\ul y, \xi)  \, d\xi \Bigg\},
	\end{align*}
	where $\wt \abf(\ul y, \xi) = (\wh \abf(\xi), \wt a_d(\wh y_0+ \ve\ul y, \xi))$. This implies that, for all $1 \leq r,  < \infty$,
	$$
	\varphi \widetilde{\qq}_{\ve_n'} \in W^{-1,r}(\R\X\R^d\X(-L,L)),
	$$
	and it is uniformly bounded in this space. On the other hand, by the well known compactness of the embedding of the space of the finite variation measures over $\R\X\R^d\X(-L,L)$ in $W^{-1,q}(\R\X\R^d\X(-L,L))$ for some $1<q<2$,
	by interpolation we deduce that
	$$
	\varphi \widetilde{\qq}_{\ve_n} \in H^{-1}(\R\X\R^d\X(-L,L)),
	$$
	and $\|\widetilde{\qq}_{\ve_n} \Vert_{H^{-1}}^2 \to 0$, as $\ve_n \rightarrow 0$. Therefore, the analysis of $(II)_{\ve_n}$ reduces to that of $(I)_{\ve_n}$, which was made in the previous step.    	
	
	{\em Step 11}: (Conclusion). {}By the steps 6 to 10, we see that $\vf_{\ve_n'}^{(3)}\to 0$  in $L^2(\R\X\R^d)$ as $\ve_n\to0$, and so $\theta\vf_{\ve_n}^{(3)}\to 0$ in $L^1(\R\X\R^d)$.  On the other hand, by steps 4 and 5,
	sending $\g$ and $\d$ to 0, we conclude that   $\vf_{\ve_n'}^{(1)},\, \vf_{\ve_n'}^{(2)}\to 0$ in $L^2(\R\X\R^d)$ and again it follows that   $\theta\vf_{\ve_n'}^{(1)},\, \theta \vf_{\ve_n'}^{(2)}\to 0$  in $L^1(\R\X\R^d)$.
	It follows that $\wt u_{\ve_n}$  is a Cauchy sequence in $L^1_\loc(\R\X\R^d)$, therefore, so is $\wt\ff_{\ve_n}$ in $L_\loc^1(\R\X\R^d\X(-L,L))$, and by Lemma~\ref{L:3.03}, we conclude that $\wt\ff_{\ve_n}\to \ff^{\tau}(t_0,\wh y_0,\cdot)$ in  $L_\loc^1(\R\X\R^d\X(-L,L))$.   Moreover, since the latter holds for each $\om\in\Om_3(t_0,\wh y_0)$, by dominated convergence we deduce that  $\wt\ff_{\ve_n}\to \ff^{\tau}(t_0,\wh y_0,\cdot)$  in $L_\loc^1(\Om\X\R\X\R^d\X(-L,L))$,
	 and we conclude that $\ff^\tau(t_0,\wh y_0,\xi)$ is a $\chi$-function a.e.\ in $\Om\X(-L,L)$,  which concludes the proof.	
\end{proof}

{\em Conclusion of the proof of Theorem~\ref{T:3.1}.} 
By Lemma~\ref{L:3.04}, it follows that for all $(t_0,\wh y_0)\in\E$, $\ff^\tau(t_0,\wh y_0, \xi)$ is a $\chi$-function a.e.\  in $\Om\X(-L,L)$, and $\E\subset\Sigma$ has total measure, by Lemma~\ref{3.lemma1} and Lemma~\ref{3.lemma2}. Thus, $\ff^\tau(\cdot,\cdot, \cdot)$ is a $\chi$-function a.e.\ in $\Om\X\Sigma\X(-L,L)$.  Hence, from Lemma~\ref{2.criterion} we conclude that $\ff^\tau$ is a strong trace and integrating in $\xi$ we  arrive at the desired conclusion for $u$ on $\cO_0$ and $\Gamma_0$.  Covering $\po\cO$ with a finite set 
$\{\Gamma_\a\}_{\a\in I_0}$,  each $\Gamma_\a$  being the graph of a smooth function, we then  finally deduce \eqref{e03.2}. 

It remains to prove \eqref{e03.2'}.  From the essential strong convergence of $ \ff_\psi(\cdot, \cdot, s,\cdot)$ in $ L^1(\Om\X\Sigma\X(-L,L))$,  it follows that, given any sequence $s_n\to0$  in $(0,1)\setminus\mathcal{N}$, with $\mathcal{N}$ of null measure,  we can obtain
a subsequence still denoted $s_n$  such that  $ \ff_\psi(\om, \cdot, s_n,\cdot)\to \ff^\tau(\om,\cdot,\cdot)$ in $ L^1(\Sigma\X(-L,L))$ for $\om$ in a subset of total measure of $\Om$. Thus, by Lemma~\ref{2.criterion0}, we deduce that $\ff^\tau(\om,\cdot,\cdot)$ is
a $\chi$-function for $\om$ in a subset of total measure of $\Om$. By Lemma~\ref{L:3.02} we conclude, using again Lemma~\ref{2.criterion0}, that   $\esslim \ff_\psi(\om, \cdot, s,\cdot)=\ff^\tau(\om,\cdot,\cdot)$ in $L^1(\Sigma\X(-L,L))$  
for $\om$ in a subset of total measure of $\Om$.  Again integrating in $\xi$, we arrive at the desired conclusion for $u$ on $\cO_0$ and $\Gamma_0$, and so by   covering $\po\cO$ with a finite set 
$\{\Gamma_\a\}_{\a\in I_0}$ as above, we then  finally deduce \eqref{e03.2'}. 

\end{proof}

  \section{Doubling of Variables, Kato-Kruzhkov Inequality, Comparison Principle and Uniqueness}\label{S:3}

With the strong trace result from section~\ref{S:2} we can now move on to the Neumann problem \eqref{e1.1}-\eqref{e1.3}. As mentioned before, the existence of strong traces is the key for proving uniqueness of the solutions, which in turn, will also be essential in the prove of existence. In this section we prove a comparison principle for solutions of \eqref{e1.1}-\eqref{e1.3}. As a consequence, we also deduce uniqueness of solutions and a maximum principle. 

  Let us  first state the following extension of proposition~9 of \cite{DV}. In the latter, the statement of the proposition is preceded by an auxiliary result establishing the existence $\bbP$-a.s.\  of the right and left weak limits of $f(t,\cdot,\cdot)=1_{u(t,\cdot,\cdot)>\xi}$
  in the sense of distributions.  A similar fact holds here also, with exactly the same proof. Namely, there are $f^\pm(t,\cdot,\cdot)$ which coincide with $f(t,\cdot,\cdot)$ for a.a.\ $t\in[0,T]$, such that
  $$
  \la f(t_*\pm\ve),\varphi\ra\to \la f^\pm, \varphi\ra,
  $$
  for all $\varphi\in C_c^1(\cO\X\R)$, and all $t_*\in(0,T)$; for $t_*=0$, only the right limit,  and $t_*=T$, only the left limit. The proof of the proposition follows the same lines of the corresponding  one in \cite{DV} and so we omit it.

  \begin{proposition}[Doubling of variables]\label{P:4.1}  Let $f_i=1_{u_i>\xi}$, where $u_i$ is a kinetic solution of \eqref{e1.1}--\eqref{e1.3}, $i=1,2$. Set $\bar f_2=1-f_2$. Then, for $0\le t\le T$, and non-negative test functions $\rho\in C_c^\infty(\R^d)$, $\phi\in C_c^\infty(\cO)$, $\psi\in\ C_c^\infty(\R)$, with $\supp\rho\subset B(0;r)$ with $r>0$ sufficiently small, we have
  \begin{multline}\label{e4.1}
  \bbE\int_{\cO^2}\int_{\R^2} \rho(x-y)\psi(\xi-\z) \phi((x+y)/2)f_1^{\pm}(t,x,\xi) \bar f_2^\pm(t,y,\z)\,d\xi\,d\z\,dx\,dy \\
 \le \bbE\int_{\cO^2}\int_{\R^2} \rho(x-y)\psi(\xi-\z)\phi((x+y)/2)f_{1,0}(x,\xi)\bar f_{2,0} (y,\z)\,d\xi\,d\z\,dx\,dy+I_\rho+I_\phi+I_\psi,
 \end{multline}
 where
 \begin{multline*}
 I_\rho=\bbE\int_0^t\int_{\cO^2}\int_{\R^2}f_1(s,x,\xi)\bar f_2(s,y,\z) (\abf(\xi)-\abf(\z))\phi((x+y)/2)\psi(\xi-\z)\,d\xi\,d\z\\ \cdot \nabla_x\rho(x-y)\,dx\,dy\,ds,
 \end{multline*}
 \begin{multline*}
 I_\phi=\frac12\bbE\int_0^t\int_{\cO^2}\int_{\R^2}f_1(s,x,\xi)\bar f_2(s,y,\z) (\abf(\z)+\abf(\xi))\rho(x-y)\psi(\xi-\z)\,d\xi\,d\z\\ \cdot \nabla_x\phi((x+y)/2)\,dx\,ds,
 \end{multline*}
  and
\begin{multline*}
I_\psi=\frac12\int_{\cO^2}\rho(x-y)\phi((x+y)/2)\bbE\int_0^t\int_{\R^2}\psi(\xi-\z)\sum_{k\ge1}|g_k(x,\xi)-g_k(y,\z)|^2\\ d\nu_{s,x}^1\otimes\nu_{s,y}^2(\xi,\z)\,dx\,dy\,ds,
\end{multline*}
where $\nu^i_{t,x}=-\partial_\xi f_i  =\d_{u_i=\xi}$, $i=1,2$.
 \end{proposition}

 As a consequence of Proposition~\ref{P:4.1}, we have the Kato-Kruzhkov inequality. Again, the proof is similar to the proof of the comparison principle in \cite{DV}, therefore we omit it. 

 \begin{theorem} \label{T:4.1} Let $u_1, u_2$ be kinetic solutions of \eqref{e1.1}--\eqref{e1.3}. Then, for a.e.\ $0\le t\le T$ and $0\le \phi\in C_c^\infty(\cO)$, we have
 \begin{multline}\label{e4.2}
 \bbE\int_{\cO} \left(u_1(t,x)-u_2(t,x)\right)_+\phi(x)\,dx\\  \le \bbE\int_0^t\int_{\cO} \sgn(u_1-u_2)_+\left(\Abf(u_1)-\Abf(u_2)\right)\cdot\nabla\phi(x)\,dx\,ds
 \\+\int_{\cO}\left(u_{10}(x)-u_{20}(x)\right)_+\phi(x)\,dx.
 \end{multline}
 \end{theorem}
 
 \medskip
 
 As a direct consequence of Theorem~\ref{T:4.1} and Theorem~\ref{T:3.1}  we have the following.

 \begin{theorem}[Comparison Principle]\label{T:4.2}    Let $u_1, u_2$ be kinetic solutions of \eqref{e1.1}--\eqref{e1.3}. Then for a.e.\ $t\in (0,T)$,
 \begin{equation}\label{e4.3}
 \bbE\int_{\cO} (u_1(t,x)-u_2(t,x))_+\, dx\le \int_{\cO}(u_{10}(x)-u_{20})_+\,dx.
 \end{equation}

 \end{theorem}

 \begin{proof}  Let $\Psi:\po\cO\X[0,1]\to \bar \cO$ be a strongly regular deformation for $\po\cO$ and let $h:\bar\cO\to [0,1]$ be defined by $h(x)=s$, if $x\in\po\cO_s$, $h(x)=1$, if $x\in\cO\setminus\Psi(\po\cO\X [0,1])$, and $h(x)=0$, for $x\notin\cO$. Define
 \begin{equation} \label{comp}
 \varphi_\rho(x):= \min\{1,\, \frac1{\rho}h(x)\}.
 \end{equation}
 The inequality \eqref{e4.2} easily extends to $\phi$ Lipschitz vanishing on $\po\cO$. So we take $\phi=\varphi_\rho$ in \eqref{e4.2}. We then make $\rho\to0$ and observe that the first integral on the right-hand side of \eqref{e4.2} vanishes when $\rho\to0$
 because of the strong trace property in \eqref{e03.2}. Indeed, we see that
 $$
 \nabla\varphi_\rho(x)=\begin{cases} -\frac1{\rho} C(\Psi_{h(x)}(x))\nu(\Psi_{h(x)}(x)), &\text{for $x\in \Psi([0,\rho]\X\po\cO)$},\\ 0, & \text{otherwise},\end{cases}
 $$
 for a smooth function $C(y)$. But, from the regularity of the deformation, we deduce that $\nu(\Psi_s(x))\to \nu(x)$ in $L^1(\po\cO)$, as $s\to0$. Then,  Theorem~\ref{T:3.1} implies that the integral
\[
\bbE\int_0^t\int_{\cO} \sgn(u_1-u_2)_+\left(\Abf(u_1)-\Abf(u_2)\right)\cdot\nabla\varphi_\rho(x)\,dx\,ds
\]
vanishes as $\rho\to 0$, and since $\varphi_\rho\to 1$ a.e. as $\rho \to 0$, we conclude \eqref{e4.3}.
 \end{proof}

We remark that the a.s.\ continuity of the trajectories of a kinetic solution follows exactly as in \cite{DV}. In particular, in the statements of Theorem~\ref{T:4.1} and Theorem~\ref{T:4.2}, the conclusion holds a.s.\  for all $t\in[0,T]$.

 We conclude this section by establishing a maximum principle for the kinetic solution of \eqref{e1.1}--\eqref{e1.3}.

 \begin{theorem}[Maximum Principle] \label{T:4.3} Let $u$ be a kinetic solution of \eqref{e1.1}--\eqref{e1.3}, with $u_0$ satisfying \eqref{e1.5'}. Then, a.s., $a\le u(t,x)\le b$, a.e.\ $(t,x)\in (0,T)\X\cO$.
 \end{theorem}

 \begin{proof} It suffices to observe that under the hypothesis \eqref{e1.6} and the fact that $g_k(x,u)$ vanishes for $u\notin (-M,M)$, $k\in\N$, the functions $v_1\equiv a$ and $v_2\equiv b$ are kinetic solutions of \eqref{e1.1}--\eqref{e1.3}.
 Therefore, applying \eqref{e4.3} first with  $u_1=a$, $u_2=u$ and then with  $u_1=u$ and $u_2=b$, we get the desired result.
 \end{proof}

\section{Existence: The parabolic approximation}\label{S:4}

For the existence of a kinetic solution to problem \eqref{e1.1}--\eqref{e1.3} we will perform the following steps. First, we establish the existence of the parabolic approximation and its kinetic formulation.
Second, we prove a spatial regularity estimate for the parabolic approximation which is independent of the vanishing artificial  viscosity.  Third, using the regularity obtained in the second, we show that the sequence
of parabolic approximate  solutions is compact in $L_\loc^1$.

We consider the following parabolic approximation of problem \eqref{e1.1}--\eqref{e1.3},
\begin{align}
& du^\ve+\nabla\cdot \Abf^\ve(u^\ve)\,dt -\ve\Delta u^\ve\,dt= \Phi_\ve(u^\ve)\,dW(t), \quad t>0, \ x\in\cO,\label{e5.1}\\
&u^\ve(0,x)=u_0^\ve(x),\qquad x\in\cO,\label{e5.2}\\
& (\Abf(u^\ve(t,x))-\ve\nabla u^\ve(t,x))\cdot \nu(x)=0,\qquad t>0,\  x\in\po\cO, \label{e5.3}
\end{align}
where $u_0^\ve$ is a smooth approximation of $u_0$, $u_0^\ve\in L^\infty(C_c^\infty(\cO))$, $a\le u_0^\ve\le b$,   $\Phi_\ve$ is a suitable Lipschitz approximation of $\Phi$ satisfying \eqref{e1.4} and \eqref{e1.5}  uniformly, with
 $g_ k^\ve$ and $G^\ve$  as in the case $\ve=0$, $g_k^\ve$ smooth with compact support contained in $\cV\X(-M,M)$.  Moreover, $g_k^\ve\equiv 0$ for $k\ge 1/\ve$.
Finally, $\Abf^\ve\in C^2(\R;\R^d)$, $\Abf^\ve(u)=\Abf(u)$, for $u\in[a,b]$,  and, setting $\abf^\ve=(\Abf^\ve)'$, we assume that $\abf^\ve\in L^\infty(\R;\R^d)$. The justification for the latter assumption is the fact that, by  Theorem~\ref{T:4.3},
any solution of \eqref{e1.1}--\eqref{e1.3} takes values in the interval $[a,b]$ and so, since our goal is to use the solution of \eqref{e5.1}--\eqref{e5.3} as an approximation as $\ve\to0$ to a kinetic solution of \eqref{e1.1}--\eqref{e1.3},  we may modify $\Abf$ outside of $[a,b]$ as we wish. In particular, we may assume that $\supp \Abf^\ve \subseteq [a-1,b+1]$, so that $\Abf^\ve$ has a primitive which is bounded uniformly in $\ve$.

\begin{theorem}\label{T:5.1} There exists a unique solution of  \eqref{e5.1}--\eqref{e5.3}, $u^\ve\in C([0,T];  L^2(\Om\X \cO))\cap L^2(\Om\X[0,T]; H^1(\cO))$, for any $\ve>0$.
Moreover, $u^\ve$ satisfies the following energy estimate
\begin{equation}\label{Energy}
\bbE\sup_{0\le t\le T} \|u^\ve(t)\|_{L^2(\cO)}^2+ 2\ve \bbE \int_0^T\|\nabla u^\ve(s)\|_{L^2(\cO)}^2\,ds  \le C(T)(\|u^\ve_0\|_{L^2(\cO)}^2+1).
\end{equation}
\end{theorem}

The plan of the proof is to apply Banach's fixed point theorem.
 Let $\cE:=L^2(\Om;C([0,T]; L^2(\cO)))\cap L^2(\Om\X[0,T]; H^1(\cO))$.  Here, we consider $\Om\X[0,T]$ endowed with the $\s$-algebra of the predictable sets, that is, the $\s$-algebra generated by the sets of the form $\{0\}\X A_0$, $[s,t)\X A_s$, $A_0\in\F_0$, $A_s\in\F_s$.

 To begin with we endow $\cE$ with the following standard  norm
 \begin{equation}\label{norm1}
 \|v\|_{\cE}^2:= \Vert v\Vert_{L^2(\Om; C([0,T]; L^2(\mathcal{O})))}^2+\Vert \nabla v \Vert_{L^2(\Om\X[0,T]; H^1(\mathcal{O}))}^2.
 \end{equation}
Later on we will introduce another equivalent norm for $\cE$ for the purpose of proving the contraction property of the mapping $K$ defined subsequently.

Let us define
\begin{equation}\label{defK}
K(v)(t)= S(t)u_0-\int_0^t S(t-s)\nabla_x\cdot \Abf(v(s))\,ds +\int_0^tS(t-s)\Phi(v) dW(s) + w^v(t),
\end{equation}
where $S(t)$ is the semigroup generated by the problem
\begin{equation}\label{heateq}
\begin{cases} w_t-\ve \Delta w=0, \quad t>0, \ x\in\cO,\\
 w(0,x)=w_0(x),\quad x\in\cO,\\
\ve\nabla w(t,x)\cdot \nu(x)=0, \quad t>0,\ x\in\cO,
 \end{cases}
 \end{equation}
 and $w^v(t)$ is the solution of
 \begin{equation}\label{wv}
 \begin{cases} w_t^v-\ve \Delta w^v=0, \quad t>0, \ x\in\cO,\\
 w^v(0,x)=0,\quad x\in\cO,\\
\ve \nabla w^v\cdot\nu(x)=\Abf(v(t))\cdot \nu(x), \quad t>0,\ x\in\cO,
 \end{cases}
 \end{equation}
 and we have dropped  the $\ve$ for simplicity of notation.

The energy estimate for the heat equation with null Neumann condition gives
\begin{equation}\label{enheat}
\frac12\|S(t)w_0\|_2^2+\ve\int_0^t\|\nabla_x(S(s)w_0)\|_2^2\,ds= \frac12\|w_0\|_2^2.
\end{equation}

 \begin{lemma}\label{L:5.1}  If $v\in \cE$, then $K(v)\in \cE$. Moreover, if $v_k\to v$ in  $L^2(\Om\X[0,T] ;H^1(\cO))$, then $K(v_k)\to K(v)$ in $\cE$.

 \end{lemma}

 \begin{proof} We write $K(v)=S(t)u_0+K_1(v)+K_2(v)+w^v$, where
 \begin{align*}
  &K_1(v)(t)=-\int_0^tS(t-s)\nabla_x\cdot \Abf(v(s))\,ds\\
 &K_2(v)(t)=\int_0^t S(t-s)\Phi(v(s))\,dW(s).
 \end{align*}

 Concerning $S(t)u_0$,  for any $h\in L^2(\cO)$,  \eqref{enheat} trivially gives
 \begin{equation}\label{e5.6}
 \|S(t)h\|_2\le \|h\|_2.
 \end{equation}
 %so
 %\begin{equation}\label{e5.7}
%\Big \|\int_0^t S(t-s) h(s)\,ds\Big\|_2^2\le T\,\int_0^T|h(s)|^2\,ds.
%\end{equation}
We denote
$$
\nabla_xS(t)h:=\nabla_x(S(t)h).
$$
We have, also from \eqref{enheat},
 \begin{equation}\label{e5.7'}
 \int_0^{T}\| \nabla_xS(t) h\|_2^2\,dt \le \frac1{2\ve}\|h\|_2^2,
\end{equation}
Therefore, we have
\begin{equation}\label{e5.8'}
\begin{aligned}
\|S(t)u_0\|_2^2&\le \|u_0\|_2^2,\\
\int_0^{T}\|S(t)u_0\|_2^2\,dt&\le T\|u_0\|_2^2,\\
\int_0^{T}\|\nabla_xS(t)u_0\|_2^2\,dt&\le\frac1{2\ve} \|u_0\|_2^2.
\end{aligned}
\end{equation}
In sum, we have
\begin{equation}\label{S_0.1}
\|S(t)u_0\|_{\cE}\le C(T) \|u_0\|_{L^2(\cO)},
\end{equation}
where, throughout this proof, $C(T)$ is a positive constant depending only on $T$ and the data of the problem \eqref{e5.1}--\eqref{e5.3}.

Concerning $K_1(v)$, again directly from  \eqref{enheat}, we get
\begin{equation}\label{e5.8''}
\left\|\int_0^t S(t-s)\nabla_x\cdot \Abf(v(s))\,ds\right\|_2^2 \le C T \int_0^{T}\| \nabla_xv(t)\|_2^2\,dt,
\end{equation}
\begin{equation}\label{e5.9}
\int_0^{T}\left\|\int_0^t S(t-s)\nabla_x\cdot \Abf(v(s))\,ds\right\|_2^2\,dt \le C T^2\int_0^{T}\| \nabla_xv(t)\|_2^2\,dt,
\end{equation}
where, throughout this proof,   $C>0$ is a constant only depending on $\ve$ and the given functions on \eqref{e5.1}
whose value may change from one line to the next.

Observe now that Cauchy-Schwarz,  \eqref{enheat} and Fubini yield, for any $h\in L^2([0,T];L^2(\cO))$, 
\begin{equation}\label{e5.8}
\begin{aligned}
&\int_0^{T}\left\|\int_0^t \nabla_xS(t-s)h(s)\,ds\right\|_2^2 \,dt\le \int_0^{T} t\int_0^t\|\nabla_xS(t-s)h(s)\|_2^2\,ds\,dt \\
&\qquad\qquad\le T\int_0^{T}\int_0^t\|\nabla_xS(t-s)h(s)\|_2^2\,ds\,dt\\
&\qquad\qquad\le T\int_0^{T}\int_s^{T} \|\nabla_xS(t-s)h(s)\|_2^2\,dt\,ds\\
&\qquad\qquad= \frac1{2\ve}T\int_0^{T}\|h(s)\|_2^2\,ds.\\
\end{aligned}
\end{equation}
Therefore, we get
\begin{equation}\label{e5.10}
\int_0^{T}\left\|\int_0^t \nabla_xS(t-s)\nabla\cdot \Abf(v(s))\,ds\right\|_2^2\,dt \le C T\int_0^{T}\|\nabla_xv(t)\|_2^2\,dt.
\end{equation}
In sum,  we have
\begin{equation}\label{K_1.1}
\|K_1(v)\|_{\cE}\le C(T) \|v\|_{L^2(\Om\X[0,T];H^1(\cO))}.
\end{equation}
Similarly, we obtain for $v_1,v_2\in\cE$,
\begin{equation}\label{K_1.2}
\|K_1(v_1)-K_1(v_2)\|_{\cE}\le C(T) \|\Abf(v_1)-\Abf(v_2)\|_{L^2(\Om\X[0,T];H^1(\cO))}.
\end{equation}
 We remark that \eqref{K_1.2} implies that the mapping $v\mapsto  \Abf(v)$ is continuous from $L^2(\Om\X[0,T];H^1(\cO))$ to $L^2(\Om\X[0,T];H^1(\cO))$ as can be easily verified.

Concerning $K_2(v)$, we need to apply  the important maximal estimate for stochastic convolution (see \cite{Tu, Ko1, Ko2}; see also  \cite{DPZ, Brz,Ha2}). Using the mentioned inequality, we have
\begin{equation}\label{e5.10'}
\begin{aligned}
& \bbE\sup_{0\le t\le T}\biggl\|\int_0^t S(t-s)\Phi(v(s))\,dW(s)\biggr\|_2^2\\
&\qquad\le C\bbE \sum_{1\le k\le N} \int_0^T\| g_k(\cdot, v(s))\|_2^2\,ds\\
& \qquad \le C\bbE \int_0^{T}(1+\|v(s)\|_2^2)\,ds.\\
\end{aligned}
\end{equation}
On the other hand,
\begin{equation}\label{e5.11}
\begin{aligned}
&\int_0^{T}\bbE\biggl\|\int_0^t S(t-s)\Phi(v(s))\,dW(s)\biggr\|_2^2\,dt\\
&\qquad\le\bbE \int_0^{T}\sum_{1\le k\le N} \int_0^t\|S(t-s) g_k(\cdot, v(s))\|_2^2\,ds\,dt\\
& \qquad \le CT\,\bbE \int_0^{T}(1+\|v(s)\|_2^2)\,ds,\\
\\ \\
&\int_0^{T}\bbE\biggl\|\int_0^t \nabla_xS(t-s)\Phi(v(s))\,dW(s)\biggr\|_2^2\,dt \\
&\qquad\le C\bbE\int_0^{T} \int_0^t \sum_{1\le k\le N}\|\nabla_x S(t-s)g_k(\cdot,v(s))\|_2^2\,ds\,dt\\
&\qquad \le \frac{C}{2\varepsilon} \bbE \int_0^{T}\sum_{1\le k\le N} \|g_k(\cdot, v(t))\|_2^2\,dt\\
&\qquad \le \frac{C}{2\varepsilon}\bbE \int_0^{T} (  1+\|v(t)\|_2^2)\,dt.
\end{aligned}
\end{equation}
In particular, we also have
\begin{equation}\label{K_2.1}
\|K_2(v)\|_{\cE}\le C(T)(1+ \|v\|_{L^2(\Om\X[0,T];H^1(\cO))}),
\end{equation}
and for $v_1,v_2\in\cE$, observing that $\|g_k(\cdot,v_1(t))-g_k(\cdot,v_2(t))\|_2\le C\|v_1(t)-v_2(t)\|_2$, we get
\begin{equation}\label{K_2.2}
\|K_2(v_1)-K_2(v_2)\|_{\cE}\le C(T) \|v_1-v_2\|_{L^2(\Om\X[0,T];H^1(\cO))}.
\end{equation}

Finally, for $w^v$  we have the following. First, assume that $v\in C_c^\infty( (0,T)\X\bar \cO)$, in which case also $\Abf^\ve(v)\in C_c^\infty((0,T)\X\bar\cO)$ by the hypotheses on $\Abf^\ve(v)$.
In this case, the problem has a classical smooth solution.  Multiply the equation for $w^v$ by $w^v$, integrate in $\cO$ to get
\begin{align*}
&\frac12\frac{d}{ds}\|w^v(s)\|_2^2  +\ve\|\nabla w^v(s)\|_2^2=  \int_{\po\cO}w^v(s,\om) \Abf(v(s,\om))\cdot\nu\,d\H^{d-1}(\om)\\
&\qquad \le  C\|w^v(s)\|_{L^2(\partial\cO)}\|\Abf(v)(s)\|_{L^2(\partial\cO)}\\
 &\qquad \le C \|w^v(s)\|_{H^1}\|\Abf(v)(s)\|_{H^1}\\
 &\qquad \le \frac\ve2\|w^v(s)\|_{H^1}^2+C\|\Abf(v)(s)\|_{H^1}^2
 \end{align*}
 where we used the trace theorem  for the Sobolev space $H^1(\cO)$ and standard estimates. Integrating in $t$ we get
 \begin{equation*}
 \frac12 \|w^v(t)\|_2^2+\ve \int_0^t\|\nabla w^v(s)\|_2^2\,ds\le \frac\ve2\int_0^t\|w^v(s)\|_{H^1}^2\,ds+ C\int_0^t\|\Abf(v(s))\|_{H^1}^2\,ds.
 \end{equation*}
 So, using Gr\"onwall, we get
 \begin{equation}\label{ewv.1}
 \frac12 \|w^v(t)\|_2^2+\frac\ve2 \int_0^t\|\nabla w^v(s)\|_2^2\,ds\le C(T)\int_0^t\|\Abf(v(s))\|_{H^1}^2\,ds.
 \end{equation}
 Similarly, for $v^1,v^2\in C_c^\infty((0,T)\X\bar\cO)$, we get
  \begin{multline}\label{ewv.2}
 \frac12 \|w^{v_1}(t)-w^{v_2}(t)\|_2^2+\frac\ve2 \int_0^t\|\nabla w^{v_1}(s)-\nabla w^{v_2}(s)\|_2^2\,ds\\ \le C(T)\int_0^t\|v_1(s)-v_2(s)\|_{H^1}^2\,ds.
 \end{multline}
 For the latter inequality we used that $\|\Abf(v_1)-\Abf(v_2)\|_{L^2(\po\cO)}\le \Lip(\Abf)\|v_1-v_2\|_{L^2(\cO)}\le C\|v_1-v_2\|_{H^1(\cO)}$.
 Now, given $v\in L^2(\Om\X[0,T]; H^1(\cO))$, a.s.\ given $\om\in\Om$, we can approximate $v(\om)$ by a sequence $v_k\in C_c((0,T)\X\bar\cO)$, and from \eqref{ewv.1} we see that $w^{v_k}$ converges in $C([0,T];L^2(\cO))\cap L^2(0,T;H^1(\cO))$
 to a certain $w^v$. Also, given any $\varphi\in  H^1(\cO)$, for each $k\in\N$, we get
 \begin{multline*}
 \int_{\cO}w^{v_k}(t,x)\varphi(x)\,dx +\ve \int_0^t\int_{\cO}\nabla w^{v_k}(s,x)\cdot\nabla\varphi(x)\, dx\,ds\\=\int_0^t\int_{\po\cO} \varphi(y)\, \Abf(v_k)\cdot\nu(y)\,d\H^{d-1}(y)\,ds.
 \end{multline*}
 So, making $k\to\infty$, observing that $v_k\to v$ and $w^{v_k}\to w^v$   in $L^2((0,T)\X\po\cO)$  by the continuity of the trace operator from $L^2(0,T;H^1(\cO))$ to $L^2((0,T)\X\po\cO)$, we get  that
 $w^v$ satisfies \eqref{wv} in the following weak sense: for all $\varphi\in\ H^1(\cO)$ we have
 \begin{multline}\label{ewv.3}
  \int_{\cO}w^{v}(t,x)\varphi(x)\,dx +\ve \int_0^t\int_{\cO}\nabla w^{v}(s,x)\cdot\nabla\varphi(x)\, dx\,ds\\=\int_0^t\int_{\po\cO} \varphi(y)\, \Abf(v)\cdot\nu(y)\,d\H^{d-1}(y)\,ds.
 \end{multline}
By passing to the limit we also see that \eqref{ewv.1} and \eqref{ewv.2} are satisfied for any $v\in L^2(0,T; H^1(\cO))$. Thus, from \eqref{ewv.1} and \eqref{ewv.2} we get
\begin{equation}\label{Wv.1}
\|w^v\|_{\cE}\le C(T)\|v\|_{L^2(\Om\X[0,T]; H^1(\cO))},
\end{equation}
and
\begin{equation}\label{Wv.2}
\|w^{v_1}-w^{v_2}\|_{\cE}\le C(T)\|v_1-v_2\|_{L^2(\Om\X[0,T]; H^1(\cO))},
\end{equation}
for $v, v_1, v_2\in L^2(\Om\X[0,T];H^1(\cO))$.

Now, putting together the inequalities  \eqref{S_0.1}, \eqref{K_1.1}, \eqref{K_1.2}, \eqref{K_2.1}, \eqref{K_2.2},  \eqref{Wv.1}, \eqref{Wv.2}  the  proof of the lemma is finished.

\end{proof}

 Now we apply Propositions~\ref{reg1}--\ref{P:5.1} in Section~\ref{S:7} to analyze the map $K(v)$.

\begin{proposition}\label{P:5.2}
Assume $v \in L^\infty(\Omega; C_c^\infty((0,T)\times \overline{\mathcal{O}})$. We have the following:
\begin{enumerate}
\item[(i)] $K(v) \in \cE_*$, where
\begin{equation}\label{e5.5}
 \cE_*:=L^2(\Om; C([0,T]; L^2(\mathcal{O}))) \cap L^2(\Omega\X[0,T];H^2(\mathcal{O})).
\end{equation}
\item[(ii)]   $K(v)$ satisfies the following  initial-boundary value problem for  a stochastic equation with coefficients taking values in $L^2(\cO)$
\begin{equation}\label{eP5.2.0}
\begin{cases}
d K(v)(t)-\ve \Delta K(v)(t)\,dt= -\nabla\cdot \Abf(v(t))\,dt +\Phi(v(t))\,dW(t)\\
K(v)(0)=u_0,\\
\ve\po_{\nu} K(v)(t)\lfloor\po\cO=\Abf(v)(t)\cdot\nu.
\end{cases}
\end{equation}
\end{enumerate}
 More generally,  for all $v\in \cE$, given $\varphi\in C^\infty(\bar\cO)$, almost surely we have
\begin{multline}\label{eP5.2.1}
\int_{\cO}K(v)(t)\varphi(x)\,dx+\ve\int_0^t\int_{\cO}\nabla K(v)\cdot\nabla\varphi(x)\,dx\,dt =\int_{\cO}u_0(x)\varphi(x)\,dx \\+\int_0^t\int_{\cO} \Abf(v(s))\cdot\nabla\varphi(x)\,dx\,dt+ \int_0^t\int_{\cO}\varphi(x)\Phi(v(s))\,dx\, dW(s).
\end{multline}

\end{proposition}

\begin{proof}
 {} From the assumption on $v$ it follows from what was seen in the proof of Lemma~\ref{L:5.1} that $w^v\in \cE_*$.
 Also, from the smoothness of $u_0$ it follows immediately $S(t)u_0\in\cE_*$. Concerning $K_1(v)$ and $K_2(v)$ the fact that $K_1(v), K_2(v)\in \cE_*$ follows from Lemma~\ref{L:5.1} and Propositions~\ref{reg1}, \ref{reg2} applied to the operator
 defined in \eqref{defA}, taking into account Proposition~\ref{P:5.1}, which concludes the proof.

 Concerning the proof of \eqref{eP5.2.1}, first,  since $S(t)u_0$ is a classical solution of the  heat equation, we clearly have
 $$
 \int_{\cO}\varphi(x) S(t)u_0\,dx+\int_0^t\int_{\cO}\nabla S(s)u_0\cdot \nabla \varphi\, dx\,ds=\int_{\cO}\varphi(x) u_0(x)\,dx.
 $$
 Now,  by the definition of $S(t)$ we see that $K_1(v)$ is a solution of the problem
\begin{equation*}
\begin{cases}
\tilde w_t-\ve\Delta \tilde w= -\nabla_x\cdot \Abf(v(t)),\\
\ve\po_{\nu} \tilde w=0,\\
\tilde w(0,x)=0.
\end{cases}
\end{equation*}
Therefore, we have
\begin{multline*}
\int_{\cO} K_1(v(t))\varphi(x)\,dx+\ve\int_0^t \int_{\cO} \nabla K_1(v(s)\cdot \nabla \varphi(x)\,dx\,ds\\
=-\int_0^t\int_{\cO} \varphi\nabla \cdot \Abf(v(s))\,dx\,ds\\
=\int_0^t\int_{\cO}\nabla_x\varphi(x)\cdot \Abf(v(s))\,dx\,ds-\int_0^t\int_{\po\cO}\varphi(y) \Abf(v(s))\cdot\nu(y)\,d\H^{d-1}(y)\,ds.
\end{multline*}
As for $K_2(v)$ we have the following. First, we observe that $S(t-s)\Phi(v(s))$ solves the problem
\begin{equation*}
\begin{cases}
 z_t-\ve\Delta  z= 0,\\
\ve\po_{\nu} z=0,\\
 z(s,x)=\Phi(v(s)),
\end{cases}
\end{equation*}
and so
\begin{multline*}
\int_{\cO} \varphi(x) S(t-s)\Phi(v(s))\,dx+\ve\int_s^t\int_{\cO} \nabla\varphi(x)\cdot \nabla S(\tau-s)\Phi(v(s))\,dx\,d\tau\\=\int_{\cO}\varphi(x) \Phi(v(s))\,dx.
\end{multline*}
Integrating in $dW(s)$ from 0 to $t$, using the stochastic Fubini theorem (see, e.g., \cite{DPZ}, \cite{GaMan}), we obtain
\begin{multline*}
\int_{\cO} \varphi(x) K_2(v)(t)\,dx+\ve\int_0^t \int_s^t\int_{\cO} \nabla\varphi(x)\cdot \nabla S(\tau-s)\Phi(v(s))\,dx\,d\tau\,dW(s)\\= \int_0^t\int_{\cO}\varphi(x) \Phi(v(s))\,dx\,dW(s).
\end{multline*}
Now, again by the stochastic Fubini theorem, we have
\begin{multline*}
\int_0^t \int_s^t\int_{\cO} \nabla\varphi(x)\cdot \nabla S(\tau-s)\Phi(v(s))\,dx\,d\tau\,dW(s)\\
=\int_0^t \int_{\cO} \nabla \varphi \cdot \nabla\left( \int_0^\tau S(\tau-s) \Phi(v(s))\,dW(s) \right)\, dx\,d\tau\\
=\int_0^t\int_{\cO} \nabla\varphi(x)\cdot \nabla K_2(v(s))\,dx\,ds,
\end{multline*}
 and so we get
 \begin{multline*}
\int_{\cO} K_2(v(t))\varphi(x)\,dx+\ve\int_0^t \int_{\cO} \nabla K_2(v(s)\cdot \nabla \varphi(x)\,dx\,ds\\
\\= \int_0^t\int_{\cO}\varphi(x) \Phi(v(s))\,dx\,dW(s).
\end{multline*}
Now, putting together the identities obtained for $S(t)u_0$, $K_1(v)$,  $K_2(v)$ and recalling \eqref{ewv.3}, we obtain that $K(v)$ satisfies \eqref{eP5.2.1}.
Since, for $v\in L^\infty(\Om; C_c^\infty((0,T)\X\bar\cO)$, $K(v)\in\cE_*$ the above integral equation easily implies \eqref{eP5.2.0}.
Finally, given  $v\in\cE$, we may approximate it in $L^2(\Om\X[0,T]; H^1(\cO))$ by a sequence $\{v_k\}\subset L^\infty(\Om;C_c^\infty((0,T)\X\bar\cO))$, write the integral identity \eqref{eP5.2.1} with $v_k$ instead of $v$ and
pass to the limit in $L^2(\Om\X[0,T]; H^1(\cO))$ when $k\to\infty$. Using the continuity property for $K(v)$ in the statement of Lemma~\ref{L:5.1}, we obtain \eqref{eP5.2.1} for $v\in\cE$, which finishes the proof.

  \end{proof}

The following proposition is a decisive step in the proof of the contraction property of $K$ on $\cE$ endowed with a suitable norm.

\begin{proposition} \label{P:5.3}
Given $\psi \in C^1(\overline{\mathcal{O}})$, $\eta \in C^2(\R)$ with $\eta'' \in L^\infty$,  $v^{1}, v^{2} \in \cE$, it holds almost surely and for every $0\leq t \leq T$, the equalities
\begin{align}
    \int_{\mathcal{O}} &\eta (K (v^{1})(t,x)) \psi(x)\, dx = \int_{\mathcal{O}} \eta(u_0(x)) \psi(x)\, dx \nonumber\\
    &-\varepsilon \int_0^t\int_{\mathcal{O}} \eta'' (K(v^1)(s,x)) |\nabla K (v^{1})(s,x)|^2 \psi(x)\, dx\, ds  \nonumber\\
    &+\int_0^t \int_\mathcal{O} \eta''(K (v^1)(s,x)) \nabla K(v^{1})(s,x)\cdot \Abf(v^1(s,x)) \psi(x) dx ds \nonumber\\\
    &-\int_0^t \int_\mathcal{O} \eta'(K(v^1)(s,x)) ( \varepsilon \nabla K(v^{1})(s,x) - \Abf(v^1(s,x)) )\cdot \nabla \psi(x)\, dx\, ds \nonumber\\
    &+  \int_0^t \int_\mathcal{O} \eta'((K(v^{1})(s,x))\Phi(v^1(s))\psi(x) \,dx\,dW(s)\nonumber\\
    &+ \frac{1}{2}  \int_0^t \int_\mathcal{O} \eta''(K (v^{1})(s,x))G^2(x,v^1(s,x))\psi(x)\, dx \, ds \label{Ito1}
\end{align}
where $G(x,v)$ is as in \eqref{e1.4}, and
\begin{align}
    \int_{\mathcal{O}} &\eta\big(K(v^{1})(t,x) - K(v^{2})(t,x)\big) \psi(x) \,dx  \nonumber\\
    &=-\varepsilon \int_0^t\int_{\mathcal{O}} \eta''\big(K (v^{1})(s,x) - K(v^{2})(s,x)\big) \nonumber \\
    &\>\>\>\>\>\>\>\>\>\>\>\>\>\>\>\>\>\>\>\>\>\>\>\>\>\>\>\>  |\nabla K (v^{1})(s,x) - \nabla K(v^{2})(s,x) |^2 \psi(x) \,dx\, ds  \nonumber\\
    &+\int_0^t \int_\mathcal{O} \eta''\big(K(v^1)(s,x) - K(v^2)(s,x)\big)  \big(\nabla K( v^{1})(s,x) - \nabla K(v^{2})(s,x)\big)\nonumber
    \\ &\>\>\>\>\>\>\>\>\>\>\>\>\>\>\>\>\>\>\>\>\>\>\>\>\>\>\>\> \cdot \big(\Abf(v^1(s,x)) - \Abf(v^2(s,x)) \big) \psi(x) \,dx \,ds \nonumber\\
    &-\int_0^t \int_\mathcal{O} \eta'\big(K(v^1)(s,x) - K(v^2)(s,x)\big) \big( \varepsilon \nabla K( v^{1})(s,x) - \varepsilon \nabla K( v^{2})(s,x) \nonumber \\
    &\>\>\>\>\>\>\>\>\>\>\>\>\>\>\>\>\>\>\>\>\>\>\>\>\>\>\>\>-  \Abf(v^1(s,x)) + \Abf(v^2(s,x)) \big)\cdot \nabla \psi(x) \,dx\, ds \nonumber\\
    &+ \int_0^t \int_\mathcal{O} \eta'\big(K( v^{1})(s,x) - K(v^{2})(s,x)\big) (\Phi(v^1(s))-\Phi(v^2(s)))\,dx\,dW(s)\nonumber\\
    &+ \frac{1}{2} \int_0^t \int_\mathcal{O} \eta''\big(K(v^{1})(s,x) - K(v^{2})(s,x)\big) \nonumber
    \\ &\>\>\>\>\>\>\>\>\>\>\>\>\>\>\>\>\>\>\>\>\>\>\>\>\>\>\>\> \sum_{k=1}^\infty  \big|g_k(x, v^1(s,x)) - g_k(x,v^2(s,x))\big|^2\psi(x) \,dx\, ds. \label{Ito2}
\end{align}

\end{proposition}

\begin{proof}
 Let us assume initially  that $v^1$ and $v^2 \in L^\infty(\Omega; C_c^\infty((0,T)\times\overline{\mathcal{O}})$. Then, by the previous proposition, $u^1(t) = K(v^1)$ lies in $C([0,T];L^2(\Om\X\mathcal{O})) \cap L^2(\Om\X[0,T]; H^2(\mathcal{O}))$, so that one may write in $L^2(\mathcal{O})$ that almost surely and for every $0\leq t \leq T$,
\begin{multline}
  u^1(t,x) = u_0(x) + \varepsilon\int_0^t \Delta u^1(s,x) ds - \int_0^t \operatorname{div}_x \Abf(v^1(s,x))ds \\+ \int_0^t \Phi(x, v^1(s,x))dW(s) \label{equ1}
\end{multline}
with
\begin{equation}
\varepsilon\frac{\partial u^1 }{\partial \nu}(t,x) = \Abf(v^1(t,x)) \text{ in the sense of traces in } \partial \mathcal{O}. \label{equ2}
\end{equation}

Hence, by  It\^ o formula (see, e.g., \cite{DPZ}), almost surely and for all $0 \leq t \leq T$, we have
\begin{align}
  \eta(u^1(t,x)) &= \eta(u_0(x)) + \varepsilon\int_0^t \eta'(u^1(s,x)) \Delta u^1(s,x) \,ds \nonumber \\
  &- \int_0^t \eta'(u^1(s,x)) \nabla_x \cdot \Abf(v^1(s,x))\, ds  \nonumber \\
  &+\int_0^t\psi(x) \eta'(u^1(s,x)) \Phi(v^1(s,x))\, dW(s)  \nonumber \\
  &+ \frac{1}{2} \int_0^t \psi(x)\eta''(u^1(s,x)) G^2(x,v^1(s,x))\, ds. \label{equ3}
\end{align}
We now multiply  equation \eqref{equ3}  by $\psi \in C^1 (\bar\cO)$, integrate in $x \in \mathcal{O}$, use integration by parts,
to get
\begin{align}
  \int_{\cO}\psi(x)\eta(u^1(t,x))\,dx &= \int_{\cO}\psi(x) \eta(u_0(x))\,dx \nonumber \\
  &- \varepsilon \int_0^t\int_{\cO} \psi(x) \eta''(u^1(s,x)) |\nabla u^1(s,x)|^2 \,dx\,ds \nonumber \\
  &-\ve\int_0^t\int_{\cO}\eta'(u^1(s,x))\nabla\psi(x)\cdot\nabla u^1(s,x)\,dx\,ds \nonumber\\
  &+\ve \int_0^t\int_{\po\cO}\psi(y) \eta'(u^1(s,y)) \po_{\nu} u^1(s,y)\,d\H^{d-1}(y)\,ds\nonumber\\
  &+\int_0^t\int_{\cO}\psi(x) \eta''(u^1(s,x)) \nabla u^1(s,x)\cdot \Abf(v^1(s,x))\,dx\, ds  \nonumber \\
  &+\int_0^t\int_{\cO}\eta'(u^1(s,x))\nabla\psi(x)\cdot \Abf(v^1(s,x))\,dx\,ds\nonumber\\
  &- \int_0^t\int_{\po\cO}\psi(y) \eta'(u^1(s,y)) \Abf(v^1(s,y))\cdot\nu(y)\,d\H^{d-1}(y)\,ds\nonumber\\
  &+ \int_0^t\int_{\cO}\psi(x) \eta'(u^1(s,x)) \Phi(v^1(s,x))\,dx\,dW(s)  \nonumber \\
  &+ \frac{1}{2} \int_0^t\int_{\cO}\psi(x) \eta''(u^1(s,x)) G^2(x,v^1(s,x))\,dx\, ds. \label{equ3'}
\end{align}
Then we use \eqref{equ2}, from which the  fourth and seventh term on the right-hand side of \eqref{equ3'} cancel each other,  to finally obtain \eqref{Ito1}.

As for \eqref{Ito2}, it may be obtained by a totally similar argument.

In order to obtain both identities  for $v^1,v^2 \in \cE$, we approximate them in $L^2(\Om\X[0,T]; H^1(\cO))$ by sequences $v_k^1,v_k^2\in L^\infty(\Om;C_c^\infty((0,T)\X\bar \cO))$,  and pass to the limit in the identities \eqref{Ito1}, \eqref{Ito2}  for $v_k^1,v_k^2$ using the continuity property of $K(v)$ stated in Lemma~\ref{L:5.1}. Since $\varphi'' \in (C\cap L^\infty)(-\infty, \infty)$, all terms on the equations are preserved on the limit, so \eqref{Ito1} and \eqref{Ito2} are proven.

\end{proof}

Let us now define a new norm for $\cE$, equivalent to norm $\|\cdot\|_{\cE}$ defined  in \eqref{norm1}. We set
\begin{equation}\label{norm2}
\Vert u \Vert_{*\cE}^2 =   \sup_{0\leq t \leq T} e^{-C_* t/\alpha}\bbE \Bigg\{  \sup_{0\leq s \leq t} \|u(s)\|_{L^2(\mathcal{O})}^2 + \varepsilon \int_0^t \|\nabla u(s)\|_{L^2(\mathcal{O})}^2 \, ds  \Bigg\},
\end{equation}
where $C_*>0$ and $0<\alpha<1$ will  be  chosen later.

As a corollary of Proposition~\ref{P:5.3}, let us prove that $K(v)$ is a contraction on $\cE$ endowed with the $\|\cdot\|_{*\cE}$, with $C_*$ suitably chosen.

\begin{proposition} \label{P:5.4} If $\cE$ is endowed with the norm $\|\cdot\|_{*\cE}$, with $C_*$ suitably chosen, then  $K:\cE\to\cE$ is a contraction.
\end{proposition}
\begin{proof}
Given $v^1$ and $v^2$ in $\cE$, let us apply \eqref{Ito2} with  $\psi(x) \equiv 1$ and $\eta(s) = s^2/2$ to obtain
\begin{align*}
 & \frac{1}2 \bbE  \sup_{0\le t\le T}\|K(v^1)(t) - K(v^2)(t)\|_{L^2(\mathcal{O})}^2 + \varepsilon \bbE \int_0^T \|\nabla K(v^1)(s) - \nabla K(v^2)(s)\|_{L^2(\mathcal{O})}^2\,ds \\
  &\leq \bbE \int_0^T \int_{\cO} |\nabla K(v^1)(s,x) - \nabla K(v^2)(s,x)| |\Abf(v^1(s,x)) - \Abf(v^2(s,x))|\, dx\,ds \\
  &+ \bbE \sup_{0\le t\le T}\left|\sum_{k\ge1} \int_0^t \int_{\cO} (K(v^1)(s,x) - K(v^2)(s,x)) ( g_k(v^1(s,x)) - g_k(v^2(s,x)))\, dx\,d\b_k(s)\right|\\
  &+ \frac{1}{2} \sum_{k\ge1} \bbE \int_0^T \int_{\cO}\big|g_k\big(x, v^1(s,x)\big) - g_k\big(x, v^2(s,x)\big) \big|^2 \,dx\,ds\\
  &\le \frac\ve2 \bbE \int_0^T \|\nabla K(v^1)(s) - \nabla K(v^2)(s)\|_{L^2(\mathcal{O})}^2\,ds+C\bbE\int_0^T\|v_1(s)-v_2(s)\|_{L^2(\cO)}^2\,ds\\
  &+C\bbE\left(\int_0^T\sum_{k\ge1} \left|\int_{\cO}(K(v^1)(s,x) - K(v^2)(s,x)) ( g_k(v^1(s,x)) - g_k(v^2(s,x)))\, dx\right|^2\,ds\right)^{1/2}
  \end{align*}
  where we have used the Burkholder-Davis-Gundy inequality (see, e.g., \cite{O}). So that the right-hand side of the above inequality may be estimated as
  \begin{align*}
   &\le \frac\ve2\bbE \int_0^T \|\nabla K(v^1)(s) - \nabla K(v^2)(s)\|_{L^2(\mathcal{O})}^2\,ds+C\bbE\int_0^T\|v_1(s)-v_2(s)\|_{L^2(\cO)}^2\,ds\\
  &+C\bbE\left(\int_0^T\|K(v^1)(s) - K(v^2)(s)\|_{L^2}^2 \sum_{k\ge 1}\| g_k(v^1(s)) - g_k(v^2(s)\|_{L^2}^2\,ds\right)^{1/2}\\
   &\le \frac\ve2 \bbE \int_0^T \|\nabla K(v^1)(s) - \nabla K(v^2)(s)\|_{L^2(\mathcal{O})}^2\,ds+C\bbE\int_0^T\|v_1(s)-v_2(s)\|_{L^2(\cO)}^2\,ds\\
  &+C\bbE\left(\sup_{0\le t\le T}\|K(v^1)(t) - K(v^2)(t)\|_{L^2}^2\right)^{1/2} \left(\int_0^T\| v^1(s) - v^2(s)\|_{L^2}^2\,ds\right)^{1/2}\\
    &\le \frac\ve2 \bbE \int_0^T \|\nabla K(v^1)(s) - \nabla K(v^2)(s)\|_{L^2(\mathcal{O})}^2\,ds+C\bbE\int_0^T\|v_1(s)-v_2(s)\|_{L^2(\cO)}^2\,ds\\
  &+\frac14\bbE\sup_{0\le t\le T}\|K(v^1)(t) - K(v^2)(t)\|_{L^2}^2 ,
\end{align*}
from which it follows
\begin{align}
  \mathbb{\bbE} \sup_{0\le t\le T} \|K(v^1)(t) &- K(v^2)(t)\|_{L^2(\mathcal{O})}^2 + \varepsilon \mathbb{E} \int_0^T \|\nabla K(v^1)(s) - \nabla K(v^2)(s)\|_{L^2(\mathcal{O})}^2 \nonumber \\
  &\leq C_* \mathbb{E} \int_0^T \|v^1(s) - v^2(s)\|_{L^2(\mathcal{O})}^2ds, \label{contract}
\end{align}
for some  constant $C_* > 0$, depending only on the data of the problem \eqref{e5.1}--\eqref{e5.3}. Now, take \eqref{contract} with $t$ instead of $T$, multiply both sides of it  by
$e^{-C_*t/\a}$, take the $\sup_{0\le t\le T}$ majorizing the resulting right-hand side by
\begin{multline*}
 C_*\sup_{0\le t\le T} e^{-C_*t/\a} \int_0^t e^{C_*s/\a} e^{-C_*s/\a } \mathbb{E}\sup_{0\le \tau\le s} \|v^1(\tau) - v^2(\tau)\|_{L^2(\mathcal{O})}^2\,ds
\\ \le  C_*\sup_{0\le t \le T} e^{-C_*t/\a} \|v^1-v^2\|_{*\cE}^2 \int_0^t e^{C_*s/\a}\,ds
 \le \a\|v^1-v^2\|_{*\cE}^2,
 \end{multline*}
 and then taking the $\sup_{0\le t\le T}$ on the resulting left-hand side we deduce that
 $$
 \| K(v^1)-K(v^2)\|_{*\cE}\le \a^{1/2}  \|v^1-v^2\|_{*\cE}.
 $$
 Since $0<\a<1$ we have the desired conclusion.
\end{proof}

\begin{proof}[Conclusion of the proof of Theorem~\ref{T:5.1}]  Proposition~\ref{P:5.4} implies the existence of a unique fixed point  $u^\ve$ for the operator $K:\cE\to\cE$. In particular,
from Proposition~\ref{P:5.2},  $u^\ve$ satisfies almost surely, for all $\varphi\in C^\infty(\bar \cO)$,
\begin{multline}\label{eP5.2.1'}
\int_{\cO}u^\ve(t)\varphi(x)\,dx+\ve\int_0^t\int_{\cO}\nabla u^\ve(s,x)\cdot\nabla\varphi(x)\,dx\,ds =\int_{\cO}u_0(x)\varphi(x)\,dx \\+\int_0^t\int_{\cO} \Abf(u^\ve(s))\cdot\nabla\varphi(x)\,dx\,ds+
\int_0^t\int_{\cO}\varphi(x)\Phi(u^\ve(s))\,dx\, dW(s).
\end{multline}
This means that  $u^\ve$ is a solution to the initial-boundary value problem \eqref{e5.1}--\eqref{e5.3}.
Now, from \eqref{eP5.2.1'}, it is easy to deduce that, for $\phi\in C^\infty([0,T]\X\bar \cO)$, we  have
\begin{multline}\label{eP5.2.1''}
\int_{\cO}u^\ve(t)\phi(t,x)\,dx- \int_0^t\int_{\cO} u^\ve(s,x)\phi_s(s,x)\,dx\,ds \\+ \ve \int_0^t\int_{\cO}\nabla u^\ve(s,x)\cdot\nabla\phi(s,x)\,dx\,ds
\\=\int_{\cO}u_0(x)\phi(0,x)\,dx +\int_0^t\int_{\cO} \Abf(u^\ve(s))\cdot\nabla\phi(s,x)\,dx\,ds\\+
\int_0^t\int_{\cO}\phi(s,x)\Phi(u^\ve(s))\,dx\, dW(s),
\end{multline}
which is another equivalent way to formulate the fact that $u(t,x)$ is a solution of \eqref{e5.1}--\eqref{e5.3}.

Now,  suppose $\bar u\in \cE$ is another solution of \eqref{e5.1}--\eqref{e5.3}, that is, if \eqref{eP5.2.1''} is satisfied with $\bar u$ instead of $u^\e$.  Then,  for a given  $t\in(0,T]$,  we  take in \eqref{eP5.2.1''}, with $\bar u$ instead of $u^\ve$,
$\phi(s,x)=S(t-s)\varphi(x)$, with $\varphi\in C_c^\infty(\cO)$,   and use the symmetry of $S(t-s)$ as an operator on $L^2(\cO)$, to get
\begin{multline}\label{eP5.2.1'''}
\int_{\cO} \bar u(t)\varphi(x)\,dx- \int_0^t\int_{\cO} \bar u(s,x) \po_sS(t-s)\varphi(x)\,dx\,ds \\- \ve \int_0^t\int_{\cO} \bar u(s,x)\cdot\Delta S(t-s)\varphi(x)\,dx\,ds
\\=\int_{\cO}S(t)u_0(x)\varphi(x)\,dx -\int_0^t\int_{\cO} S(t-s)\nabla\cdot \Abf(\bar u(s)) \varphi(x)\,dx\,ds\\
+\int_{\po\cO}S(t-s)\varphi(y)  \Abf(\bar u(s,y))\cdot\nu(y)\,d\H^{d-1}(y)+ \int_0^t\int_{\cO} S(t-s) \Phi(\bar u(s)) \varphi(x)\,dx\, dW(s).
\end{multline}
Thus, using the fact that $\po_t S(t-s)\varphi=\ve \Delta S(t-s)\varphi$,  we deduce
\begin{multline}\label{eP5.2.1iv}
\int_{\cO} \bar u(t)\varphi(x)\,dx =\int_{\cO}S(t)u_0(x)\varphi(x)\,dx\\ -\int_0^t\int_{\cO} S(t-s)\nabla\cdot \Abf(\bar u(s)) \varphi(x)\,dx\,ds
\\+\int_{\po\cO}S(t-s)\varphi(y)  \Abf(\bar u(s,y))\cdot\nu(y)\,dS(y)+ \int_0^t\int_{\cO} S(t-s) \Phi(\bar u(s)) \varphi(x)\,dx\, dW(s).
\end{multline}
Now, from \eqref{ewv.3}, similarly, we deduce that  $w^{\bar u}$  satisfies
\begin{multline}\label{eP5.2.1v}
\int_{\cO}w^{\bar u}(t)\phi(t,x)\,dx- \int_0^t\int_{\cO} w^{\bar u}(s,x)\phi_s(s,x)\,dx\,ds \\+ \ve \int_0^t\int_{\cO}\nabla w^{\bar u}(s,x)\cdot\nabla\phi(s,x)\,dx\,ds
\\=\int_0^t\int_{\po\cO} \phi(s,y) \Abf(\bar u(s,y))\cdot\nu(y)\,d\H^{d-1}(y)\,ds
\end{multline}
Again taking $\phi(s,x)=S(t-s)\varphi(x)$ and using that $\po_t S(t-s)\varphi=\ve \Delta S(t-s)\varphi$ we get
\begin{equation}\label{eP5.2.1vi}
\int_{\cO}w^{\bar u}(t)\varphi(x)\,dx =\int_0^t\int_{\po\cO} S(t-s)\varphi(y) \Abf(\bar u(s,y))\cdot\nu(y)\,d\H^{d-1}(y)\,ds.
\end{equation}
Then, using \eqref{eP5.2.1vi} into \eqref{eP5.2.1iv}
\begin{multline}\label{eP5.2.1vii}
\int_{\cO} \bar u(t)\varphi(x)\,dx =\int_{\cO}S(t)u_0(x)\varphi(x)\,dx \\
-\int_0^t\int_{\cO} S(t-s)\nabla\cdot \Abf(\bar u(s)) \varphi(x)\,dx\,ds\\
+ \int_0^t\int_{\cO} S(t-s) \Phi(\bar u(s)) \varphi(x)\,dx\, dW(s) +\int_{\cO}w^{\bar u}(t)\varphi(x)\,dx .
\end{multline}
Since $\varphi\in C_c^\infty(\cO)$ is arbitrary we conclude that $\bar u$ satisfies $K(\bar u)=\bar u$, that is, $\bar u$ is also a fixed point of $K$ and so, by the uniqueness of the fixed point in Banach's theorem, we have $\bar u= u^\ve$.

Finally, regarding the energy estimate \eqref{Energy}, from identity \eqref{Ito1} of Proposition~\ref{P:5.3} applied to the fixed point $u^\ve$ with $\eta(s)=s^2/2$ and proceeding as in the proof of Proposition~\ref{P:5.4} we have that
\begin{align*}
 & \frac{1}2 \bbE  \sup_{0\le s\le t}\|u^\ve(t)\|_{L^2(\mathcal{O})}^2 + \varepsilon \bbE \int_0^t \|\nabla u^\ve(s)\|_{L^2(\mathcal{O})}^2\,ds \\
 &\le \frac12\|u_0\|_{L^2(\cO)}^2+ \bbE \left|\int_0^t \int_\Omega \Abf(u^\ve(s,x))\cdot\nabla u^\ve(s,x) dx\, ds\right| \\
 &\qquad\qquad+\frac14\bbE\sup_{0\le s\le t}\|u^\ve(t)\|_{L^2(\cO)}^2 +C\bbE\int_0^t\|u^\ve(s)\|_{L^2(\cO)}^2\,ds\\
 &\leq \frac12\|u_0\|_{L^2(\cO)}^2+ C \max_{\l\in\R} |\tilde{\Abf}(\l)|  \\
 &\qquad\qquad+\frac14\bbE\sup_{0\le s\le t}\|u^\ve(s)\|_{L^2(\cO)}^2 + C\int_0^t \bbE\sup_{0\le r\le s}\|u^\ve(r)\|_{L^2(\cO)}^2\,ds,\\
\end{align*}
where $\tilde{\Abf}'(u)=\Abf(u)$. We recall that by our assumptions on the approximate flux function $\tilde\Abf$ is a bounded function. Then, using Gronwall's inequality, we obtain \eqref{Energy}.
Let us point out that the uniform boundedness of $\tilde{\Abf}$ is not essential as $u^\ve$ satisfies a maximum principle (see Theorem \ref{T:4.3'} below).
\end{proof}

For future reference we state here the following direct consequence of Proposition~\ref{P:5.3}.

\begin{lemma}[Entropy identity]\label{L:5.4}  Let $u^\ve\in\cE$ be the solution of \eqref{e5.1}--\eqref{e5.3}. For all $\eta\in C^2(\R)$,  for all $ \psi\in C_c^\infty(\cO)$, for all $0\le s\le t\le T$,
\begin{multline}\label{eL5.4}
\la \eta(u^\ve(t)),\psi\ra-\la \eta(u(s)),\psi\ra= -\ve\int_s^t\la \eta''(u^\ve(r))|\nabla u^\ve(r)|^2,\psi\ra+\int_s^t\la q(u^\ve(r)),\nabla\psi\ra dr\\
-\ve\int_s^t \la \nabla\eta(u^\ve(r)),\nabla \psi\ra\, dr +\sum_{k\ge 1}\int_s^t\la g_k^\ve(\cdot,u^\ve(r))\eta'(u^\ve(r)), \psi \ra\,d\b_k(r)\\
+\frac12\int_s^t \la {G^\ve}^2(\cdot,u^\ve(r))\eta''(u^\ve),\psi\ra\,dr,
\end{multline}
a.s., where $q(u)=\int_0^u a(\xi) \eta' (\xi)\,d\xi$.  Moreover, $u$ for $\phi\in C_c^\infty([0,T)\X\R^d)$, we  have
\begin{multline}\label{eL5.4.2}
 \int_0^T\int_{\cO} u^\ve(t,x)\phi_t(t,x)\,dx\,dt + \int_0^T\int_{\cO} \Abf(u^\ve(t))\cdot\nabla\phi(t,x)\,dx\,dt\\
+\int_0^T\int_{\cO}\phi(t,x)\Phi(u^\ve(t))\,dx\, dW(t) = \ve \int_0^T\int_{\cO}\nabla u^\ve(t,x)\cdot\nabla\phi(t,x)\,dx\,dt.
\end{multline}
\end{lemma}

\begin{proof} Relation \eqref{eL5.4} follows immediately from \eqref{Ito1}, using integration by parts, since $u^\ve$ is the fixed point of $K$.
As to relation \eqref{eL5.4.2}, it follows from \eqref{eP5.2.1''} by taking a test function in $\phi\in C_c^\infty([0,T)\X\R^d)$ and evaluating it at $t=T$.

\end{proof}

We close this section with the following maximum principle for the parabolic approximation.

 \begin{theorem}[Maximum Principle for the parabolic approximation] \label{T:4.3'} Let $u^\ve$ be the solution of \eqref{e5.1}--\eqref{e5.3}. Then, a.s., $a\le u^\ve (t,x)\le b$, a.e.\ $(t,x)\in (0,T)\X\cO$.
 \end{theorem}

\begin{proof}
We take   in \eqref{Ito1}  $\psi\equiv 1$ and $\eta(u^\ve)=\frac{1}{2}[u-b]_{\d,+}^2$,  where $[u-b]_{\d,+}$ is a  $C^2$ convex approximation of $[u-b]_+$ , the latter being the positive part of
$u-b$, such that $u\mapsto [u-b]_{\d,+}'$ is monotone nondecreasing, $[u-b]_{\d,+}'=1$, for $u>b+\d$, and $[u-b]_{\d,+}'=0$, for $u\le b$. Then, after sending $\d\to0$,  we obtain a.s.
\begin{align}
    \frac{1}{2}\int_{\mathcal{O}} &[u^\ve(t,x)-b]_+^2\, dx = \nonumber\\
    &-\varepsilon \int_0^t\int_{\mathcal{O}} {\bf1}_{u^\ve(s,x)>b} |\nabla u^\ve(s,x)|^2 \, dx\, ds  \nonumber\\
    &+\int_0^t \int_\mathcal{O} {\bf1}_{u^\ve(s,x)>b} \nabla u^\ve(s,x)\cdot \Abf(u^\ve(s,x)) dx ds \nonumber\\\
    &+  \int_0^t \int_\mathcal{O} [u(s,x)-b]_+\Phi(u^\ve(s)) \,dx\,dW(s)\nonumber\\
    &+ \frac{1}{2}  \int_0^t \int_\mathcal{O} {\bf1}_{u^\ve(s,x)>b}G^2(x,u^\ve(s,x))\, dx \, ds. \label{Ito1_+}
\end{align}

Now, by virtue of \eqref{e1.6}, Young's inequality with $\ve$ yields
\begin{multline*}
\int_0^t \int_\mathcal{O} {\bf1}_{u^\ve(s,x)>b} \nabla u^\ve(s,x)\cdot \Abf(u^\ve(s,x))\, dx\, ds \\
\leq \frac{\ve}{2}\int_0^t\int_{\mathcal{O}} {\bf1}_{u^\ve(s,x)>b} |\nabla u^\ve(s,x)|^2 \, dx\, ds + C_\ve \Lip(\Abf) \int_0^t\int_{\mathcal{O}} [u^\ve(t,x)-b]_+^2\, dx\, ds.
\end{multline*}

On the other hand, by assumption, $g_k(\cdot,\xi)=0$ for any $\xi>b$ so that the last two integrals on the right hand side of \eqref{Ito1_+} are equal to zero.

Thus, by Gronwall's inequality we conclude that  a.s.
\[
\frac{1}{2}\int_{\mathcal{O}} [u^\ve(t,x)-b]_+^2\, dx = 0,
\]
for a.e. $t\in(0,T)$.

Similarly, if $[u-a]_-$ denotes the negative part of $u-a$, in the same way we can also prove that a.s.
\[
\frac{1}{2}\int_{\mathcal{O}} [u^\ve(t,x)-a]_-^2\, dx = 0,
\]
for a.e. $t\in(0,T)$, which implies the result.
\end{proof}

\section{Existence: The vanishing viscosity limit}\label{S:5}

In this section we prove the convergence of the parabolic approximation \eqref{e5.1}--\eqref{e5.3} when $\ve\to0$ to the unique solution of \eqref{e1.1}--\eqref{e1.3}.

  \subsection{Kinetic formulation  for the parabolic approximation}\label{SS:5.1}

  The following proposition is essentially established in \cite{DV} in the periodic case and it can be proved as stated here in exactly the same way.

  \begin{proposition}\label{P:5.2.1} Let $u_0^\ve\in C_c^\infty(\cO)$ and let $u^\ve$ be the solution of \eqref{e5.1}--\eqref{e5.3}. Then $f^\ve={\bf1}_{u^\ve>\xi}$ satisfies: for all  $\varphi\in C_c^1(\cO\X [0,T)\X \R)$,
  \begin{multline}\label{eP521i}
  \int_0^T \la f^\ve(t),\po_t\varphi(t)\ra\,dt + \la f_0,\varphi(0)\ra+\int_0^T\la f^\ve(t), \abf(\xi)\cdot\nabla \varphi(t)-\ve \Delta\varphi(t)\ra\, dt\\
  =-\sum_{k\ge1}\int_0^T\int_{\cO}\int_{\R}g_k^\ve(x,\xi)\varphi(t,x,\xi)\,d\nu_{t,x}^\ve(\xi)\,dx\,d\b_k(t)\\
 -\frac12\int_0^T\int_{\cO}\int_{\R} \po_\xi\varphi(t,x,\xi) G_\ve^2(x,\xi)\,d\nu_{t,x}(\xi)^\ve\,dx\,dt+ m^\ve(\po_\xi\varphi),
 \end{multline}
 a.s., where $f_0(\xi)={\bf1}_{u_0>\xi}$, $\nu_{t,x}^\ve=\d_{u(t,x)=\xi}$, and, for $\phi\in C_b(\bar \cO\X[0,T]\X\R)$,
 \begin{equation}\label{e5.3.0}
 m^\ve(\phi)= \int_{\bar\cO\X[0,T]} \phi(t,x,u^\ve(t,x))\,\ve |\nabla u^\ve|^2\,dx\,dt,
 \end{equation}
 so $m^\ve=\ve |\nabla u^\ve|^2\d_{u^\ve=\xi}$.
 \end{proposition}

 %Once we have the kinetic formulation for the solution of \eqref{e5.1}--\eqref{e5.3} in Proposition~\ref{P:5.2.1} we can obtain the analogues of Proposition~\ref{P:4.1}  and Theorems~\ref{T:4.1}--\ref{T:4.3}, which are proven following the same lines as the proofs of these results. In particular, just for the record, we have the following maximum principle.

 \subsection{Local uniform space regularity of the parabolic approximation} \label{SS:5.2}

 By Proposition~\ref{P:5.2.1} we have that $\chi^\ve:={\bf1}_{u>\xi}-{\bf 1}_{0>\xi}$ is a weak solution to the parabolic stochastic equation
 \begin{equation}\label{e5.3.1}
 \po_t\chi^\ve+\abf(\xi)\cdot\nabla \chi^\ve-\ve\Delta \chi^\ve=\po_{\xi}q-\sum_{k=1}^\infty (\po_\xi\chi)g_k\dot{\b_k} +\sum_{k=1}^\infty \d_0 g_k\dot{\b_k},
 \end{equation}
  where $q=m^\ve-\frac12 G^2\d_{u=\xi}$ and $m^\ve$ is given by \eqref{e5.3.0}.

 Next, we state a local version of the corollary~3.3 in \cite{GH}. For that we first fix some $\vartheta\in C_0^\infty(\mathbb{R})$ nonnegative such that $\vartheta(\xi)=1$ for $a\leq \xi\leq b$, and note that by  the maximum principle from theorem~\ref{T:4.3'}, we have, in the sense of distributions, that
\[
\nabla\chi^\ve = \nabla(\vartheta(\xi)\chi^\ve) = \vartheta(\xi)\delta_{u^\ve=\xi}\nabla u^\ve=-\partial_\xi[{\bf1}_{u^\ve>\xi}\nabla u^\ve \vartheta(\xi)]+\vartheta'(\xi){\bf1}_{u^\ve>\xi}\nabla u^\ve.
\]
 
Then,  multiplying \eqref{e5.3.1} by $\varphi\in C_c^\infty([0,T)\times\cO)$ we get
 \begin{multline}\label{e5.3.2}
   \po_t (\varphi\chi^\ve)+\abf(\xi)\cdot\nabla (\varphi\chi^\ve)-\ve\Delta (\varphi\chi^\ve)\\
   =\po_{\xi}q^\varphi-\sum_{k=1}^\infty  (\po_\xi (\varphi\chi^\ve))g_k\dot{\b_k} +\sum_{k=1}^\infty \varphi \d_0 g_k\dot{\b_k} \\+ \chi^\ve(\po_t\varphi + a^\ve(\xi)\cdot\nabla\varphi -\ve\Delta\varphi)-2\ve\vartheta'(\xi){\bf1}_{u^\ve>\xi}\nabla u^\ve\cdot\nabla\varphi,
 \end{multline}
  where $q^\varphi=\varphi m^\ve- \varphi\frac12 G^2\d_{u=\xi} +2\ve\nabla\varphi\cdot \nabla u^\ve \vartheta(\xi){\bf1}_{u^\ve>\xi}$.  We observe that $q^\varphi$ is also a.s.\  a finite measure on $[0,T]\X\cO\X\R$ with total variation uniformly bounded with respect to $\ve$, by the energy estimate \eqref{Energy} and the maximum principle from Theorem \ref{T:4.3'}.

  We first remark that the condition \eqref{e1.7} implies the non-degeneracy condition of \cite{GH}, namely, for some $\a\in(0,1)$,
 \begin{align}
& \om_{\cL}(J,\d) \lesssim \left(\frac{\d}{J^\b}\right)^\a\qquad \forall \d>0,\, \forall J\gtrsim 1, \label{e1.7'} \\
&\sup_{\tiny{\begin{matrix}\tau\in\R,n\in\Z^d\\|n|\sim J \end{matrix}}}\sup_{\xi\in (-L_0,L_0)} |\cL_\xi (i\tau,in;\xi)| \lesssim J^\b, \label{e1.7''}
 \end{align}
 with $\b=1$, where,
 $$
 \om_{\LL}(J;\d):=\sup_{\tiny{\begin{matrix} \tau\in\R, n\in\Z^d\\ |n|\sim J\end{matrix}}} |\Om_{\LL}(\tau,n,\d)|,
$$
and $\cL_\xi(i\tau,i n;\xi)=\po_\xi\cL(i\tau,i n; \xi)$. Indeed, we note that if $|n|\sim J$ then
$$
\Om_\cL(\tau,n,\d)\subset \Om_\cL(\frac{\tau}{|n|},\frac{n}{|n|},C\frac{\delta}{J}),
$$
for some $C>0$. Therefore, from \eqref{e1.7} we conclude that \eqref{e1.7'} is satisfied with $\beta=1$. Moreover, as $\cL_\xi(i\tau,i n;\xi)=i a'(\xi)\cdot n$ we see that \eqref{e1.7''} is satisfied trivially with $\beta=1$ as well.

  Concerning the symbol of the kinetic parabolic approximation
  $$
  \cL^\ve(i\tau,in,\xi):= i(\tau +\abf(\xi)\cdot n)+\ve |n|^2,
  $$
  for $J,\d >0$, let
  \begin{align*}
\Om_{\LL^\ve}(\tau,n;\d)&:=\{\xi\in (-L_0,L_0) \,:\, |\LL^\ve(i\tau,in,\xi)|\le \d\},\\
\om_{\LL^\ve}(J;\d)&:=\sup_{\tiny{\begin{matrix} \tau\in\R, n\in\Z^d\\ |n|\sim J\end{matrix}}} |\Om_{\LL^\ve}(\tau,n,\d)|
\end{align*}
and $\cL_\xi^\ve:=\po_\xi\cL^\ve$. As in \cite{GH}, we note that, for some $C>0$,
 $$
 \{\xi\in (-L_0,L_0)\,:\, |\cL^\ve(i\tau,in,\xi)|\le \d\}\subset \{\xi\in (-L_0,L_0)\,:\, |\cL(i\tau,in, \xi)| \le C\d\}
 $$
 which, combined with \eqref{e1.7'}, implies
 $$
 \om_{\cL^\ve} (J,\d)\le \om_{\cL}(J,C\d) \lesssim \left(\frac{\d}{J}\right)^\a\qquad \forall \d>0,\, \forall J\gtrsim 1.
 $$
 Further, $\cL_\xi^\ve(i\tau,in,\xi)= \cL_\xi(i\tau,in,\xi)$, and thus
 $$
  \sup_{\tiny{\begin{matrix}\tau\in\R,n\in\Z^d\\|n|\sim J \end{matrix}}}\sup_{\xi\in (-L_0,L_0)} |\cL_\xi^\ve (i\tau,in;\xi)|
  \le  \sup_{\tiny{\begin{matrix}\tau\in\R,n\in\Z^d\\|n|\sim J \end{matrix}}}\sup_{\xi\in (-L_0,L_0)} |\cL_\xi (i\tau,in;\xi)| \lesssim J.
  $$

 Therefore, $\LL^\ve$ satisfies the nondegeneracy conditions \eqref{e1.7'} and \eqref{e1.7''} uniformly in $\ve$, with $\beta=1$.

  Next, we may choose $\varphi$ in equation  \eqref{e5.3.2} as $\varphi(t,x)=\phi(t)\psi(x)$, where $\psi\in C_c^\infty(\cO)$ and $\phi=\phi^\lambda\in C_c^\infty([0,\infty))$ is a cut-off in time such that $0\leq \phi\leq 1$, $\phi\equiv 1$ on $[0,T-\lambda)$, $\phi\equiv 0$ on $[T,\infty)$ and $|\po_t \phi|\leq \frac{1}{\lambda}$ for some $\lambda\in (0,1)$.

  Note that this localization reduces the problem of the regularity of averages of $\psi\chi^\ve$ to the periodic case treated in \cite{GH}. The only difference from the kinetic equation studied in \cite{GH} (equation (3.3) from that paper) is the appearance of the term $\chi(\abf^\ve(\xi)\cdot\nabla\varphi -\ve\Delta\varphi)-2\ve\vartheta'(\xi){\bf1}_{u^\ve>\xi}\nabla u^\ve\cdot\nabla\varphi$, which may be treated exactly as the term $\chi \po_t\phi$ in their argument.

  Thus, the averaging techniques from the proof of corollary~3.3 in \cite{GH} may be applied to equation \eqref{e5.3.2} and, eventually sending $\lambda$ to zero, we obtain the following result.

  \begin{theorem}\label{T:5.3.1} Suppose \eqref{e1.7} is satisfied. Let $u^\ve$ be the kinetic solution of \eqref{e5.1}--\eqref{e5.3} and $\psi\in C_c^\infty(\cO)$. Then
\begin{equation}\label{e5.3.3}
 \| \psi u^\ve\|_{L^r(\Om\X[0,T]; W^{s,r}(\cO))}\le C_\psi( \| u_0\|_{L^{3}}^{3} +1),
 \end{equation}
uniformly in $\ve>0$,  with $s<\frac{\a^2\b}{6(1+2\a)}$,   $\frac1r>\frac{1-\theta}2+\theta$ and $\theta=\frac{\a}{4+\a}$. In particular,
\begin{equation}\label{e5.3.3.1}
 \| u^\ve\|_{L^r(\Om\X[0,T]; W^{s,r}(\cO_0))}\le C_{\cO_0}( \| u_0\|_{L^{3}}^{3} +1),
 \end{equation}
 for any $\cO_0\subset\subset \cO$, uniformly in $\ve$.
 \end{theorem}

 \subsection{Compactness argument}\label{SS:5.3}

 The general lines of the compactness argument described here are motivated by the compactness argument  put forth  in \cite{Ha}.

 \begin{proposition} \label{P:10.4} For all $\l\in(0,1/2)$,  there exists a constant $C>0$ such that for all $\ve\in(0,1)$
$$
\bbE\|u^\ve\|_{C^\l([0,T];H^{-1}(\cO))}\le C.
$$
\end{proposition}
\begin{proof} Recall that, due to \eqref{Energy}, the set $\{u^\ve\,:\,\ve\in(0,1)\}$ is bounded in
$$
L^2(\Om; L^2(0,T; H^1(\cO))).
$$
By Theorem~\ref{T:4.3'}, we may take $\Abf^\ve=\Abf$, which is Lipschitz. We then have,  in particular, that
$$
\{\div(\Abf(u^\ve))\},\quad \{\ve\Delta u^\ve\}
$$
are bounded in $L^2(\Om, L^2(0,T;H^{-1}(\cO)))$, and consequently
$$
\bbE\big\|u^\ve-\int_0^{\cdot}\Phi^\eta(u^\ve)\,dW\big\|_{C^{1/2}([0,T];H^{-1}(\cO)}\le C.
$$
Moreover, for all $\l\in(0,1/2)$, almost all paths of the above stochastic integral are $\l$-H\"older continuous $L^2(\cO)$-valued functions and
$$
\bbE\big\|\int_0^{\cdot}\Phi^\ve(u^\ve)\,dW\big\|_{C^\l([0,T];L^2(\cO))}\le C.
$$
Indeed, this is a consequence of the Kolmogorov continuity theorem (see, e.g., \cite{DPZ}) since the following uniform estimate is true. Let $a>2$, $s,t\in[0,T]$,
\begin{align*}
\bbE\Big\|\int_s^t\Phi^\eta(u^\ve)\,dW\Big\|^a&\le C\bbE\Big(\int_s^t\big\|\Phi^\eta(u^\ve)\big\|_{L^2(U;L^2(\cO))}^2\,dr\Big)^{a/2}\\
&\le C|t-s|^{a/2-1}\bbE\int_s^t\Big(\sum_{k\ge1}\|g_k^\eta(u^\ve)\|_{L^2(\cO)}^2\Big)^{a/2}\,dr\\
&\le C|t-s|^{a/2}\big(1+\bbE\sup_{0\le t\le T}\|u^\ve(t)\|_{L^2(\cO)}^a\big)\\
&\le C|t-s|^{a/2},
\end{align*}
 where we have made use of Burkholder inequality, \eqref{e1.4}  and Theorem \ref{T:4.3'}.
 \end{proof}

 Observe that, since, in Theorem~\ref{T:5.3.1}, $1< r<2$, from Proposition~\ref{P:10.4} it also follows that  $\bbE\|u^\ve\|_{C^\l([0,T];W^{-1,r}(\cO))}\le C$, for some $C>0$ independent of $\ve$.

 Let us define the path space
$$
\mathcal{X}_u=L^r(0,T;L^r(\cO))\cap C([0,T];W^{-2,r}(\cO)).
$$
Let us denote by $\mu_{u^\ve}$ the law of $u^\ve$ on $\mathcal{X}_u$, $\ve\in(0,1)$.

\begin{proposition}\label{P:10.5} The set $\{\mu_{u^\ve}\,:\, \ve\in(0,1)\}$ is tight and, therefore, relatively weakly compact in $\mathcal{X}_u$.
\end{proposition}

\begin{proof}
Let $\cO_1\subset\subset \cO_2  \subset\subset \cO_3 \subset\subset \ldots \subset \subset \cO$ be a sequence of nonempty smooth open sets such that $\cup_{n=1}^\infty \cO_n=\cO$. Given $R>0$, let us consider the set
\begin{multline*}
 K_R=\{u\in L^\infty((0,T)\X\cO)\cap  L^r(0, T;W_{\text{loc}}^{s,r}(\cO))\cap C^\l([0,T];W^{-1,r}(\cO))\,:\,\\
\|u\|_{L^\infty((0,T)\X\cO)} \leq R,\quad  \|u\|_{C^\l([0,T];W^{-1,r}(\cO))}\le R, \\ \text{ and }\quad  \|u\|_{L^r(0,T;W^{s,r}(\cO_n))} \leq 2^n (C_{\cO_n} + 1) R, \quad \forall n \geq 1\},
 \end{multline*}
where the constants $C_{\cO_n}$ are given by \eqref{e5.3.3.1}.

We assert that $K_R$ is a relatively compact subset of $\Xal_u$. Indeed, let $\{ \psi_k\}_{k\in\mathbb{N}}$ be a sequence in $K_R$. Taking a countable dense set in $(0,T)$ consisting of Lebesgue points of the elements of the sequence, as Banach space valued functions, from the boundedness in $L^\infty(\cO)$, we may extract a subsequence (not relabeled) which strongly converges in $W^{-1,r}(\cO)$ at each point of the dense set. Hence, by the boundedness in $C^\l(0,T;W^{-1,r}(\cO))$, the subsequence strongly  converges in
$W^{-1,r}(\cO)$  at all points of $[0,T]$, and so, by dominated convergence, it strongly converges in $L^r(0,T; W^{-1,r}(\cO))$. On the other hand, if $\cO_0$ is any smooth open subset of $\cO$ with $\bar\cO_0\subset \cO$, by interpolation we have (see, e.g., \cite{BL})
$$
\|\varphi\|_{L^r(0,T;W^{2,r}(\cO_0))}\le \|\varphi\|_{L^r(0,T;W^{1,r}(\cO_0))}^{s/(1+s)}\|\varphi\|_{L^r(0,T;W^{2+s,r}(\cO_0))}^{1/(1+s)}.
$$
Then, taking $\varphi=(-\Delta)^{-1}\psi_k$, where by $-\Delta$ we mean the minus Laplacian operator with 0 Dirichlet condition on $\po\cO_0$, we conclude that it strongly converges in $L^r(0,T;L^r(\cO_0))$, using that $(-\Delta)^{-1}$ isomorphically takes  $L^r(0,T;L^r(\cO_0))$ onto
$L^r(0,T; W^{2,r}\cap W_0^{1,r}(\cO_0))$.

Applying the above argument repeatedly for $\cO_1$, $\cO_2$, etc.\ in the place of $\cO_0$, by a diagonal argument we find a subsequence that converges in $L^r(0,T;L_{loc}^r(\cO))$. And since this sequence is uniformly bounded in $L^\infty((0,T)\times\cO)$, it converges in $L^r(0,T;L^r(\cO))$, by the dominated convergence theorem.

Finally, using the embedding
\begin{align*}
C^\l([0,T];W^{-1,r}(\cO))&\overset{c}{\hookrightarrow} C([0,T];W^{-2,r}(\cO)),
\end{align*}
we conclude that (the subsequence of) $\{ \psi_k\}_{\k\in\mathbb{N}}$ is convergent in $\Xal_u$ by possibly passing to a further subsequence, thus proving that $K_R$ is relatively compact in $\Xal_u$.

Now, concerning  the tightness of  $\{\mu_{u^\ve}\,:\,\ve\in(0,1)\}$, we see that
 \begin{multline*}
 \mu_{u^\ve}(K_R^C)\le \bbP\Big(\|u^\ve\|_{L^\infty((0,T)\X\cO)}>R\Big)+\bbP\Big(\|u^\ve\|_{C^\lambda([0,T];W^{-1,r}(\cO))} > R\Big)\\
 +\sum_{n=1}^\infty \bbP\Big(\|u^\ve\|_{L^r(0,T;W^{s,r}(\cO_n)}> 2^n (C_{\cO_n} + 1)R \Big) = (I) + (II) + \sum_{n=1}^\infty (III)_n.
 \end{multline*}
 Of course, because of Theorem \ref{T:4.3'},
 $$(I) \leq \frac{1}{R}\bbE\|u^\ve\|_{L^\infty((0,T)\X\cO)} \leq \frac{C}{R}$$
 (indeed, $(I) = 0$ if $R > |a|, |b|$), and, because of Proposition \ref{P:10.4},
 $$(II) \leq \frac{1}{R}\bbE\|u^\ve\|_{C^\l([0,T];W^{-1,r}(\cO)} \leq \frac{C}{R}.$$
 Analogously, but now applying estimate \eqref{e5.3.3.1} in Theorem \ref{T:5.3.1},
 $$(III)_n \leq \frac{1}{2^n(C_{\cO_n} + 1) R}\bbE\|u^\ve\|_{L^r(0,T;W^{s,r}(\cO_n))} \leq \frac{C}{2^n R}.$$

 Everything considered,
  $$\mu_{u^\ve}(K_R^C) \leq C/R,$$
  and since $R>0$ is arbitrary, we conclude that  $\{ \mu^\ve \}$ is tight.

\end{proof}

Passing to a weakly convergent subsequence $\mu^n=\mu_{u^{\ve_n}}$, and denoting the limit law by $\mu$, we now apply the Skorokhod representation theorem (see, e.g., \cite{B2}) to infer the following result.

\begin{proposition}\label{P:10.6} There exists a probability space $(\bar \Om,\bar \F,\bar \bbP)$ with a sequence of $\Xal_u$-valued random variables $\bar u^n$, $n\in\N$, and
$\bar u$ such that:
\begin{enumerate}
\item[(i)] the laws of $\bar u^n$ and $\bar u$ under $\bar \bbP$ coincide with $\mu^n$ and $\mu$, respectively,
\item[(ii)] $\bar u^n$ converges $\bar \bbP$-almost surely to $\bar u$ in the topology of $\Xal_u$.
\end{enumerate}
\end{proposition}

\bigskip
Now, let us define $\tilde \Om= \bar \Om\X \Om$, and $\tilde \bbP=\bar \bbP\X \bbP$, the product measure. Also, let $\tilde \F$ be the $\s$-algebra generated by $\bar \F\X \F$. We also extend
$\bar u^n$, $\bar u$ and $W:\Om\to C([0,T]; {\frak U}_0)$ to $\tilde \Om$ by simply setting $\tilde u^n(\bar\om,\om):=\bar u^n(\bar\om)$, $\tilde u(\bar\om,\om):=\bar u(\bar\om)$, $\tilde W(\bar\om,\om):=W(\om)$.

We have $(\tilde u,\tilde W)\in \Xal_u \X C([0,T]; {\frak U}_0)\subset C([0,T],W^{-2,r}(\cO))\X C([0,T], {\frak U}_0)$, with continuous inclusion. On the other hand, given any Banach space $E$,
and $t\in[0,T]$, the operator  $\rho_t: C([0,T];E)\to E$, with $\rho_t k=k(t)$, is continuous. So, let $(\tilde\F_t)$ be the $\tilde \bbP$-augmented canonical filtration of the process $(\rho_t\tilde u, \rho_t\tilde W)$, $t\in[0,T]$, that is
$$
\tilde\F_t=\s\left(\s(\{\rho_s\tilde u,\rho_s\tilde W\,:\, s\in[0,t]\})\cup\{N\in\tilde\F\,:\,\tilde \bbP(N)=0\}\right).
$$
We observe that, by the maximum principle,  $\tilde u^n$ is uniformly bounded in $L^\infty(\tilde \Om\X [0,T]\X\cO)$, and so $\tilde u^n \wto \tilde u$ in the weak star topology of $L^\infty(\tilde\Om\X[0,T]\X\cO)$. In particular, $\tilde u$ is predictable.

We notice that the  process $\tilde W$ is a  $(\tilde\F_t)$-cylindrical Wiener process, that is, $\tilde W=\sum_{k\ge1}\tilde \b_k e_k$, where $\{\b_k\}_{k\ge1}$, with $\tilde \b_k:\tilde\Om\X[0,T]\to\R$,  is the collection of mutually independent real-valued $(\tilde\F_t)$-processes given by $\tilde \b_k(\bar \om,\om,t)=\b_k(\om,t)$. In particular, a.s.\ in $\tilde\Om$, $\tilde{W}\in C([0,T],{\frak U}_0)$.

\subsubsection{Identification of the limit}\label{SS:5.4}

We say that $\big((\tilde \Om,\tilde\F,(\tilde\F_t),\tilde \bbP),\tilde W,\tilde u\big)$ is a  martingale weak entropy solution to \eqref{e1.1}--\eqref{e1.3}  if  $\tilde u$ satisfies Definition~\ref{D:2.3}, with  $(\tilde \Om, \tilde \bbP)$ instead of $(\Om,\bbP)$,  and $\tilde W$ instead of $W$.

\begin{proposition}\label{P:10.7} $\big((\tilde \Om,\tilde\F,(\tilde\F_t),\tilde \bbP),\tilde W,\tilde u\big)$ is a martingale weak entropy solution to \eqref{e1.1}--\eqref{e1.3}.
\end{proposition}
\begin{proof}
Given a convex $\eta\in C^2(\R)$ and   a test function $0\le \psi \in C_c^\infty( \cO)$ let us define for all $s\le t\in[0,T]$
\begin{align*}
M_\eta^n(t)&=\la \eta(u^n(t)),\psi\ra-\la \eta(u_0^n),\psi\ra -\int_0^t\la q(u^n(s)),\nabla \psi\ra\,ds \\
& +\ve_n\int_0^t\la \nabla \eta(u^n(s),\nabla \psi\ra\,ds+\ve_n\int_0^t\la \eta''(u^{\ve_n}(s))|\nabla u^{\ve_n}(s)|^2,\psi\ra\,ds \\
&-\frac12\int_0^t \la {G^{\ve_n}}^2(\cdot,u^{\ve_n}(s))\eta''(u^{\ve_n}),\psi\ra\,ds,\qquad n\in\N,
\end{align*}
\begin{align*}
\tilde M_\eta^n(t)&=\la \eta(\tilde u^n(t)),\psi\ra-\la \eta(\tilde u_0^n),\psi\ra -\int_0^t\la q(\tilde u^n(s)),\nabla \psi\ra\,ds \\
& +\ve_n\int_0^t\la \nabla \eta(\tilde u^n(s)),\nabla \psi\ra\,ds+\ve_n\int_0^t\la \eta''(\tilde u^{\ve_n}(s))|\nabla \tilde u^{\ve_n}(s)|^2,\psi\ra\,ds\\
&-\frac12\int_0^t \la {G^{\ve_n}}^2(\cdot,\tilde u^{\ve_n}(s))\eta''(\tilde u^{\ve_n}),\psi\ra\,ds,\qquad n\in\N,\\
\end{align*}
\begin{align*}
\tilde M_\eta(t)=\la \eta(\tilde u(t)),\psi\ra-\la &\eta(\tilde u_0),\psi\ra -\int_0^t\la q(\tilde u(s)),\nabla \psi\ra\,ds \\
& +\la \tilde\mu_\eta, \chi_{[0,t)}\psi\ra  -\frac12\int_0^t \la {G}^2(\cdot,\tilde u(s))\eta''(\tilde u),\psi\ra\,ds,
\end{align*}
where $\tilde \mu_\eta$ is the limit, by passing to a subsequence if necessary, of
$$
\ve\eta''(\tilde u^{\ve_n})|\nabla \tilde u^{\ve_n}|^2
$$
in $L_w^2(\tilde\Om; \M_b([0,T]\X\cO))$, the space of the weak star measurable mappings $m:\tilde \Om\to\M_b([0,T]\X\cO)$, such that  $\bbE\|m\|_{\M_b}^2<\infty$, where $\M_b([0,T]\X\cO)$ is the space of bounded Radon measures over $[0,T]\X\cO$. Indeed, we have
\begin{align*}
&\tilde\bbE\left|\ve_n\int_0^T\int_{\cO} \eta''(\tilde u^{\ve_n})|\nabla \tilde u^{\ve_n}|^2\,dt\,dx\right|^2=
\bbE\left|\ve_n\int_0^T\int_{\cO} \eta''( u^{\ve_n})|\nabla  u^{\ve_n}|^2\,dt\,dx\right|^2\le C,
\end{align*}
which follows from the entropy identity \eqref{eL5.4} with $\psi\equiv1$, $t=T$, $s=0$,  by taking the square, then the expectation, using It\^o isometry and making trivial estimates.

Let $\cD\subset [0,T]$ be a subset of full measure such that $\tilde u^n(t)\to \tilde u(t)$ in $L^r(\tilde \Om\X\cO)$ and $\tilde \bbP\left(\tilde \mu_\eta(\{\tau=t\}\X\cO)\right)=0$, for $t\in\cD$. We claim that the processes
\begin{equation}\label{emartin1}
\tilde M_\eta, \qquad \tilde M_\eta^2-\sum_{k\ge1} \int_0^{\cdot}\la g_k(\tilde u)\eta'(\tilde u),\varphi\ra^2\, dr,\qquad   \tilde M_\eta \tilde \b_k-\int_0^{\cdot}\la g_k(\tilde u)\eta'(\tilde u),\varphi\ra\,dr,
\end{equation}
 are $(\tilde \F_t)$-martingales indexed by $t\in\cD$.

 Indeed, for all $n\in\N$, the process
 $$
 M_\eta^n=\int_0^{\cdot}\la \eta'(u^n) \Phi^n(u^n)\,dW(s),\psi\ra=\sum_{k\ge1}\int_0^{\cdot}\la\eta'(u^n) g_k^n(u^n),\psi\ra\,d\b_k(s)
 $$
 is a square integrable $(\F_t)$-martingale by \eqref{e1.4} and \eqref{eL5.4}. Denoting by $\lQ\cdot,\cdot\rQ$ the quadratic variation, by the Doob-Meyer decomposition (see, e.g., \cite{KS}) we then have that
 $$
 (M_\eta^n)^2-\sum_{k\ge1}\int_0^{\cdot}\la \eta'(u^n)g_k^n(u^n),\psi\ra^2\,ds,\qquad M_\eta^n\b_k-\int_0^{\cdot}\la\eta'(u^n)g_k^n(u^n),\psi\ra\,ds
 $$
 are $(\F_t)$-martingales, since
 $$
 \lQ M_\eta^n\rQ=\sum_{k\ge1} \int_0^{\cdot} \la \eta'(u^n) g_k^n(u^n),\psi\ra^2\,ds,\qquad \lQ M_\eta^n,\b_k\rQ=\int_0^{\cdot}\la\eta'(u^n)g_k^n(u^n),\psi\ra\,ds.
 $$

 Let $\gamma: W^{-2,r}(\cO)\X {\frak U}_0\to [0,1]$ be arbitrarily given.  From what we have just seen and the equality of laws, we have
 \begin{equation}\label{e10.4}
\begin{aligned}
&\tilde \bbE\g(\rho_s\tilde u^n,\rho_s\tilde W)[\tilde M_\eta^n(t)-\tilde M_\eta^n(s)]\\
&\qquad=\bbE\g(\rho_su^n,\rho_sW)[M_\eta^n(t)-M_\eta^n(s)]=0,
\end{aligned}
\end{equation}
\begin{equation}\label{e10.5}
\begin{aligned}
&\tilde \bbE\g(\rho_s\tilde u^n,\rho_s\tilde W)\Big[(\tilde M_\eta^n)^2(t)-(\tilde M_\eta^n)^2(s)-\sum_{k\ge1}\int_s^t\la \eta'(u^n)g_k^n(u^n),\varphi\ra^2\,dr\Big]\\
&\qquad= \bbE\g(\rho_s u^n,\rho_s W)\Big[(M_\eta^n)^2(t)-(M_\eta^n)^2(s)-\sum_{k\ge1}\int_s^t\la \eta'(u^n)g_k^n(u^n),\varphi\ra^2\,dr\Big]\\
&\qquad=0
\end{aligned}
\end{equation}
\begin{equation}\label{e10.6}
\begin{aligned}
&\tilde \bbE\g(\rho_s\tilde u^n,\rho_s\tilde W)\Big[\tilde M_\eta^n(t)\tilde\b_k^n(t)-\tilde M_\eta^n(s)\tilde\b_k^n(s)-\int_0^s\la \eta'(u^n)g_k^n(\tilde u^n),\varphi\ra\,dr\Big]\\
&\qquad=\bbE\g(\rho_s u^n,\rho_s W)\Big[M_\eta^n(t)\b_k(t)- M_\eta^n(s)\b_k(s)-\int_s^t\la \eta'(u^n) g_k(u^n),\varphi\ra\,dr\Big]\\
&\qquad=0.
\end{aligned}
\end{equation}
Then, for $s,t\in \cD$ the expectations in \eqref{e10.4}--\eqref{e10.6} converge by the Vitali convergence theorem, since all terms are uniformly integrable and converge $\tilde\bbP$-a.s.\ by Proposition~\ref{P:10.6}. Hence
\begin{equation}\label{emartin2}
\begin{aligned}
&\tilde\bbE\g(\rho_s\tilde u,\rho_s\tilde W)\left[\tilde M_\eta(t) -\tilde M_\eta(s)\right]=0,\\
&\tilde \bbE\g(\rho_s\tilde u,\rho_s \tilde W) \left[\tilde M_\eta^2(t)-\tilde M_\eta^2(s)-\sum_{k\ge 1}\int_s^t\la \eta'(\tilde u)g_k(\tilde u),\psi\ra^2\,dr \right]=0,\\
&\tilde \bbE\g(\rho_s\tilde u,\rho_s\tilde W)\left[\tilde M_\eta\tilde\b_k(t)-\tilde M(s)\tilde\b_k(s)-\int_s^t\la \eta'(\tilde u)g_k(\tilde u),\psi\ra\,dr\right]=0,
\end{aligned}
\end{equation}
which gives the $(\tilde \F_t)$-martingale property for $t,s\in\cD$.

If all the processes in \eqref{emartin1} were continuous-time martingales then we would have
\begin{equation}\label{emartin3}
\lQ \tilde M_\eta -\int_0^{\cdot}\la \eta'(\tilde u)\Phi(\tilde u)\,d\tilde W, \psi\ra \rQ=0,
\end{equation}
which  implies the equality $\tilde M_\eta=\int_0^{\cdot}\la \eta'(\tilde u) \Phi(\tilde u)\,d\tilde W,\psi\ra $, and so
\begin{equation}\label{emartin4}
\begin{aligned}
&\la \eta(\tilde u(t)),\psi\ra-\la \eta(\tilde u_0),\psi\ra =\int_0^t\la q(\tilde u(s)),\nabla \psi\ra\,ds \\
& -\la \tilde\mu_\eta, \chi_{[0,t)}\psi\ra  +\frac12\int_0^t \la {G}^2(\cdot,\tilde u(s))\eta''(\tilde u),\psi\ra\,ds+ \int_0^t\la \eta'(\tilde u) \Phi(\tilde u)\,d\tilde W,\psi\ra.
\end{aligned}
\end{equation}
In the case we are dealing here, that is, of martingales indexed by $t\in\cD$, we employ proposition~A.1 in \cite{Ha} to conclude, from \eqref{emartin2}, the validity of \eqref{emartin4} for all $0\le \psi\in C_c^\infty(\cO)$, $t\in\cD$, $\tilde\bbP$-a.s. {} In particular, for all $s,t\in\cD$ we have
\begin{equation*}
\begin{aligned}
&\la \eta(\tilde u(t)),\psi\ra-\la \eta(\tilde u(s)),\psi\ra =\int_s^t\la q(\tilde u(r)),\nabla \psi\ra\,dr \\
& -\la \tilde\mu_\eta, \chi_{[s,t)}\psi\ra  +\frac12\int_s^t \la {G}^2(\cdot,\tilde u(r))\eta''(\tilde u),\psi\ra\,dr+ \int_s^t\la \eta'(\tilde u) \Phi(\tilde u)\,d\tilde W,\psi\ra,
\end{aligned}
\end{equation*}
and so, for all $\theta\in C_c^\infty([0,T))$ we get
\begin{equation*}
\begin{aligned}
&\int_0^T\la \eta(\tilde u(t)),\theta_t\psi\ra\,dt +\la\eta(u_0),\theta(0)\psi\ra-\int_0^T\theta(t)\la q(\tilde u(t)),\nabla \psi\ra\,dt \\
&= \la \tilde\mu_\eta, \theta\psi\ra  -\frac12\int_0^T \la {G}^2(\cdot,\tilde u(r))\eta''(\tilde u),\theta\psi\ra\,dr- \int_0^T\la \eta'(\tilde u) \Phi(\tilde u)\,d\tilde W,\theta\psi\ra,
\end{aligned}
\end{equation*}
from which, by density, \eqref{e2.4} follows, with $\tilde\Om,\, \tilde W,\, \tilde \bbP$ instead of $\Om,\, W,\,  \bbP$.

The same argument just used for the martingales $M_\eta^n$, $\tilde M_\eta^n$, $\tilde M_\eta$, can be similarly applied to the martingales
\begin{align}
&N^n(t)= \la u_n(t),\psi\ra -\la u_0^n,\psi\ra -\int_0^t\la \Abf(u_n(s)),\nabla \psi\ra\,ds  \label{eN.1}\\
&\qquad\qquad\qquad+\ve_n\int_0^t\la\nabla u_n(s), \nabla\psi\ra\,ds,\nonumber \\
&\tilde N^n(t)= \la \tilde u_n(t),\psi\ra -\la \tilde u_0^n,\psi\ra -\int_0^t\la \Abf(\tilde u_n(s)),\nabla \psi\ra\,ds \label{eN.2}\\
&\qquad\qquad\qquad+\ve_n\int_0^t\la\nabla \tilde u_n(s),\nabla\psi\ra\,ds, \nonumber \\
&\tilde N(t)= \la \tilde u(t),\psi\ra -\la \tilde u_0,\psi\ra -\int_0^t\la \Abf(\tilde u(s)),\nabla \psi\ra\,ds, \label{eN.3}
\end{align}
where now $\psi\in C^\infty( \R^d)$ and $\la\cdot,\cdot\ra$ keeps denoting the inner product in $L^2(\cO)$, which then leads us to
$$
\tilde N(t)=\int_0^t \la \Phi(\tilde u(s))\,d\tilde W(s), \psi\ra,
$$
for $t\in \cD$. From this, similarly to what was just done for $\tilde M_\eta$,  we arrive at \eqref{e2.3}  This concludes the proof.

 \end{proof}

\begin{proof}[Conclusion of the existence part of Theorem~\ref{T:1.1}]  The conclusion of the proof of the existence part of Theorem~\ref{T:1.1}  follows the same lines in subsection~4.5 of \cite{Ha}, that is,
we apply the Gy\"ongy and Krylov's criterion for convergence in probability \cite{GK}. The latter states that a sequence $u^n$ of random variables assuming values in a complete metric space $X$ converges in probability
if and only if given any pair of subsequences $(u^{n_k}, u^{m_k})$ the  corresponding sequence of joint laws $\{\mu_{n_k,m_k}\}$, by passing to a subsequence if necessary, converges weakly to a probability
measure $\mu$ satisfying $\mu\left((x,y)\in X\X X;\, x=y\right)=1$. The way to use this criterion is to proceed as above but using a pair of subsequences instead of just  a single subsequence. So, one proves the tightness
of the joint laws of a pair of subsequences, then one proves that each of them converge to a martingale weak entropy solution of \eqref{e1.1}--\eqref{e1.3} for the same probability space $(\tilde\Om,\tilde \bbP)$ and the same
Wiener process $\tilde W$. Then we use the equivalence between weak entropy and kinetic solutions and the uniqueness of kinetic solutions to conclude that the joint laws converge to a measure concentrated on
the diagonal of the cartesian product $\Xal_u\X\Xal_u$.  Hence, the sequence converges in probability, which implies the convergence a.e.\ in $\Om\X[0,T]\X\cO$ and the limit is a kinetic solution of \eqref{e1.1}--\eqref{e1.3}.

\end{proof}

 \appendix
\section{ On convolutions of semigroups in Hilbert spaces}\label{S:7}

First of all, let us recall the well known spectral theorem in its multiplicative operator form, whose statement, exactly as given in  \cite{RS}, we recall here.

\begin{proposition} \label{spectralthm}
Let $A$ be a self-adjoint operator on a separable Hilbert space $H$ with domain $D(A)$. Then there is a measure space $(M, \mu)$ with $\mu$ a finite measure, a unitary operator $U: H \rightarrow L^2(M, d\mu)$, and a real-valued function $f$ on $M$ which is finite a.e.\  so that
\begin{enumerate}
  \item  $\psi \in D(A)$ if and only if $ f( \> \cdot \>) (U \psi) (\> \cdot \>) \in L^2(M, d\mu) $;
  \item If $\varphi \in U(D(A))$, then $(U A U^{-1} \varphi)(m) = f(m) \varphi(m)$.
\end{enumerate}
\end{proposition}

In the remainder of this section, we will preserve the notations and assumptions of this spectral theorem. Also we will assume that the operator $A$ is {\em nonnegative}, which allows us to characterize its generated semigroup $S(t) = \exp - t A$ by means of the operational calculus simply as
$$
(U  \exp \{ -t A \} \psi)(m) = \exp\{ - t f(m) \} (U\psi)(m).
$$
Furthermore, it allows us to characterize the spaces
$$
H_A^\alpha := D(A^{\alpha}) = U^{-1} (L^2(M, (1+f(m))^{2\alpha}\, d\mu)) =: U^{-1} (V_f^\alpha),
$$
for $\alpha \geq 0$. For $\alpha < 0$, we set $H_A^\alpha = (H_A^{-\alpha})^*$, which may still naturally be identified with $V_f^\alpha = L^2(M, (1+f(m))^{2\alpha} d\mu)$. Of course, then $(I+A)^{\beta}$ defines a linear isometry between $H_A^\alpha$ and $H_A^{\alpha - \beta}$. Observe that $S(t)$ is still a contraction semi-group on the spaces $H_A^\alpha$.

Let $T > 0$. For any $-\infty < \alpha < \infty$, we may define the Duhamel convolution operator
$$
(\cI h)(t) = \int_0^t S(t-s) h(s) ds
$$
for $h \in L^2(0,T; H_A^\alpha)$. Clearly, if $h\in C([0,T]; H_A^\a)$,
\begin{align}
&\sup_{t\in[0,T]} \Vert \cI h(t)\Vert_{H_A^\a}^2\le T^2\sup_{t\in[0,T]} \|h(t)\|_{H_A^\a}^2, \label{e6.1}
\end{align}

Also clearly, one has that
\begin{equation}
\int_0^T \Vert \cI h(s) \Vert_{H_A^\alpha}^2 ds \leq \frac{T^2}{2} \int_0^T \Vert h(s) \Vert_{H_A^\alpha}^2 ds. \label{regduhamel0}
\end{equation}
However, one may say more regarding the regularization provided by the operator $\cI$.

\begin{proposition} \label{reg1}
  In the notations above, $\cI$ maps $L^2(0,T;H_A^\alpha)$ into $L^2(0,T;H_A^{\alpha+1})$ and
  \begin{equation}
    \int_0^T \Vert \cI h(s) \Vert_{H_A^{\alpha+1}}^2 ds \leq C \int_0^T \Vert h(s) \Vert_{H_A^\alpha}^2 ds, \label{regduhamel}
  \end{equation}
  for some absolute constant $C$ depending only on $T$.
\end{proposition}
\begin{proof}
  By the remarks above,  we see that it suffices to analyze the case $\alpha = 0$. Since
  $$
  \int_0^T \|\cI h(s)\|_{H_A^1}^2ds \leq 2  \Big( \int_0^T \|\cI h(s)\|_{H}^2ds + \int_0^T \|A(\cI h)(s)\|_{H}^2ds \Big),
  $$
  and the first term was already estimated in \eqref{regduhamel0},  we may concentrate ourselves only on estimating the second one.
  Applying the spectral theorem, Proposition~\ref{spectralthm}, its operational calculus and Cauchy-Schwarz   inequality, we get
  \begin{align*}
  \int_0^T\Vert A (\cI h)(t)& \Vert_{H}^2 dt \\
  &= \int_0^T \int_M f(m)^2\Big| \int_0^t \exp\{- (t-s) f(m)\} \big(U h(s)\big)(m)\, ds \Big|^2\, d\mu(m)\,dt \\
  &=  \int_0^T\int_M \Big| \int_0^t \exp\{- (t-s) f(m)\} f(m) \big(U h(s)\big)(m)\,ds \Big|^2 \,d\mu(m)\,dt \\
  &\leq \int_0^T  \int_M \int_0^t  \exp\{- (t-s) f(m)\}  f(m) \big| \big(U h(s)\big)(m) \big|^2\, ds\, d\mu(m)\,dt.
  \end{align*}
  For $h \in L^2(0,T;H)$, $Uh \in L^2(0,T; M)$, and thus Tonelli's theorem yields
  \begin{align*}
 \int_0^T\Vert A (\cI h)(t) &\Vert_{H}^2 dt \\ &\leq \int_M \int_0^T \int_0^t  \exp\{- (t-s) f(m)\}  f(m) \big| \big(U h(s)\big)(m) \big|^2 ds dt\, d\mu(m) \\
  &= \int_M \int_0^T \int_s^T  \exp\{- (t-s) f(m)\}  f(m) \big| \big(U h(s)\big)(m) \big|^2 dt\, ds d\mu(m) \\
  &\leq \int_M \int_0^T \big| \big(U h(s)\big)(m) \big|^2 \,ds\, d\mu(m) \\
  &=\Vert h \Vert_{L^2(0,T; H)}^2.
  \end{align*}
  This shows the validity of the inequality.
\end{proof}

Our second inequality will be for {\em stochastic convolutions}. So, let us first fix some notations and additional hypothesis.

 As in Section~\ref{S:1},  let  $(\Omega, \mathscr{F}, (\mathscr{F})_{t\geq 0}, \mathbb{P})$ be a stochastic basis with a complete and right-continuous filtration. Moreover, let $\mathcal{P}$ be the predictable $\sigma-$algebra on $\Omega \times [0,T]$ associated to $(\mathscr{F}_t)_{t\geq 0}$ and $W$  be a cylindrical Wiener process, i.e.,
$$
W(t) = \sum_{k=1}^\infty \beta_k(t) e_k
$$
where the $\beta_k$'s are mutually independent real-valued standard Wiener processes relative to $(\mathscr{F}_t)_{t\geq 0}$, and $(e_k)$ is an orthonormal basis of another separable Hilbert space $\mathfrak{U}$.

Let $T>0$ and $-\infty  < \alpha < \infty$. Under these conditions, we may introduce the stochastic Duhamel operator
$$
(\cI_W \Psi)(t) = \int_0^t S(t-s) \Psi(s) \, dW(s)
$$
for predictable processes $\Psi \in L^2(\Omega \times [0,T];L_2(\mathfrak{U};H_A^\alpha))$. Concerning properties of $\cI_W$, we have the following result.

\begin{proposition} \label{reg1,5}
  In the notations above, $\cI_W$ maps $L^2(\Omega \times [0,T];L_2(\mathfrak{U};H_A^\alpha))$ into $L^2(\Omega \times [0,T]; H_A^{\alpha+1/2})$ and
\begin{equation}\label{regduhamel1,5}
\Vert \mathcal{I}_W  \Psi \Vert_{L^2(\Omega \times [0,T]; H_A^{\alpha+1/2})} \leq C \Vert \Psi \Vert_{L^2(\Omega;L^2(0,T;L_2(\mathfrak{U};H_A^\alpha))},
\end{equation}
for some $C>0$ depending only on $T$.
\end{proposition}
\begin{proof}
  The verification of \eqref{regduhamel2}  is similar to that of \eqref{regduhamel}, but here we have to use also It\^o isometry. For this reason, the smoothing effect is weaker. Writing $\Psi(\omega, t) e_k = \psi_k(\omega, t)$, by  It\^o isometry and the spectral theorem, Proposition~\ref{spectralthm}, we have
  \begin{align*}
  \mathbb{E} &\int_0^T  \Vert \cI_W \Psi(t) \Vert_{H_A^{\alpha+1/2}}^2 \\ &=  \mathbb{E} \int_0^T \Bigg\Vert \int_0^t S(t-s) \Psi(s)\, dW(s) \Bigg\Vert_{H_A^{\alpha+1/2}}^2\,dt \\
  &= \mathbb{E} \int_0^T \int_0^t \big\Vert S(t-s) \Psi(s) \big\Vert_{L_2(\mathfrak{U}; H_A^{\alpha+1/2})}^2\,ds\,dt \\
  &= \sum_{k=1}^\infty \mathbb{E} \int_0^T\int_0^t  \big\Vert S(t-s) \psi_k(s) \big\Vert_{H_A^{\alpha+1/2}}^2\, ds\, dt \\
  &= \sum_{k=1}^\infty \mathbb{E} \int_0^T  \int_0^t \int_M e^{-2(t-s)f(m)} (1 + f(m))^{2\alpha +1} |(U \psi_k(s))(m)|^2 \,d\mu(m)\, ds\, dt  \\
  &= \sum_{k=1}^\infty \mathbb{E}\int_M \int_0^T  \int_s^T e^{-2f(m)(t-s)} (1+f(m))  \\
  &\>\>\>\>\>\>\>\>\>\>\>\>\>\>\>\>\>\>\>\>\>\>\>\>\>\>\>\>\>\>\>\>\>\>\>\>\>\>\>\quad\quad\quad (1 + f(m))^{2\alpha} |(U \psi_k(s))(m)|^2 \,dt\, ds\, d\mu(m) \\
  &\leq C \sum_{k=1}^\infty \mathbb{E} \int_M\int_0^T (1 + f(m))^{2\alpha} |(U \psi_k(s))(m)|^2 \,ds\,d\mu(m) \\
  &\leq C \, \mathbb{E} \int_0^T \Vert \Psi(s) \Vert_{L_2(\mathfrak{U};H_A^\alpha)}^2 ds,
  \end{align*}
  hence the proposition.
\end{proof}

In this paper, we use this proposition as follows. Let $\Phi: H \to L(\mathfrak{U}, H)$ be continuous. If, for any $k \in \N$, $g_k : H \rightarrow H$ is given by $\Phi(h) e_k = g_k(h)$, assume that each $g_k : H_A^{1/2} \rightarrow H_A^{1/2}$ is continuous and that there exist constants $\gamma_k > 0$, such that
\begin{align*}
  \Vert g_k(h) \Vert_{H} &\leq \gamma_k (1 + \Vert h \Vert_{H} ), \\
  \Vert A^{1/2} g_k(h) \Vert_{H} &\leq \gamma_k (1 + \Vert A^{1/2} h \Vert_{H} ), \text{and } \\
  \sum_{k=1}^\infty \gamma_k^2 &= D < \infty.
\end{align*}

\begin{proposition} \label{reg2}
  Under the hypothesis above, if $u \in L^2(\Omega \times [0,T]; H_A^{1/2})$ is predictable, then $$I_W \Phi(u) \in L^2\big(\Omega;L^2(0,T;H_A^{1})\big),$$ and
  \begin{equation}\label{regduhamel2}
    \Vert \mathcal{I}_W \Phi(u) \Vert_{L^2(\Om \X [0,T]; H_A^{1})} \leq C \, ( 1 +  \Vert u \Vert_{L^2(\Omega \times [0,T]; H_A^{1/2})} ),
  \end{equation}
  where $C$ only depends on $T$ and $D$.
\end{proposition}
\begin{proof}
We just need to verify that $\Psi = \Phi(u)$ is as in the statement of Proposition \ref{reg1,5} with $\alpha = 1/2$. Since $\Vert h \Vert_{H_A^{1/2}}^2 \leq C(\Vert h \Vert_H^2 + \Vert A^{1/2} h \Vert_H^2),$ we have that
\begin{align*}
  \Vert \Phi(u) &\Vert_{L^2(\Omega;L^2(0,T;L_2(\mathfrak{U};H_A^{1/2}))}^2 \\ &\leq C\big(\Vert \Phi(u) \Vert_{L^2(\Omega;L^2(0,T;L_2(\mathfrak{U};H))}^2 + \Vert A^{1/2} \Phi(u) \Vert_{L^2(\Omega;L^2(0,T;L_2(\mathfrak{U};H))}^2 \big) \\
  &= C \sum_{k=1}^\infty \big( \Vert g_k(u) \Vert_{L^2(\Omega;L^2(0,T;H))}^2 + \Vert A^{1/2} g_k(u) \Vert_{L^2(\Omega;L^2(0,T;H))}^2\big) \\
  &\leq C  \sum_{k=1}^\infty \gamma_k^2\,\big(1 + \Vert u \Vert_{L^2(\Omega;L^2(0,T;H))}^2 + \Vert A^{1/2} u \Vert_{L^2(\Omega;L^2(0,T;H))}^2 \big) \\
  &\leq C\, D\, \big( 1 + \Vert u \Vert_{L^2(\Omega;L^2(0,T;H_A^{1/2}))}^2 \big).
\end{align*}
Therefore, as the argument above also shows the predicability of $\Phi(u)$, the desired result is now a direct consequence of Proposition \ref{reg1,5}.
\end{proof}

Now, in order to apply the above abstract theory, let us fix
$$
H =L^2(\mathcal{O}),
$$
and, denoting as usual $H^k(\cO)$ the $k$-th order Sobolev space,
\begin{equation}\label{defA}
\begin{aligned}
D(A) &= \Big\{ u \in H^2(\mathcal{O}) ; \frac{\partial u}{\partial \nu} = 0 \text{ on } \partial \mathcal{O} \text{ (in the sense of traces in $H^1(\cO)$) } \Big\},\\
 Au &= -\Delta u,\quad  \text{for $u \in D(A)$}.
\end{aligned}
\end{equation}

The proof of following result is  somewhat standard and  so we omit it.

\begin{proposition}\label{P:5.1} The operator $A: D(A)\to H$ defined in \eqref{defA} is self-adjoint. Moreover, we have
\begin{equation}
  H_A^{1/2}=D((I+ A)^{1/2})= H^1(\mathcal{O}). \label{Dmeio}
\end{equation}
\end{proposition}

\end{document}